\documentclass[11pt]{amsart}
\usepackage[utf8]{inputenc}
\usepackage{amsthm}
\usepackage{amssymb}
\usepackage{latexsym}
\usepackage{multicol}
\usepackage{verbatim,enumerate}
\usepackage{hyperref}
\usepackage{amsmath, amscd}
\usepackage{mathrsfs}
\usepackage[all,cmtip]{xy}
\usepackage[capitalise]{cleveref}
\usepackage{soul}

\advance\textwidth by 1.2in \advance\oddsidemargin by -.6in \advance\evensidemargin by -.6in
\usepackage{tikz}
\usetikzlibrary{decorations.markings}
\tikzstyle{vertex}=[ circle, fill, draw, inner sep=0pt, minimum size=4pt,]
\tikzstyle{edge}= [thick]

\theoremstyle{definition}
\newtheorem{thm}{Theorem}[subsection]
\newtheorem{prop}[thm]{Proposition}

\newtheorem{con}[thm]{Conjecture}
\newtheorem{lem}[thm]{Lemma}

\newtheorem{ex}[thm]{Example}
\newtheorem{rem}[thm]{Remark}

\numberwithin{equation}{subsection}

\newcommand{\wt}{\operatorname{wt}}

\newcommand{\nc}{\newcommand}
\newcommand{\rnc}{\renewcommand}

\nc{\cal}{\mathcal} \nc{\goth}{\mathfrak} \rnc{\bold}{\mathbf}

\newcommand{\supp}{\operatorname{supp}}
\newcommand{\seg}{\mathbb S}	
\newcommand{\mseg}{\mathbb M}

\renewcommand{\Bbb}{\mathbb}
\nc\bomega{{\mbox{\boldmath $\omega$}}} \nc\bpsi{{\mbox{\boldmath $\Psi$}}}
 \nc\balpha{{\mbox{\boldmath $\alpha$}}}
 \nc\bbeta{{\mbox{\boldmath $\beta$}}}
 \nc\bpi{{\mbox{\boldmath $\pi$}}}
  \nc\bpis{{\mbox{\boldmath \scriptsize$\pi$}}}
 \nc\bullets{{\mbox{\scriptsize $\bullet$}}}
 \nc\bvarpis{{\mbox{\boldmath \scriptsize$\varpi$}}}
  \nc\bvarpi{{\mbox{\boldmath $\varpi$}}}
  \def\bs#1{\boldsymbol{#1}}
  \def\tlie#1{\tilde{\mathfrak{#1}}}

\def\endd{\hfill$\diamond$}
\newcommand{\spar}{\operatorname{spar}}

\nc\bepsilon{{\mbox{\boldmath $\epsilon$}}}
  \nc\bomegas{{\mbox{\boldmath\scriptsize $\omega$}}}
  \nc\bepsilons{{\mbox{\boldmath \scriptsize$\epsilon$}}}

\nc\hlien{\hat{\lie n}^+}
  \nc\bxi{{\mbox{\boldmath $\xi$}}}
\nc\bmu{{\mbox{\boldmath $\mu$}}} \nc\bcN{{\mbox{\boldmath $\cal{N}$}}} \nc\bcm{{\mbox{\boldmath $\cal{M}$}}} \nc\blambda{{\mbox{\boldmath
$\lambda$}}}%\nc\mathbb Nu{{\mbox{\boldmath $\nu$}}}

\newcommand{\lie}[1]{\mathfrak{#1}}

\makeatletter
\def\section{\def\@secnumfont{\mdseries}\@startsection{section}{1}%
  \z@{.7\linespacing\@plus\linespacing}{.5\linespacing}%
  {\normalfont\scshape\centering}}
\def\subsection{\def\@secnumfont{\bfseries}\@startsection{subsection}{2}%
  {\parindent}{.5\linespacing\@plus.7\linespacing}{-.5em}%
  {\normalfont\bfseries}}
\makeatother

 \nc{\Hom}{\operatorname{Hom}}
  \nc{\mode}{\operatorname{mod}}
\nc{\End}{\operatorname{End}} \nc{\wh}[1]{\widehat{#1}} \nc{\Ext}{\operatorname{Ext}}
 \nc{\ch}{\operatorname{ch}} \nc{\ev}{\operatorname{ev}}
\nc{\Ob}{\operatorname{Ob}} \nc{\soc}{\operatorname{soc}} \nc{\rad}{\operatorname{rad}} \nc{\head}{\operatorname{head}}
\def\Im{\operatorname{Im}}

 \nc{\Cal}{\cal} \nc{\Xp}[1]{X^+(#1)} \nc{\Xm}[1]{X^-(#1)}
\nc{\on}{\operatorname} \nc{\Z}{{\bold Z}} \nc{\J}{{\cal J}} \nc{\C}{{\bold C}} \nc{\Q}{{\bold Q}}

\nc{\N}{{\Bbb N}} \nc\boa{\bold a} \nc\bob{\bold b} \nc\boc{\bold c} \nc\bod{\bold d} \nc\boe{\bold e} \nc\bof{\bold f} \nc\bog{\bold g}
\nc\boh{\bold h} \nc\boi{\bold i} \nc\boj{\bold j} \nc\bok{\bold k} \nc\bol{\bold l} \nc\bom{\bold m} \nc\bon{\bold n} \nc\boo{\bold o}
\nc\bop{\bold p} \nc\boq{\bold q} \nc\bor{\bold r} \nc\bos{\bold s} \nc\boT{\bold t} \nc\boF{\bold F} \nc\bou{\bold u} \nc\bov{\bold v}
\nc\bow{\bold w} \nc\boz{\bold z} \nc\boy{\bold y} \nc\ba{\bold A} \nc\bb{\bold B} \nc\bc{\mathbb C} \nc\bd{\bold D} \nc\be{\bold E} \nc\bg{\bold
G} \nc\bh{\bold H} \nc\bi{\bold I} \nc\bj{\bold J} \nc\bk{\bold K} \nc\bl{\bold L} \nc\bm{\bold M}  \nc\bo{\bold O} \nc\bp{\bold
P} \nc\bq{\bold Q} \nc\br{\bold R}  \nc\bt{\bold T} \nc\bu{\bold U} \nc\bv{\bold V} \nc\bw{\bold W} \nc\bx{\bold x} \nc\KR{\bold{KR}} \nc\rk{\bold{rk}} 
\nc\het{\text{ht }}
\nc\bz{\mathbb Z}
\nc\bn{\mathbb N}

\nc\toa{\tilde a} \nc\tob{\tilde b} \nc\toc{\tilde c} \nc\tod{\tilde d} \nc\toe{\tilde e} \nc\tof{\tilde f} \nc\tog{\tilde g} \nc\toh{\tilde h}
\nc\toi{\tilde i} \nc\toj{\tilde j} \nc\tok{\tilde k} \nc\tol{\tilde l} \nc\tom{\tilde m} \nc\ton{\tilde n} \nc\too{\tilde o} \nc\toq{\tilde q}
\nc\tor{\tilde r} \nc\tos{\tilde s} \nc\toT{\tilde t} \nc\tou{\tilde u} \nc\tov{\tilde v} \nc\tow{\tilde w} \nc\toz{\tilde z} \nc\woi{w_{\omega_i}}
\nc\chara{\operatorname{Char}}

\begin{document}
	\title[]{Imaginary modules arising from\\ tensor products of snake modules }
	\author{Matheus Brito}
	\address{Departamento de Matematica, UFPR, Curitiba - PR - Brazil, 81530-015}
	\email{mbrito@ufpr.br}
	\thanks{The work of M.B. and A.M. was partially supported CNPq grant 405793/2023-5. The work of A.M. was also partially supported by Fapesp grant 2018/23690-6. }
	\author{Adriano Moura}
	\address{Departamento de Matemática, Universidade Estadual de Campinas, Campinas - SP - Brazil, 13083-859.}
	\email{aamoura@unicamp.br}
	\thanks{}

	\begin{abstract}
	Motivated by the limitations of cluster algebra techniques in detecting imaginary modules, we build on the representation-theoretic framework developed by the first author and Chari to extend the construction of such modules beyond previously known cases, which arise from the tensor product of a higher-order Kirillov--Reshetikhin module and its dual. Our first main result gives an explicit description of the socle of tensor products of two snake modules, assuming the corresponding snakes form a covering pair of ladders. By considering a higher-order generalization of the covering relation, we describe a sequence of inclusions of highest-$\ell$-weight submodules of such tensor products. We conjecture all the quotients of subsequent modules in this chain of inclusions are simple and imaginary, except for the socle itself, which might be real. We prove the first such quotient is indeed simple and, assuming an extra mild condition, we also prove it is imaginary, thus giving  rise to new classes of imaginary modules within the category of finite-dimensional representations of quantum loop algebras in type $A$.
    \end{abstract}

\maketitle
\section{Introduction} 

This paper is dedicated to the study of certain aspects of the tensor product of two prime snake modules for a quantum loop algebra $U_q(\tlie g)$ in the case that the finite-dimensional simple Lie algebra $\lie g$ is of type $A$. We assume $q$ is neither zero nor a root of unity. Since the discovery by Hernandez-Leclerc \cite{HL10, HL13a} that the Grothendieck rings of certain proper subcategories of the abelian category of finite-dimensional $U_q(\tlie g)$-modules are cluster algebras, the task o classifying the prime and the real simple objects of the category became one of the most studied topics in the area. Since a module is real if its tensor square is simple, the task of classifying the real modules is equivalent to that of classifying the simple modules which are not real. Such modules were termed imaginary in \cite{Lec02}, where the first example of an imaginary simple module was given.  The results of \cite{BC19a, KKOP24, Nak04,Qin17} imply that, in principle, the real modules can be determined using cluster algebra machinery. In practice, however, determining if a given Drinfeld polynomial arises from a cluster monomial is not a simple task, which motivates the search for alternative techniques. For instance, this is the spirit of \cite{BMS24}. Furthermore, the combinatorial framework of cluster algebras does not provide representation theoretic insights for the existence of imaginary modules. Seeking for such insights is the spirit of the present paper.
	
The class of snake modules introduced in \cite{MY12} has been proven to be particularly important in many aspects. In particular, their $q$-characters can be computed by a relatively simple combinatorial data, which we refer to as Mukhin-Young (MY) paths, and the connection with cluster monomials can be made explicitly, thus showing snake modules are real. The class of Kirillov-Reshetikhin (KR) modules and, more generally, that of minimal affinizations, are proper subclasses of that of (prime) snake modules. 
In \cite{BC23}, the authors considered the subclass of prime snake modules whose highest weight is a multiple of a fundamental weight, thus terming such modules ``higher order KR modules''. Suppose $V$ is such a module. By studying $V\otimes V^*$, a new family of imaginary modules was described in \cite{BC23}, assuming $N={\rm rank}(\lie g)>2$. The pioneering example of imaginary modules from \cite{Lec02} is recovered by this construction. The main results of the present paper substantially extends this construction of imaginary modules by studying much more general tensor products of two prime snake modules. The assumption $N>2$ is implied by the assumptions of the construction. We remark that \cite{BC23} also gives a different construction specific to the case $N=2$ (and it is well known that all simple modules are real if $N=1$). 

Before continuing, let us recall that the isomorphism classes of simple modules of the category of (type 1) finite-dimensional $U_q(\tlie g)$-modules are classified in terms of Drinfeld polynomials. We denote by $\mathcal P^+$ the set of all Drinfeld polynomials, which is a multiplicative monoid in a natural way. For $\bs\pi\in\mathcal P^+$, we let $V(\bs\pi)$ be any element in the corresponding isomorphism class. The trivial Drinfeld polynomial will be denoted by $\bs 1$.

As a ``preliminary'' step towards our construction of imaginary modules, the first of our main results, \cref{t:soc}, describes explicitly the socle of the tensor products of two prime snake modules, assuming the corresponding snakes form what we termed a covering pair of ladders. Let us explain what this means. Without loss of generality, we restrain ourselves to the subcategory consisting of modules which have simple factors whose Drinfeld polynomials have roots in $q^{\mathbb Z}$. Let $\tilde\seg$ be a parameterizing set for the isomorphism classes of the corresponding fundamental modules and let $\seg=\tilde\seg\cup\{\emptyset\}$. In the main part of the text, we chose to work with $\tilde\seg$ being the set of certain ``segments''. For $s\in\tilde\seg$, let $\bs\omega_s$ denote the corresponding Drinfeld polynomial, while $\bs\omega_\emptyset=\bs 1$. Given $s,s'\in\tilde\seg$, we say $s$ covers $s'$, and write $s\rhd s'$, if 
\begin{equation}\tag{1}\label{e:covers}
	V(\bs\omega_s)\otimes V(\bs\omega_{s'}) \ \ \text{is reducible and highest-$\ell$-weight.}
\end{equation}  
Recall that a module $V$ is highest-$\ell$-weight if its set of weights contains a unique maximal element, the corresponding weight space is one-dimensional, and $V$  is cyclically generated by this weight space. In the language of segments, \eqref{e:covers} can be phrased in a completely intrinsic manner, with no mention to representation theory (see \cref{ss:cover}). It is well-known that, if \eqref{e:covers} is satisfied, then $V(\bs\omega_s)\otimes V(\bs\omega_{s'})$ has length $2$ and its socle is a simple module of the form 
\begin{equation*}
	V(\bs\omega_{s\cap s'})\otimes V(\bs\omega_{s\cup s'}) \cong V(\bs\omega_{s\diamond s'}) \ \ \text{where}\ \ 
	\bs\omega_{s\diamond s'} = \bs\omega_{s\cap s'}\,\bs\omega_{s\cup s'}
\end{equation*}
and $s\cap s',s\cup s'\in\seg$. In the language of segments, $s\cap s'$ is indeed the intersection of the corresponding segments and similarly for $s\cup s'$, thus the choice of notation. A prime snake of length $l$ corresponds to a sequence $\bs s = s_1,s_2,\dots, s_l$ of elements in $\tilde\seg$ such that
\begin{equation*}
	s_{k+1}\rhd s_k \ \ \text{for all}\ \ 1\le k<l. 
\end{equation*}
Such sequences are also called ladders. In that case, $V(\bs\omega_{s_l})\otimes\cdots\otimes V(\bs\omega_{s_1})$ is highest-$\ell$-weight and the associated prime snake module is the simple quotient of this tensor product. Setting
\begin{equation*}
	\bs\omega_{\bs s} = \prod_{k=1}^l\bs\omega_{s_k},
\end{equation*}
it follows that this simple quotient is isomorphic to $V(\bs\omega_{\bs s})$. If $\bs s,\bs s'$ are ladders of length $l$, we say $\bs s$ covers $\bs s'$ and write $\bs s\rhd\bs s'$ if 
\begin{equation*}
	s_k\rhd s'_k \ \ \text{for all}\ \ 1\le k\le l.
\end{equation*}
\cref{t:soc} says that
\begin{equation*}
	\bs s \rhd \bs s' \quad\Rightarrow\quad {\rm soc}(V(\bs\omega_{\bs s})\otimes V(\bs\omega_{\bs s'})) \cong V(\bs\omega_{\bs s\diamond\bs s'})\quad
	\text{where}\quad \bs\omega_{\bs s\diamond\bs s'} = \prod_{k=1}^l \bs\omega_{s_k\diamond s'_k}.
\end{equation*}
A closely related result was obtained in \cite[Corolary 6.16]{LM16} in the framework of complex smooth representations of $GL_n(\mathbb F)$ with $\mathbb F$ being an nonarchimedian field. Using a Schur-Weyl type duality \cite{CPHecke} (we refer the reader to \cite{Gur21a} for a deeper discussion), this result gives rise to an algorithm for computing the socle of a tensor product of a snake module with a fundamental module. In the setting of representations of $GL_n(\mathbb F)$, Lapid and M\'inguez \cite{LM18} proved a necessary and sufficient condition for an irreducible module $V(\bs\omega_{\bs s})$ to be imaginary, provided $\bs s$ is any sequence (not necessarily a ladder) satisfying a certain regularity condition.  The results of \cite{BC23} and, more generally,  \cref{t:imagfromsnakegen}, describe families of imaginary modules which do not satisfy this regularity condition.

For the next step of our construction, we introduce ``higher order'' versions of the covering relation $\bs s\rhd\bs s'$. Given a sequence $\bs s = s_1,\dots,s_l$ of elements of $\seg$ and $0\le m<n<l$, set
\begin{equation*}
	\bs s(m,n) = s_{m+1},\dots, s_{n}.
\end{equation*} 
We say $\bs s$ is a $p$-cover of $\bs s'$ and write $\bs s\rhd_p\bs s'$ if
\begin{equation*}
	\bs s(0,l-k)\rhd \bs s'(k,l) \quad\text{for all}\quad 0\leq k\leq p.
\end{equation*}
In particular, we can consider
\begin{equation*}
	\bs\varpi_k = \bs\omega_{\bs s(0,l-k)\diamond \bs s'(k,l)} \ \ \text{and}\ \ \bs\pi_k = \bs\omega_{\bs s'(0,k)}\, \bs\varpi_k\,\bs\omega_{\bs s(l-k,l)}
\end{equation*}
for all $0\le k\le p$. Our main conjecture is:
\begin{equation*}
	\bs s\rhd_p\bs s' \quad\Rightarrow\quad V(\bs\pi_k) \ \ \text{is imaginary for all}\ \ 1\le k\le p. 
\end{equation*}
Our main result, \cref{t:imagfromsnakegen}, confirms this holds for $k=1$ with a mild extra assumption. We do not believe this assumption is necessary. For instance, if either $l=2$ or both $V(\bs\omega_{\bs s})$ and $V(\bs\omega_{\bs s'})$ are KR modules, we prove the conjecture for $k=1$ even in the cases that the extra condition fails (and such failing cases exist).  We also remark this assumption is trivially satisfied in the context of \cite{BC23}, so \cref{t:imagfromsnakegen} is a proper generalization of the corresponding result from  \cite{BC23}.  Regarding $V(\bs\pi_0)\cong {\rm soc}(V(\bs\omega_{\bs s})\otimes V(\bs\omega_{\bs s'}))$, we obtained examples for which it is real as well as for which it is imaginary. These examples, alongside a more systematic study of the reality of $V(\bs\pi_0)$, will appear in a future publication. It is also interesting to mention that, in \cref{ex:pi1issoc}, we show that, if $\bs s\rhd_1\bs s'$ so that $V(\bs\pi_1)$ is defined, then $V(\bs\pi_1)$ can be obtained as the socle of a tensor product of snake modules in the context of \cref{t:soc}. We also leave a similar discussion for $V(\bs\pi_k), k>1$, for a future publication.

We also prove an intermediate result towards the proof of the conjecture for all $k$ which is relevant on its own right (\cref{p:keystepgen}). Namely, we prove there exists a chain of submodules
\begin{equation*}
	{\rm soc}(V(\bs\omega_{\bs s})\otimes V(\bs\omega_{\bs s'})) = M_0 \subseteq M_1\subseteq \cdots M_p \subseteq V(\bs\omega_{\bs s})\otimes V(\bs\omega_{\bs s'})
\end{equation*} 
where $M_k$ is a highest-$\ell$-weight module whose Drinfeld polynomial is $\bs\pi_k$. In fact, we prove these are the only highest-$\ell$-weight submodules whose highest $\ell$-weights are smaller than $\bs\pi_p$ with respect to the partial order on $\mathcal P^+$ induced by the simple $\ell$-roots. We conjecture
\begin{equation*}
	M_k/M_{k-1} \cong V(\bs\pi_k) \ \ \text{for all}\ \ 1\le k\le p
\end{equation*}
and prove this is true for $k=1$ in complete generality. In \cite{Gur21b}, Gurevich gives an algorithm for computing the simple factors of a tensor product of any two ladder representations of $GL_n(\mathbb F)$  where $\mathbb F$ is a nonarchimedean field. In principle, the result does not explicitly describe a composition series, only the simple factors.
Via the aforementioned Schur–Weyl type duality,  this can be used to compute the simple factors of the tensor product of any two snake modules for $U_q(\tilde{\lie g})$ if $N\gg 0$. On the other hand, \cref{p:keystepgen}, together with the above two conjectures, captures part of the simple factors in a stronger manner: it gives part of a composition series for such tensor products. Moreover, our proof works for any $N>2$.

The paper is organized as follows. We set up the basic notation and background for the statements of the main results in Sections \ref{ss:cartan} to \ref{ss:cover}, and state them in the next two sections. The additional background required for the proofs are revised in Sections \ref{s:morebackg} and \ref{s:snakesandMY}. In particular, the latter is completely dedicated to reviewing the concept of MY paths and its application to studying $q$-characters of snake modules. Analysis of $\ell$-weight spaces via MY paths is the main tool of our proofs.  The remaining sections are dedicated to the proofs. \cref{t:soc} is proved in \cref{s:soc}. The main part of the proof is stated as an intermediary result, \Cref{p:tpsocmap}, which is proved in \cref{ss:tpsocmap}. \cref{p:keystepgen} is proved in \cref{s:JHs}, while \cref{t:imagfromsnakegen} is proved in \cref{s:imagproof}. The latter is rather lengthy, so the section has several subsections. In \cref{ss:thinwsp}, we show the $\ell$-weight spaces corresponding to $\bs\omega_{\bs s\diamond\bs s'}$ are one-dimensional both in
\begin{equation*}
	V(\bs\omega_{\bs s})\otimes V(\bs\omega_{\bs s'}) \quad\text{and}\quad (V(\bs\omega_{\bs s})\otimes V(\bs\omega_{\bs s'}))^{\otimes 2}. 
\end{equation*}
The logical structure of the proof of \cref{t:imagfromsnakegen} is explained in \cref{ss:imagfromsnakegen}, alongside several remarks about the extra assumption. The first main step of the proof, the claim \eqref{e:extrasq}, is proved in \cref{ss:extrasq}. The second main step, the claim \eqref{e:novarpivarpiW}, is proved in \cref{s:novarpivarpiW} after a crucial preparation in \cref{ss:crfact}. This is where the extra assumption is used. In fact, we give an example showing \eqref{e:novarpivarpiW} may fail without the extra assumption. Still, the module $V(\bs\pi_1)$ arising from this example is imaginary as a consequence of the extra cases proved in Sections \ref{ss:l=2} and \ref{ss:KRcase}. The former proves the conjectures for $k=1$ without the extra assumption in the case of length-$2$ ladders, while the latter proves the same replacing the length-$2$ assumption by the assumption that both snake modules are KR modules.
It is interesting to note that \cref{t:imagfromsnakegen} (with the extra assumption), together with diagram subalgebra arguments, is the main tool used in Sections \ref{ss:l=2} and \ref{ss:KRcase}. In \cref{s:count}, we characterize all pairs of  ladders $(\bs s,\bs s')$ satisfying $\bs s \rhd_1\bs s'$.

\section{Preliminaries and Main Statements}

Throughout the paper, let $\mathbb C,\mathbb Q$, and  $\mathbb Z$ denote the sets of complex, rational, and integer numbers, respectively. Let also $\mathbb Z_{\ge m} ,\mathbb Z_{< m}$, etc., denote the obvious subsets of $\mathbb Z$. Given $a,b\in\mathbb Q$, we will use the notation $b\le a$ to mean
\begin{equation*}
	a-b\in\mathbb Z_{\ge 0},
\end{equation*}
and similarly for $b<a$. 

The symbol $\cong$ means ``isomorphic to''. We shall use the symbol $\diamond$ to mark the end of remarks, examples, and statements of results whose proofs are postponed. The symbol \qedsymbol\ will mark the end of proofs as well as of statements whose proofs are omitted. 

\subsection{Cartan Data and Basic Notation}\label{ss:cartan}
	Let $\lie g$ be a simple Lie algebra of type $A_N$  over a characteristic zero algebraically closed field $\mathbb F$ and let $I$ be the set of nodes of its Dynkin diagram. We let $x_i^\pm, h_i, i\in I$, denote generators as in Serre's Theorem and let $\lie g=\lie n^-\oplus\lie h\oplus \lie n^+$ be the corresponding triangular decomposition. 
	We shall identify $I$ with the integer interval $[1,N]$ as usual. To shorten notation, we will often use the notation $h=N+1$ for the Coxeter number of $\lie g$. 
	For $i\in I$, let $i^*= w_0(i)$, where $w_0$ is the Dynkin diagram automorphism induced by the longest element of the Weyl group. More precisely, $i^*=h-i$. 
	
	We let $P$ denote the weight lattice of $\lie g$ and $P^+$ be the subset of dominant weights. The root system, set of positive roots, root lattice and the corresponding positive cone will be denoted, respectively, by $R, R^+, Q,Q^+$, while the fundamental weights and simple roots will be denoted by $\omega_i, \alpha_i,i\in I$. For convenience, we set $\omega_i=0$ for $i\in h\mathbb Z$.  The notation $V(\lambda),\lambda\in P^+$, denotes a simple $\lie g$-module of highest weight $\lambda$. 
	
    Consider the quantum affine (actually loop) algebra $U_q(\tlie g)$ over $\mathbb F$, where $q\in\mathbb F^\times$ is not a root of unity, where $\tlie g=\lie g\otimes \mathbb F[t,t^{-1}]$. We use the presentation in terms of generators and relations and the Hopf algebra structure as in \cite{M11}. In particular, the generators are denoted by $x_{i,r}^\pm, k_i^{\pm 1}, h_{i,s}, i\in I, r,s\in\mathbb Z, s\ne 0$. 
	
\subsection{Finite-Dimensional Modules}\label{ss:fdm}
    For  $i\in I$, $a\in \mathbb Z$, we let $\bs\varpi_{i,a}$ denote the corresponding fundamental $\ell$-weight, which is the Drinfeld polynomial whose unique non constant entry is the $i$-th which is equal to $1-q^au\in\mathbb F[u]$. For convenience, given $i\in h\mathbb Z$, we let $\bs\varpi_{i,a}$ denote the tuple $\bs 1$ whose entries are all constant.
	We let $\mathcal P^+$ denote the monoid multiplicatively generated by such elements and let $\mathcal P$ be the corresponding abelian group. We also let $\mathcal C$ be the full subcategory of that of fintie-dimensional $U_q(\tlie g)$-modules whose simple factors have highest $\ell$-weights in $\mathcal P^+$ and, hence, $\ell$-weights in $\mathcal P$. It was proven in \cite{HL10} that $\mathcal C$ is a rigid tensor category.
    For $\bs\pi\in\mathcal P^+$, $V(\bs\pi)$ will denote a simple $U_q(\tlie g)$-module whose highest $\ell$-weight is $\bs\pi$, while $W(\bs\pi)$ will stand for the corresponding (local) Weyl module. In particular, if $V$ is a finite-dimensional highest-$\ell$-weight $U_q(\tlie g)$-module with highest $\ell$-weight $\bs\pi$, we have an epimorphism
	\begin{equation*}
		W(\bs\pi) \to V. 
	\end{equation*}

	For an object $V\in\mathcal C$ and $\bs\omega\in\mathcal P$, $V_{\bs\omega}$ will denote the associated $\ell$-weight space and we set
	\begin{equation*}
		\wt_\ell(V) = \{\bs\omega\in\mathcal P:V_{\bs\omega}\ne 0\} \quad\text{and}\quad 	\wt_\ell^\pm(V) = 	\wt_\ell(V) \cap(\mathcal P^+)^{\pm 1}. 
	\end{equation*}
	It was proved in \cite{FR99} that, if $V'$ is another object of $\cal C$, then 
\begin{equation}
	\wt_\ell(V\otimes V') = \wt_\ell(V)\wt_\ell(V').
\end{equation}
We also let $V^*$ and ${}^* V$ be the dual modules such that the evaluation maps
	\begin{equation*}
		V^*\otimes V \to \mathbb F \quad\text{and}\quad V\otimes {}^*V\to\mathbb F
	\end{equation*}
	are module homomorphisms (cf. \Cref{ss:Hopf}), where $\mathbb F$ is regarded as the trivial representation $V(\bs 1)$.

\subsection{Multisegments and $\ell$-weights}
	
 It will be convenient to use a different index set for the simple objects of $\cal C$. Given $i,j\in\frac{1}{2}\mathbb Z$ with $i\le j$, let 
\begin{equation*}
	[i,j]=\{k\in\mathbb Q: i\le k\le j \}. 
\end{equation*}
Equivalently, $k\in[i,j]$ if, and only if, $j-k, k-i\in\mathbb Z_{\ge 0}$. We shall say $[i,j]$ is an $N$-segment (or simply a segment for short) if $j-i\leq h$. If $0< j-i< h$, we will say it is a proper segment. We denote by $\seg$ and $\tilde\seg$ the sets of segments and proper segments, respectively.  Letting $\mathcal F=\{\bs\varpi_{i,a}:i\in I,a \in \mathbb Z\}$, there exists a bijective map
	\begin{equation}\label{e:fundtoint}
		\mathcal F \to \tilde{\seg}, \quad \bs\varpi_{i,a}\mapsto [(a-i)/2,(a+i)/2]
	\end{equation}
	whose inverse is given by
	\begin{equation}\label{e:inttofund}
		\tilde{\seg} \to \mathcal F, \quad [i,j]\mapsto \bs\varpi_{j-i,i+j}.
	\end{equation}
	Set
	\begin{equation*}
		\bs\omega_{i,j} = \bs\varpi_{j-i,i+j} \quad\text{for}\quad [i,j]\in\seg, 
	\end{equation*}
	and note $\bs\omega_{i,i}=\bs\bomega_{i,i+h}=\bold 1$. We shall also use the notation 
	\begin{equation*}
		\bs\omega_s = \bs\omega_{i,j} \quad\text{if}\quad s = [i,j]\in\seg. 
	\end{equation*}
	It follows that $\mathcal P$ is the free abelian group generated by $\bs\omega_{i,j}$ with $[i,j]\in\tilde{\seg}$. Hence,  the isomorphism classes of simple objects in $\mathcal C$ are in bijection with the free abelian group generate by $\tilde{\seg}$.

	Set 
	\begin{equation*}
		\mathbb M = \bigcup_{l\ge 0} \seg^l \quad\text{and}\quad \tilde{\mathbb M} = \bigcup_{l\ge 0} \tilde{\seg}^l.
	\end{equation*}
	Here, for convenience, we have set $\tilde{\seg}^0= {\seg}^0=\{\emptyset\}$. We shall refer to an element  $\bs s\in\seg^l$ as a multisegment of length $\ell(\bs s)=l$. It will be said to be a proper multisegment if $\bs s\in\tilde{\seg}^l$. Given $\bs s=(s_1,\dots,s_l)\in\mathbb M$, set also 
	\begin{equation*}
		\bs\omega_{\bs s} = \prod_{k=1}^l \bs\omega_{s_k},
	\end{equation*}
	where we understand $\bs\omega_{\emptyset}=\bs 1$. One easily sees that, if $\bs s,\bs s'\in\tilde{\mathbb M}$, then $\bs\omega_{\bs s} = \bs\omega_{\bs s'}$ if and only if $\bs s'$ is obtained from $\bs s$ by a permutation of its entries. 
	Hence, the isomorphism classes of simple objects in $\mathcal C$ are in bijection with the equivalence classes of the corresponding equivalence relation $\bs s\sim\bs s'$. 
	
	Given $s=[i,j]\in \seg$, set
	\begin{equation}\label{e:dualint}
		s^* = [j-h,i] \quad\text{and}\quad {}^*\!s = [j,i+h]. 
	\end{equation}
	One easily checks
	\begin{equation*}
		\quad s^* \in \seg \quad {\rm and} \quad  {}^*\!s \in \seg.
	\end{equation*}
	If $\bs s=(s_1,\dots,s_l)\in \seg^l$, set
	\begin{equation*}
		\bs s^* = (s_1^*,\dots,s_l^*),\quad {}^*\!\bs s = ({}^*\!s_1,\dots,{}^*\!s_l),
	\end{equation*}
	and
	\begin{equation}\label{e:s(k,m)}
		\bs s(k,m)= (s_{k+1},\dots, s_m) \ \ \text{for}\ \  0\leq k<m \leq l.
	\end{equation}
	The Hopf algebra structure on $U_q(\tlie g)$ is chosen so that, if $V=V(\bs\omega_{\bs s})$, then $V^*\cong V(\bs\omega_{\bs s^*})$ and  ${}^*V(\bs\omega_{\bs s})\cong V(\bs\omega_{{}^*\!\bs s})$. 	We also recall that the lowest $\ell$-weight of $W(\bs\omega_{\bs s})$ is $(\bs\omega_{{}^*\!\bs s})^{-1}$.

	Given $s=[i,j]\in \tilde{\seg}$, the following elements are known as simple $\ell$-roots:
	\begin{equation}\label{e:slroots}
		\bs\alpha_s=\bs\alpha_{i,j} = \bs\omega_{i,j}\bs\omega_{i+1,j+1}(\bs\omega_{i+1,j}\bs\omega_{i,j+1})^{-1}.
	\end{equation}
	The subgroup of $\mathcal P$ generated by them is called the $\ell$-root lattice of $U_q(\tlie g)$ and will be denoted by $\mathcal Q$. It is well-known that $\mathcal Q$ is freely generated by the simple $\ell$-roots. Let also $\cal Q^+$ be the submonoid generated by the simple $\ell$-roots.
    Consider the partial order on $\mathcal P$ given by
	\begin{equation*}
		\bs\varpi\le\bs\omega \quad {\rm if}\quad \bs\omega\bs\varpi^{-1}\in\cal Q^+.
	\end{equation*}

	Inspired by \eqref{e:inttofund}, let us introduce the following terminology and notation. If $s=[i,j]\in\tilde{\seg}$, by the (Dynkin) support and the spectral parameter of $s$ we mean, respectively, the numbers
	
	\begin{equation}
		\supp(s) = j-i \quad\text{and}\quad \spar(s) = i+j. 
	\end{equation}
	For $\bs s\in \tilde{\mathbb M}$, set also
	\begin{equation}\label{e:supspar}
		\supp(\bs s) = \{\sup(s_k): 1\le k\le\ell(\bs s)\} \quad\text{and}\quad \spar(\bs s) = \{\spar(s_k):1\le k\le \ell(\bs s)\}. 
	\end{equation}

\subsection{Coverings and Ladders} \label{ss:cover}
\newcommand{\lad}{\mathbb L}		
Two segments $s_1, s_2$ will be said to be overlapping if
\begin{equation}
	s_1\cap s_2\ne \emptyset \quad\text{and}\quad s_m \ne s_1\cup s_2 \ \text{for}\ m\in\{1,2\}.  
\end{equation}
If $s_1=[i_1,j_1]$ and $s_2=[i_2,j_2]$, this is equivalent to saying that there exists $k,l\in\{1,2\}, k\ne l$, such that
	\begin{equation*}
		i_k<i_l\leq j_k<j_l.
	\end{equation*}
In that case, note
\begin{equation*}
	s_1\cup s_2 = [i_k,j_l] \quad\text{and}\quad s_1\cap s_2 = [i_l,j_k].
\end{equation*}
It is important to note that, if $i_2-i_1\notin \mathbb Z$ or, equivalently, $j_2-j_1\notin \mathbb Z$, then $s_1$ and $s_2$ do not overlap since $s_1\cap s_2=\emptyset$.

We will say $(s_1,s_2)$ is a covering pair 
if it is formed by an overlapping pair of segments such that 
\begin{equation}
	s_1\cup s_2\in\seg \quad\text{and}\quad \spar(s_2)>\spar(s_1). 
\end{equation}
In that case, we may also say  $s_2$ covers $s_1$ and write $s_2\rhd s_1$. Moreover, $s_1$ and $s_2$ are necessarily proper segments and $s_1\cap s_2\in\seg$. One also easily checks
\begin{equation}\label{e:nabig}
	s_2\rhd s_1 \quad\Rightarrow\quad \spar(s_1)< \spar(s_1\cap s_2),\, \spar(s_1\cup s_2) <\spar(s_2)
\end{equation}
as well as
\begin{equation}\label{e:dualcover}
	s_2\rhd s_1 \quad\Rightarrow\quad (1+{}^*\!s_1)\rhd s_2,
\end{equation}
where $(1+s) = [i+1,j+1]$ if $s=[i,j]$. 
Furthermore,
\begin{equation}\label{e:concacudual}
	s_2\rhd s_1 \ \ \Leftrightarrow\ \ s_2^*\rhd s_1^*, \quad (s_1\cap s_2)^* = s_1^*\cup s_2^*, \quad (s_1\cup s_2)^* = s_1^*\cap s_2^*,
\end{equation}
and similarly for left duality. 

We shall say the set $\{s_1,s_2\}$ is $h$-connected (or simply connected for short) if
\begin{equation}
	\text{either}\ \ s_1\rhd s_2\ \ \text{or}\ \ s_2\rhd s_1
\end{equation}
or, equivalently, if $s_1$ and $s_2$ are overlapping intervals such that $s_1\cup s_2\in\seg$. 
We say $\bs s=(s_1,\cdots, s_l)\in\mseg$ covers $\bs s'=(s'_1,\dots,s'_l)\in\mseg$ and write $\bs s\rhd\bs s'$, if 
\begin{equation}
	s_k\rhd s'_k  \ \ \text{for all}\ 1\le k\le l.
\end{equation} 
The multisegment $\bs s$ is said to be a ladder if, writing $s_k=[i_k,j_k]$, we have
\begin{equation}\label{e:deflad}
	i_1<\cdots <i_l \ \ {\rm and} \ \  j_1<\cdots< j_l.
\end{equation}
In that case, it is straightforward from the definition that
\begin{equation}
	1\le k_1<k_2<\cdots< k_r \le l \quad\Rightarrow\quad (s_{k_1},\dots,s_{k_r}) \ \text{is a ladder.}
\end{equation} 
It is not difficult to check that \eqref{e:deflad} is equivalent to
\begin{equation}\label{e:defladsusp}
	\spar(s_{k_2})-\spar(s_{k_1}) > |\supp(s_{k_2})-\supp(s_{k_1})| \quad\text{for all}\quad 1\le k_1<k_2\le l.
\end{equation}
A connected ladder is a ladder satisfying
\begin{equation}
	s_{k+1}\rhd s_k \quad\text{for all}\quad 1\le k<l.
\end{equation} 
We shall denote by $\lad^l$ the set of ladders of length $l$ and by $\lad$ the set of all ladders. The following lemma is easily established.

\begin{lem}\label{l:propdiamond}
	If $(s_1,s_2)\in\lad$ and $s\rhd s_k, k=1,2$, then $\{s\cap s_2,s\cup s_2\}\subseteq\tilde\seg$.\qed 
\end{lem}

We shall need the following (cf. \cite[Lemma 5.1.3]{MS24}).

\begin{lem}\label{l:complete}
	If $\bs s=(s_1,\dots,s_l)$ is a ladder and $s_l\rhd s_1$,  then $\bs s$ is a connected ladder.
\end{lem}

\begin{proof}
	We need to show $s_{k+1}\rhd s_k$ for each $1\leq k<l$. Write $s_k=[i_k,j_k]$ and note the assumptions that $s_l\rhd s_1$  implies
	\begin{equation*}
		i_1<i_l\le j_1<j_l.
	\end{equation*}
	In particular, since $j_1\le j_k<j_l$ and $i_1<i_{k+1}\le i_l$ for all $1\le k<l$, we get
	\begin{equation*}
		0\le j_1-i_l\le j_k-i_{k+1}<j_l-i_1
	\end{equation*}
	and, hence,
	\begin{equation*}
		s_k\cap s_{k+1} = [i_{k+1},j_k]\ne\emptyset \quad\text{and}\quad s_k\cup s_{k+1} = [i_k,j_{k+1}]\subseteq s_1\cup s_l.
	\end{equation*}
	It immediate follows that $s_k$ and $s_{k+1}$ overlap and, since $s_1\cup s_l\in\seg$, we also have $s_k\cup s_{k+1}\in\seg$. Moreover, the assumption that $\bs s$ is a ladder implies $\spar(s_k)<\spar(s_{k+1})$ and, hence, $s_{k+1}\rhd s_k$ as desired.
\end{proof}

	If $\bs s=(s_1,\dots,s_l)$ and $\bs s' = (s'_1,\dots,s'_m)$, set
\begin{equation}
	\bs s\vee\bs s' = (s_1,\dots,s_l,s'_1,\dots,s'_m). 
\end{equation}
If $\bs s\rhd\bs s'$, let $\bs s\cap\bs s'$ and $\bs s \cup\bs s'$ be the multisegments whose $k$-th entries are, respectively,
\begin{equation*}
	s_k\cap s'_k \quad\text{and}\quad s_k\cup s'_k. 
\end{equation*}
One easily checks
	\begin{equation}\label{e:capcuplad}
		\bs s,\bs s'\in\lad \ \ \text{and}\ \ \bs s\rhd\bs s' \quad\Rightarrow\quad \bs s\cup\bs s',\bs s\cap\bs s'\in\lad.
	\end{equation}
Set also
\begin{equation}
	\bs s\diamond \bs s' = (\bs s\cap\bs s')\vee(\bs s \cup\bs s'). 
\end{equation}

\subsection{Coverings and Socles of Certain Snake Tensor Products}\label{ss:soc} 
A necessary and sufficient condition for a tensor product of fundamental modules to be reducible is well-known. It can be stated in terms of segments as follows. 
	\begin{equation}
		V(\bs\omega_{s_1})\otimes V(\bs\omega_{s_2}) \ \ {\rm is\ reducible}\ \  \Leftrightarrow \ \ \{s_1,s_2\} \ \ {\rm is\ connected}.
	\end{equation}
	Moreover, in that case, 
	\begin{equation}
		V(\bs\omega_{s_1\cap s_2})\otimes V(\bs\omega_{s_1\cup s_2})\cong V(\bs\omega_{s_1\cap s_2}\,\bs\omega_{s_1\cup s_2})  
	\end{equation}
	and, if $s_2\rhd s_1$, we have the following short exact sequences
	\begin{equation}\label{e:tpfundnhw}
		\begin{aligned}
			0\to V(\bs\omega_{s_1\cap s_2})\otimes V(\bs\omega_{s_1\cup s_2})\to V(\bs\omega_{s_2})\otimes V(\bs\omega_{s_1})\to V(\bs\omega_{s_1}\bs\omega_{s_2})\to 0,\\
		0\to V(\bs\omega_{s_1}\bs\omega_{s_2}) \to V(\bs\omega_{s_1})\otimes V(\bs\omega_{s_2})\to V(\bs\omega_{s_1\cap s_2})\otimes V(\bs\omega_{s_1\cup s_2})\to 0.
		\end{aligned}
	\end{equation}

	If $\bs s$ is a ladder, the module $V(\bs\omega_{\bs s})$ is known as a snake module. In that case, $V(\bs\omega_{\bs s})$ is prime if and only if $\bs s$ is a connected ladder. 	
The main result of this section is the following generalization of \eqref{e:tpfundnhw}.
We let $\soc V$ and ${\rm hd} V$ stand for the socle and head of the module $V$.

\begin{thm}\label{t:soc}
	If $\bs s,\bs s'\in\mathbb L^l$ are such that $\bs s\rhd\bs s'$, then 
	\begin{equation}\label{e:soc}
		\soc (V(\bs\omega_{\bs s})\otimes V(\bs\omega_{\bs s'}))\cong V(\bs\omega_{\bs s\diamond\bs s'})\cong {\rm hd}(V(\bs\omega_{\bs s'})\otimes V(\bs\omega_{\bs s})).
	\end{equation}\endd
\end{thm}

It follows from \eqref{e:tpfundnhw} that, if $l=\ell(\bs s)$, we have a maps
\begin{equation}
	V(\bs\omega_{s_k\cap s'_k})\otimes V(\bs\omega_{s_k\cup s'_k})\to V(\bs\omega_{s_k})\otimes V(\bs\omega_{s'_k}) \quad\text{for all}\quad 1\le k\le l.
\end{equation}
The most intricate part of the proof of \Cref{t:soc} is the following proposition whose proof will be given in \Cref{ss:tpsocmap}.

\begin{prop}\label{p:tpsocmap}
	There exists a non-zero map 
	\begin{equation*}
		V(\bs\omega_{\bs s'})\otimes V(\bs\omega_{\bs s})\to V(\bs\omega_{\bs s\cup \bs s'})\otimes V(\bs\omega_{\bs s\cap \bs s'}).
	\end{equation*}\endd
\end{prop}	

\subsection{A Construction of Imaginary Modules} 
Let $\bs s,\bs s' \in \seg^l$ and recall \eqref{e:s(k,m)}.  For $0\leq p <l$, we say that $\bs s$ is a $p$-cover of $\bs s'$, and write $\bs s \rhd_p \bs s'$,  if 
\begin{equation*}
	\bs s(0,l-k)\rhd \bs s'(k,l) \quad\text{for all}\quad 0\leq k\leq p.
\end{equation*}
In that case, for each $0\leq k\leq p$, define 
\begin{equation*}
	\bs s\cup_k\bs s' = \bs s(0,l-k)\cup\bs s'(k,l), \ \ \ \bs s\cap_k\bs s' = \bs s(0,l-k)\cap\bs s'(k,l),
\end{equation*}
and $\bs s\diamond_k \bs s' = (\bs s\cap_k\bs s')\vee(\bs s\cup_k\bs s')$. To shorten notation, set 
\begin{equation}\label{pikdef}
	\bs\pi_k = \bs\omega_{\bs s'(0,k)}\, \bs\varpi_k\,\bs\omega_{\bs s(l-k,l)} \quad\text{with}\quad \bs\varpi_k = \bs\omega_{\bs s\diamond_k\bs s' }.
\end{equation}
In particular, $\bs\omega_{\bs s\diamond\bs s'} = \bs\pi_0 = \bs\varpi_0$.

\begin{prop}\label{p:keystepgen} Let $\bs s,\bs s'\in \mathbb L^l$, suppose $\bs s\rhd_p \bs s'$ for some $0\leq p<l$, and set 
$$V= V(\bs\omega_{\bs s})\otimes V(\bs\omega_{\bs s'}).$$
If $\bs\pi\in \mathcal P^+$ and $\bs\pi\le \bs\pi_p$,  then
	\begin{equation}\label{e:hlwsubs}
		\Hom(W(\bs\pi),V)\neq 0 \quad\Leftrightarrow\quad \bs\pi \in \{\bs\pi_k: 0\le k\leq  p\}.
	\end{equation}
	Moreover, there exist a chain of  highest $\ell$-weight submodules
	\begin{equation*}
		{\rm soc}\left(V\right) = M_0\subseteq M_1\subseteq \cdots \subseteq M_p \subseteq V
	\end{equation*}
	such that the highest $\ell$-weight of $M_k$ is $\bs\pi_k, 0\le k\leq p$. 
	\endd
\end{prop}

The first part of \Cref{p:keystepgen} is  proved in \Cref{ss:keystepgen}, while the second is proved in \Cref{ss:keystepgeninc}. In the notation of \Cref{p:keystepgen}, we pose the following conjecture.

\begin{con} For all $1\le k\le p$,
		\begin{equation}\label{e:cjimag}
		M_{k}/M_{k-1} \ \ \text{is simple and imaginary.}
	\end{equation}\endd
\end{con}

The following is our main result. The proof is given in \cref{s:imagproof}.

\begin{thm}\label{t:imagfromsnakegen}
	The first  claim in \eqref{e:cjimag} is true for $k=1$ and, moreover, there exists a nonzero module map $\psi:M_1^{\otimes 2}\to V(\bs\omega)$ for some $\bs\omega\ne\bs\pi_1^2$. If, in addition, 
	\begin{equation}\label{almostdual'}
		s_2\cap s_1' = \emptyset \quad\text{and}\quad {}^*\!s_2\cap {}^*\!s_1'= \emptyset,
	\end{equation}
	then $\psi$ induces a nonzero map $V(\bs\pi_1)^{\otimes 2}\to V(\bs\omega)$. In particular, the second claim in \eqref{e:cjimag} is also true for $k=1$ in this case.\endd
\end{thm}

We remark that \eqref{almostdual'} is equivalent to
\begin{equation}\label{almostdual}
	j_1'<i_2 \quad\text{and}\quad j_2-i_1'>h.
\end{equation}
This theorem was proved in \cite{BC23} under more restrictive assumptions: $\bs s' = \bs s^*, |\supp(\bs s)|=1$, and $\supp(\bs s)\cap\{1,N\}=\emptyset$. The last assumption guarantees $\bs s\rhd_1\bs s^*$. Assumption \eqref{almostdual} holds in that case since $\bs s' = \bs s^*$ implies that $j_1'= i_1$ and $j_1=h+i_1'$ and, hence, using that $\bs s\in \mathbb L$, we have
\begin{equation*}
	j_1'=i_1<i_2 \ \ {\rm and}\ \ j_2>j_1=h+i_1'.
\end{equation*}
As mentioned in the introduction, it was obtained in \cite{LM18} a necessary and sufficient condition for an irreducible module $V(\bs\omega_{\bs s})$ to be imaginary, provided $\bs s$ is regular, i.e., if 
\begin{equation*}
	i_k\ne i_m \ \ \text{and}\ \ j_k\ne j_m \ \ \text{for all}\ \ 1\leq k\neq m\leq l.
\end{equation*}
Clearly, \cref{t:imagfromsnakegen} gives rise to families of imaginary modules which do not satisfy this  regularity condition.

The logical structure of the proof of \cref{t:imagfromsnakegen} is explained in \cref{ss:imagfromsnakegen}. In particular, \eqref{almostdual'} is used as a technical assumption to guarantee \eqref{e:novarpivarpiW} holds. In fact, \cref{exweylfails} below shows \eqref{e:novarpivarpiW} may fail if \eqref{almostdual'} fails. However, we also show that the second part of \eqref{e:cjimag} holds for $k=1$ for this example.  More precisely, we show in \cref{ss:l=2} that  \eqref{e:cjimag} holds for $k=1$ in complete generality when $l=2$.  Alternatively, the validity of \eqref{e:cjimag} for \cref{exweylfails} also follows from the discussion in \cref{ss:KRcase}, where we show \eqref{e:cjimag} holds for $k=1$ if $V(\bs\omega_{\bs s})$ and $V(\bs\omega_{\bs s'})$ are Kirillov-Reshetikhin modules without the assumption \eqref{almostdual'}. So far, we have not found a general argument in the case \eqref{e:novarpivarpiW} fails.

In \cref{s:count}, we will characterize all pairs of connected ladders $(\bs s,\bs s')$ satisfying $\bs s \rhd_1\bs s'$. In the next example, we show how to construct  ladders satisfying all the assumptions in \cref{t:imagfromsnakegen}, including \eqref{almostdual'}. 
	
\begin{ex}\label{ex:imagfromsnakegen}  
For any choice of ladders $\bs s,\bs s'$ such that $\bs s \rhd_1 \bs s'$, we will check that there exist $s,s'\in\tilde\seg$ such that replacing $(\bs s,\bs s')$ by $(\tilde{\bs s},\tilde{\bs s}')$ where
\begin{equation*}
	\tilde{\bs s} = s\vee \bs s \quad\text{and}\quad \tilde{\bs s}' = s'\vee \bs s',
\end{equation*}
all the assumptions of the theorem are satisfied. Indeed, one easily checks the conditions are satisfied provided 
	\begin{equation}\label{e:imagfromsnake}
		s_1\rhd s \rhd s_1'\rhd s', \ \ \  s\rhd s', \ \ \  j'<i_1, \ \ \ j_1-i'>h,
	\end{equation}
	where $s=[i,j], s'=[i',j']$. Let us check that such $s$ and $s'$ do exist.  Since $\bs s\rhd_1\bs s'$ we have $s_1\rhd s_2'\rhd s_1'$ and hence 
	$\{s\in \tilde\seg : s_1\rhd s\rhd s_1'\}\neq \emptyset$. Let $s=[i,j]$ be any element of this set. Since $j_1-i_1'\leq h$ the following inequalities hold
	\begin{equation*}
		j-h<i_1'<i<i_1 \leq j_1'<j< j_1.
	\end{equation*}
	In particular, there exists $i',j'$ such that  
	\begin{equation*}
		j-h\leq i'<j_1-h \leq i_1'< i \leq j'<i_1\leq j_1'<j,
	\end{equation*}
	which proves the existence of $s$ and $s'$ satisfying \eqref{e:imagfromsnake}, as desired. \endd
\end{ex}

We now give an example connecting \cref{t:imagfromsnakegen} with \cref{t:soc}. 
	
\begin{ex}\label{ex:pi1issoc}
	Let $\bs s,\bs s'\in \mathbb L^l$ be such that $\bs s\rhd_1\bs s'$ and let $\bs\pi_1$ be as in \eqref{pikdef}. 	Let us show that there exist $\tilde{\bs s}, \tilde{\bs s}'\in \mathbb L$ such that 
	\begin{equation}\label{imagsocle}
		\bs\omega_{\tilde{\bs s}\diamond\tilde{\bs s}'}=\bs\pi_1.
	\end{equation}
	In particular, \cref{t:soc} implies
	\begin{equation*}
		V(\bs\pi_1)\cong \soc (V(\bs\omega_{\tilde{\bs s}})\otimes V(\bs\omega_{\tilde{\bs s}'})).
	\end{equation*}
	In other words, any $\bs\pi_1$ defined as in \eqref{pikdef} is the Drinfeld polynomial of the socle of a tensor product of snake modules. In particular, in light of \cref{ex:imagfromsnakegen}, this construction provides families of examples of tensor products of snake modules whose socles are imaginary.
	
	Using the fact that $\bs s\rhd_1\bs s'$ and $\bs s,\bs s'\in \mathbb L^l$, we have
	\begin{equation*}
		i_1'<i_2'<i_1\leq j_1'<j_1 \ \ {\rm and}\ \ i_{l}'<i_l \leq j_{l}'< j_{l-1}<j_l.
	\end{equation*}
	In particular, there exists $a,b$ such that 
	\begin{equation}\label{absnake}
		i_1'<a<i_1\leq j_1'<j_2' \ \ {\rm and} \ \  i_{l}'<i_l\leq j_l'<b<j_l.
	\end{equation}
	Setting
	\begin{equation*}
		\tilde{\bs s} = ([a,j_1'])\vee \bs s(0,l-1)\vee ([b,j_l]), \ \ \tilde{\bs s}' = ([i_1',a])\vee \bs s'(1,l)\vee ([i_l,b]),
	\end{equation*}
	a straightforward checking using  \eqref{absnake}  shows that $\tilde{\bs s},\tilde{\bs s}'\in \mathbb L^{l+1}$, $\tilde{\bs s}\rhd \tilde{\bs s}$, and
	\eqref{imagsocle} holds.\endd
\end{ex}

\section{Further Background and Technical Results}\label{s:morebackg}

\subsection{General Hopf Algebra Results}\label{ss:Hopf} 
Given a Hopf algebra $\mathcal H$ over $\mathbb F$, its category $\mathcal C$ of finite-dimensional representations is an abelian monoidal category and we denote the dual of a module  $V$ by $V^*$. More precisely, the action of $\mathcal H$ of $V^*$ is given by
\begin{equation}\label{rightdual}
	(hf)(v) = f(S(h)v) \quad\text{for}\quad h\in\mathcal H, f\in V^*, v\in V.
\end{equation}
It is well known that 
\begin{equation*}
	\operatorname{Hom}_{\mathcal H}(\mathbb F, V\otimes V^*) \ne 0 \qquad\text{and}\qquad \operatorname{Hom}_{\mathcal H}(V^*\otimes V, \mathbb F)\ne 0.
\end{equation*}
If the antipode is invertible, one can define a second notion of dual module by replacing $S$ by $S^{-1}$ in \eqref{rightdual}, which will be  denoted by $^*V$. The following are also well-known:
\begin{equation}\label{e:rldualcancel}
	^*(V^*)\cong (^*V)^* \cong V
\end{equation}
and, given $\mathcal H$-modules $V_1,V_2,V_3$, 
\begin{equation}\label{e:frobrec}
	\begin{aligned}
		\operatorname{Hom}_{\mathcal C}(V_1\otimes V_2, V_3)\cong \operatorname{Hom}_{\mathcal C}(V_1, V_3\otimes V_2^*), \\
		\operatorname{Hom}_{\mathcal C}(V_1, V_2\otimes V_3)\cong \operatorname{Hom}_{\mathcal C}(V_2^*\otimes V_1, V_3),
	\end{aligned}
\end{equation}
and 
\begin{equation}\label{e:dualtp}
	(V_1\otimes V_2)^*\cong V_2^*\otimes V_1^*.
\end{equation}

\begin{lem}[{\cite[Lemma 2.6.2]{MS24}}]\label{l:nonzeromorph}
	Let $V_1,V_2,V_3, L_1,L_2$ be $\mathcal H$-modules and assume  $V_2$ is simple. If
	\begin{equation*}
		\varphi_1:L_1\rightarrow V_1\otimes V_2\quad\textrm{and}\quad\varphi_2: V_2\otimes V_3\rightarrow L_2
	\end{equation*}
	are nonzero homomorphisms, the composition
	$$L_1\otimes V_3 \xrightarrow{\varphi_1\otimes \operatorname{id}_{V_3}} V_1\otimes V_2\otimes V_3 \xrightarrow{\operatorname{id}_{V_1}\otimes \varphi_2} V_1\otimes L_2	$$
	does not vanish. Similarly, if 
	\begin{equation*}
		\varphi_1:V_1\otimes V_2\rightarrow L_1\quad\textrm{and}\quad\varphi_2:L_2\rightarrow V_2\otimes V_3
	\end{equation*}
	are nonzero homomorphisms, the composition
	$$V_1\otimes L_2\xrightarrow{\operatorname{id}_{V_1}\otimes\varphi_2} V_1\otimes V_2\otimes V_3 \xrightarrow{\varphi_1\otimes\operatorname{id}_{V_3}} L_1\otimes V_3$$
	does not vanish.\qed
\end{lem}

\subsection{Tensor Products and Highest-$\ell$-weight Modules}\label{ss:hwc} 

The following result will play an important role in this paper (see \cite[Theorem 7.5 and Remark 7.6]{CM05} for a proof and further related references).

\begin{prop}\label{p:weylper}
	Given $l\ge 1$, let $\bs s = (s_1,\cdots , s_l)\in \tilde\seg^l$, and assume  $\spar(s_k)\geq \spar(s_{k+1})$ for $1\leq k<l$. Then, 
	\begin{equation*}
		W(\bs\omega_{\bs s})\cong V(\bs\omega_{s_1})\otimes \cdots \otimes V(\bs\omega_{s_l}).
	\end{equation*}
	In particular, if $V_1$ and $V_2$ are quotients of $W(\bomega_{\bs s(0,k)})$ and $W(\bomega_{\bs s(k,l)})$, respectively, for some $1\leq k< l$, then $V_1\otimes V_2$  is  highest-$\ell$-weight. \qed
\end{prop}
	Recall from \cite{FR99} that 
	\begin{equation}\label{e:subgpweight}
		W(\bs\omega_{\bs s})_{\bs\varpi}\ne 0 \quad\Rightarrow\quad \bs\varpi\in \left\langle \bomega_s : \spar(s)\geq \min\{\spar(\bs s)\}\right\rangle\subset \cal P. 
	\end{equation}

\begin{lem}[{\cite[Proposition 4.1.2]{MS24}}]\label{l:subsandq}
	Let $\bs\pi,\bs\pi'\in\mathcal P^+$, $V=V(\bs\pi)\otimes V(\bs\pi')$, and $W=V(\bs\pi')\otimes V(\bs\pi)$. Then, $V$ contains a submodule isomorphic to $V(\bs\varpi), \bs\varpi\in\mathcal P^+$, if and only if there exists an epimorphism $W\to V(\bs\varpi)$.\qed
\end{lem}

\begin{thm}\label{cyc}
	Let $S_1,\cdots, S_l$ be simple $U_q(\tlie g)$-modules and assume $S_i$ is real for $i\notin J$ for some $J\in\{\{1,2\},\{l-1,l\}\}$. If $S_i\otimes S_j$ is highest-$\ell$-weight for all $1\leq i< j\leq l$, then $S_1\otimes\cdots\otimes S_l$ is highest-$\ell$-weight.\qed
\end{thm}

\begin{rem}
	An analogous statement was proved in \cite{GM21} in a different context. The proof  works essentially word by word, mutatis-mutandis, in the present context. In fact, the proof makes sense in a much broader abstract context, except the point where \cite[Lemma 2.2]{GM21} is used. However, in the context of the present paper, that lemma is known from properties of the normalized $R$-matrix.\endd
\end{rem}

We shall use the following in the proof of \cref{p:keystepgen}.

\begin{prop}\label{p:pontas}
	Let $\bs\pi_1,\bs\pi_2,\bs\varpi_1,\bs\varpi_2,\bs\omega\in\mathcal P^+$ be such that there exist nonzero homomorphisms
	\begin{equation*}
		V(\bs\varpi_1)\otimes V(\bs\pi_1)\to V(\bs\varpi_1\bs\pi_1), \quad V(\bs\pi_2)\otimes V(\bs\varpi_2)\to V(\bs\varpi_2\bs\pi_2), \quad V(\bs\omega)\to V(\bs\pi_1)\otimes V(\bs\pi_2).
	\end{equation*}
	Then, there exists a nonzero homomorphism
	\begin{equation*}
		V(\bs\varpi_1)\otimes V(\bs\omega)\otimes V(\bs\varpi_2)\to V(\bs\varpi_1\bs\pi_1)\otimes V(\bs\pi_2\bs\varpi_2). 
	\end{equation*}
	Moreover, if $V(\bs\varpi_1)\otimes V(\bs\omega),  V(\bs\omega)\otimes V(\bs\varpi_2)$, and $V(\bs\varpi_1)\otimes V(\bs\varpi_2)$ are highest-$\ell$-weight and either $V(\bs\varpi_1)$ or $V(\bs\varpi_2)$ is real, then there exists a nonzero homomorphism
	\begin{equation*}
		W(\bs\varpi_1\bs\omega\bs\varpi_2)\to V(\bs\varpi_1\bs\pi_1)\otimes V(\bs\pi_2\bs\varpi_2).
	\end{equation*}	
\end{prop}

\begin{proof}
	Begin by applying \Cref{l:nonzeromorph} with $V_1 = V(\bs\varpi_1), V_2 = V(\bs\pi_1), V_3 = V(\bs\pi_2), L_1 = V(\bs\varpi_1\bs\pi_1)$, and $L_2 = V(\bs\omega)$ to obtain a nonzero map
	\begin{equation*}
		V(\bs\varpi_1)\otimes V(\bs\omega)\to V(\bs\varpi_1\bs\pi_1)\otimes V(\bs\pi_2). 
	\end{equation*}
	Then, apply \Cref{l:nonzeromorph} with $V_1 = V(\bs\varpi_1\bs\pi_1), V_2 = V(\bs\pi_2), V_3 = V(\bs\varpi_2), L_1 = V(\bs\varpi_1)\otimes V(\bs\omega), L_2 = V(\bs\pi_2\bs\varpi_2)$ to complete the proof of the first claim. The extra assumptions for the second claim, together with \cref{cyc}, implies $W(\bs\varpi_1\bs\omega\bs\varpi_2)$ is highest-$\ell$-weight, from where the second claim immediately follows.
\end{proof}

%\subsection{Random Facts} 
We shall use the following results.

\begin{prop}[{\cite[Proposition 1.3.3]{BC23}}]\label{p:bc1.3.3}
	Let $\bs\pi,\bs\varpi\in\mathcal P^+$ and suppose $M$ and $N$ are quotients of $W(\bs\pi)$ and $W(\bs\varpi)$, respectively. Then, $M\otimes N$ is cyclic on the $\ell$-weight-space $(^*\bs\pi)^{-1}\bs\varpi$.   In particular, if $\Hom(M\otimes N,V)\ne 0$ for some module $V$, then $V_{(^*\bs\pi)^{-1}\bs\varpi}\ne 0$. \qed
\end{prop}

\begin{lem}[{\cite[Lemma 1.3.4]{BC23}}]\label{l:existhwv} 
	The following holds for any $\bs\pi,\bs\pi_1,\bs\pi_2\in\cal P^+$: 
	\begin{equation*}
		\Hom(W(\bs\pi), V(\bs\pi_1)\otimes V(\bs\pi_2))\ne 0 \ \ \Rightarrow\ \ \bs\pi\bs\pi_1^{-1}\in\wt_\ell V(\bs\pi_2)\ \ {\rm{and}}\ \ {}^*\bs\pi_2({}^*\bs\pi)^{-1}\in\wt_\ell V(\bs\pi_1).
	\end{equation*}\qed
\end{lem}

\begin{thm}\cite[Theorem 3.12]{KKKO15} \label{simsoc} 
	Suppose $\bs\pi\in\cal P^+$ is such $V(\bs\pi)$ is real. Then, for all $\bs\pi'\in\cal P^+$, the module $V(\bs\pi)\otimes V(\bs\pi')$ has simple head and simple socle.  Moreover, the socle of $V(\bs\pi)\otimes V(\bs\pi')$ is the head of $V(\bs\pi')\otimes V(\bs\pi)$.\qed
\end{thm}

\subsection{Tensoring with Right and Left Subintervals}\label{ss:tpfund} 
Given $i\le k\le j$ with $[i,j]\in\seg$, in light of  \eqref{e:tpfundnhw}, we have
\begin{equation}\label{e:tpfundsint}
	V(\bs\omega_{i,k})\otimes V(\bs\omega_{k,j})\to V(\bs\omega_{i,j}) \quad\text{and}\quad V(\bs\omega_{i,j}) \to 	V(\bs\omega_{k,j})\otimes V(\bs\omega_{i,k}).
\end{equation}

Suppose $s_1,s_2\in\tilde\seg$ satisfy $s_2\rhd s_1$ or, equivalently  $(s_1,s_2)$ is a connected ladder. 
Writing $s_k= [i_k,j_k]$, set
\begin{equation}\label{e:rlint}
	s_1 \lrcorner s_2= [j_1,j_2] \quad\text{and}\quad s_1\llcorner s_2 = [i_1,i_2].
\end{equation} 
It follows that $s_1\cap s_2 = [i_2,j_1], s_1\cup s_2 = [i_1,j_2]\in\seg$ and $s_1 \lrcorner s_2,s_1\llcorner s_2\in\tilde\seg$. In particular, we also have $(s_1 \lrcorner s_2)^*, {}^*(s_1\llcorner s_2)\in\tilde\seg$.  
It then follows from \eqref{e:tpfundsint} and \eqref{e:frobrec} that there exist nonzero maps
\begin{equation}\label{e:tplrsi}
	\begin{aligned}
		V(\bs\omega_{s_1\cap s_2})\otimes V(\bs\omega_{s_1\lrcorner s_2})\to V(\bs\omega_{s_2})&, \quad V(\bs\omega_{s_1\lrcorner s_2}^*)\otimes V(\bs\omega_{s_1\cup s_2})\to V(\bs\omega_{s_1})\\
		V(\bs\omega_{s_1\llcorner s_2})\otimes V(\bs\omega_{s_1\cap s_2})\to V(\bs\omega_{s_1})&, \quad V(\bs\omega_{s_1\cup s_2})\otimes V({}^*\bs\omega_{s_1\llcorner s_2})\to V(\bs\omega_{s_2}).
	\end{aligned}
\end{equation}
We interpret $s_1 \lrcorner s_2$ and $s_1\llcorner s_2$ as the right and left subintervals of $s_1\cup s_2$, respectively. 

Suppose $\bs s,\bs s'\in\lad^l$ are such that $\bs s\rhd \bs s'$. It follows from the first line of \eqref{e:tplrsi} that 
\begin{equation}\label{e:tplrsim}
	V(\bs\omega_{s_k\cap s'_k})\otimes V(\bs\omega_{s'_k\lrcorner s_k})\to V(\bs\omega_{s_k}) \quad\text{and}\quad V(\bs\omega_{s'_k\lrcorner s_k}^*)\otimes V(\bs\omega_{s_k\cup s'_k})\to V(\bs\omega_{s'_k})
\end{equation}
for all $1\le k\le l$. We let
\begin{equation}\label{e:rlintm}
	\bs s'\lrcorner\bs s = (s'_1\lrcorner s_1,\dots,s'_l\lrcorner s_l)
\end{equation}
and we similarly  define $\bs s'\llcorner\bs s$. One easily checks 
\begin{equation}\label{e:cornlad}
	\bs s,\bs s'\in\lad^l \ \ \text{and}\ \ \bs s\rhd \bs s' \quad\Rightarrow\quad \bs s'\lrcorner\bs s,\, \bs s'\llcorner\bs s\in\lad^l. 
\end{equation}
We note however that they are not necessarily connected ladders even if $\bs s,\bs s'$ are.

\section{Snake Modules and Mukhin-Young Lattice Paths}\label{s:snakesandMY}
We  recall some  results of Mukhin and Young  established in \cite{MY12a, MY12} which explore the combinatorics of certain diagonal rank-2 lattice paths to describe results about the $\ell$-weights of snake modules as well as certain exact sequences arising from tensor products of such modules.

\subsection{MY Paths}\label{ss:myp} We reformulate in the language of intervals the results established in \cite{MY12a}. Let $\mathbb P$ be the set 
 of all functions $p:[0,h]\to \mathbb Z$ satisfying 
\begin{equation}\label{e:preMYpath}
	p(0)\in\mathbb Z \quad\text{and}\quad |p(k+1)-p(k)|=1 \quad\text{for}\quad 0\le k<h.
\end{equation}
In particular,
\begin{equation}\label{boundp}
	|p(k)-p(k')| \leq |k-k'|, \ \ k,k'\in [0,h],
\end{equation}
and $p(k)\pm k \in\mathbb Z$ for all $k\in[0,h]$. 
Given such $p$ and $k\in[1,N]$, consider the interval
\begin{equation}\label{spkint}
	s_{p,k} = [(p(k)-k)/2,(p(k)+k)/2],
\end{equation}
which is clearly a segment and, moreover, 
\begin{equation}\label{e:supspspk}
	\supp(s_{p,k}) = k \quad\text{and}\quad \spar(s_{p,k})=p(k).
\end{equation}
One can easily check using \eqref{boundp} that 
\begin{equation}\label{e:nocpinpath}
	\{s_{p,k}, s_{p,k'}\} \quad\text{is not connected for all}\quad k,k'\in [0,h].
\end{equation}

Given $s=[i,j]\in\seg$, let $\mathbb P_{i,j}=\mathbb P_s$ be the subset of $\mathbb P$ satisfying 
\begin{equation}
	p(0) = 2j \quad\text{and}\quad p(h) = h+2i,
\end{equation}
or, equivalently,
\begin{equation}\label{edgescond}
	p(0) = \spar(s)+\supp(s) \quad\text{and}\quad 
	p(h) = \spar(s)+\supp(s)^*.
\end{equation}
We refer to such a function $p$ as an MY path. Letting $k'\in\{0,h\}$ in \eqref{boundp}, one easily checks
\begin{equation}\label{boundp'}
	\max\{2j-k,k+2i\}\le p(k)\le \min\{k+2j,2h+2i-k\} \quad\text{for all}\quad k\in [0,h].
\end{equation}
In particular,
\begin{equation}\label{boundp''}
	p(k) + k^* \ge (k+2i) +(h-k) = p(h) \quad\text{for all}\quad k\in [0,h]
\end{equation}
and
\begin{equation}\label{boundp'''}
	p(k) + k \ge  p(0) \quad\text{for all}\quad k\in [0,h].
\end{equation}

Let 
\begin{equation*}
	E^-_p = \{k\in[1,N]: p(k\pm 1)=p(k) - 1\} \quad\text{and}\quad E^+_p = \{k\in[1,N]: p(k\pm 1)=p(k) + 1\}. 
\end{equation*}
In other words, $E^\pm_p$ are the points of local maxima and minima of $p$ in $[1,N]$. Let $E_p=E_p^+\cup E_p^-$.
It is straightforward to check that
\begin{equation}\label{e:bottonline}
	p(k) + k^* = p(h) \ \ \Rightarrow\ \ E_p\subseteq [1,k]  \quad\text{and}\quad 
	p(k) + k = p(0) \ \ \Rightarrow\ \ E_p\subseteq [k,N]. 
\end{equation}
Given $k\in E^\pm_p$, let $\tau_kp:[0,h]\to \mathbb Z$ be given by 
\begin{equation*}
	\tau_kp(m) = p(m) \ \ \text{if}\ \ m\ne k \ \ \text{and}\ \ \tau_kp(k) = p(k) \pm 2.
\end{equation*}
One easily checks
\begin{equation}\label{e:locmax<>min}
	p\in\mathbb P_s \ \ \text{and}\ \ k\in E^\pm_p \quad\Rightarrow\quad \tau_kp\in\mathbb P_s \ \ \text{and}\ \ k\in E^\mp_{\tau_kp}.
\end{equation}
Set also
\begin{equation*}
	\boc_{p}^{\pm} = \left\{s_{p,k}: 1\leq k\leq N, \ \ k\in E^\pm_p\right\} \quad\text{and}\quad \boc_p = \boc^-_p\cup\boc^+_p.
\end{equation*}
In light of \eqref{e:supspspk}, this can be rewritten as 
\begin{equation}\label{e:defboc}
	s'\in \boc^\pm_p \quad\Leftrightarrow\quad p(\supp(s')-1)= p(\supp(s')+1)=p(\supp(s'))\pm 1 = \spar(s')\pm 1.
\end{equation}
Note
\begin{equation*}
	\{\supp(s'): s'\in\boc_p^\pm\} = E_p^\pm. 
\end{equation*}
Using this and \eqref{e:locmax<>min}, it is easily seen that
\begin{equation}\label{uplowcorner}
	s'\in \boc_{p}^+\quad\Leftrightarrow\quad 1+s'\in\boc_{\tau_{\supp(s')}p}^-.
\end{equation}
Clearly, the knowledge of either $\boc_p^{+}$ or $\boc_p^{-}$ is sufficient to reconstruct an element $p\in \mathbb P_{s}$. One easily checks using \eqref{boundp'} that 
\begin{equation}\label{boundp-}
	s'\in\boc^-_p \text{ and } k=\supp(s') \ \Rightarrow\ \max\{2j-k,k+2i\}< p(k)
\end{equation}
while
\begin{equation}\label{boundp+}
	s'\in\boc^+_p \text{ and } k=\supp(s') \ \Rightarrow\ p(k)< \min\{k+2j,2h+2i-k\}.
\end{equation}

\begin{lem}\label{l:ladpmax}
	If $p$ is an MY path and $\bs s=(s_1,\dots,s_l)\in\lad$, then $|\{k: s_k\in\boc^\pm_p\}|\le 1$. 
\end{lem}

\begin{proof}
	Suppose $1\le k_1<k_2\le l$ are such that $s_{k_1},s_{k_2}\in\boc_p^\pm$.  This means $s_{k_j}=s_{p,m_j}$ for some $m_j\in[1,N], j\in\{1,2\}$, and then \eqref{e:supspspk} implies 
	\begin{equation*}
		p(\supp(s_{k_j})) = \spar(s_{k_j}), \ \ j\in\{1,2\}. 
	\end{equation*}
	Since $p$ is an MY path, this, together with \eqref{boundp}, implies 
	\begin{equation*}
		|\spar(s_{k_1})-\spar(s_{k_2})| \leq |\supp(s_{k_1})-\supp(s_{k_2})|,
	\end{equation*}
	yielding a contradiction with \eqref{e:defladsusp}. 
\end{proof}

Let $p_{i,j}, p_{i,j}^* \in\mathbb P_{i,j}$ be defined by 
\begin{equation}\label{e:pp*mod}
	p_{i,j}(k) = i+j+|k+i-j| \quad\text{and}\quad p^*_{i,j}(k) = i+j+h - |h+i-j-k|.
\end{equation}	
This can be rewritten as
\begin{equation}\label{e:pp*mods}
	p_s(k) = \spar(s)+|k-\supp(s)| \quad\text{and}\quad p^*_s(k) = \spar(s)+h - |k-\supp(s)^*|.
\end{equation}	
In particular,
\begin{equation}\label{uniquecorner}
	\boc_{p}^- = \emptyset  \Leftrightarrow p =p_s \quad  {\rm and} \quad \boc_{p}^+= \emptyset   \Leftrightarrow p = p^*_s.
\end{equation}
Moreover,
\begin{equation}\label{e:pathlub}
	p_s(k)\leq p(k)\leq p^*_s(k) \quad\text{for all}\quad 1\leq k\leq h, \ \ p\in\mathbb P_s,
\end{equation}
and
\begin{equation}\label{e:pathbound}
	p \in \mathbb P_s\setminus\{p_s,p_s^*\} \ {\rm and} \ s'\in \boc_p^{\pm} \quad\Rightarrow\quad \spar(s)<\spar(s')<\spar({}^*s).
\end{equation}
 The elements $p_s$ and $p_s^*$ are referred to a the highest and lowest path of $\mathbb P_s$, respectively, in \cite{MY12}. 
Finally, set
\begin{equation*}
	\boc_{s}^{\pm} = \bigcup_{p\in \mathbb P_{s}}\boc_p^{\pm}. 
\end{equation*}
One easily checks
	\begin{equation*}
		\boc^-_s\cup\boc^+_s\ne\emptyset \quad\Leftrightarrow\quad s\in\tilde\seg. 
\end{equation*}

\begin{lem}\label{l:mysmallerspar}
	Let $p$ be an MY path and $s\in\seg\setminus \boc_p$. If $p(\supp(s))\le \spar(s)$,
	\begin{equation}\label{e:mysmallerspar}
		p(\supp (s)) - \supp(s)<p(0),  \quad\text{and}\quad p(\supp (s)) - \supp(s)^*<p(h),
	\end{equation}
	there exists $s'\in\boc^+_p$ such that $\spar(s')<\spar(s)$.
\end{lem}

\begin{proof}
	To shorten notation, let $i=\supp(s)$. The assumptions in \eqref{e:mysmallerspar} imply $i\in[1,N]$. Since $s\notin\boc_p$, we have $p(i+1) - p(i-1)=\pm 2$. 
	Suppose it is $2$ and let $1\le r\le i$ be such that
	\begin{equation*}
		p(i-r) = p(i)-r.
	\end{equation*}
	The assumption $p(i) - i<p(0)$ implies $i-r> 0$. Hence, $i-r\in E^+_p$ and,
	letting $s'\in\seg$ be determined by $\supp(s') = i-r$ and $\spar(s')=p(i-r)$, it follows that $s'\in\boc^+_p$ and $\spar(s')<p(i)\le \spar(s)$. 
	
	If $p(i+1) - p(i-1)=- 2$, let $1\le r\le i^*$ be such that
	\begin{equation*}
		p(i+r) = p(i)-r.
	\end{equation*}
	The assumption $p(i) - i^*<p(h)$ implies $i+r<h$. Hence, $i+r\in E^+_p$ and,
	letting $s'\in\seg$ be determined by $\supp(s') = i+r$ and $\spar(s')=p(i+r)$, it follows that $s'\in\boc^+_p$ and $\spar(s')<p(i)\le \spar(s)$. 
\end{proof}

\begin{lem}[{\cite[Section 6.4]{MY12}}]\label{l:peakjb} 
	If $s_1\rhd s_2$, there exists unique  $p= p_{s_2}^{s_1}\in \mathbb P_{s_2}$ such that $\boc_p^- = \{s_1\}$. Moreover, in that case, $\boc_p^+ = \{s_1\cup s_2,s_1\cap s_2\}\cap \tilde \seg$.\qed
\end{lem}

\begin{lem}\label{ellwttau}
	If $s_1,s_2\in\seg$, then $s_1\in\boc_{s_2}^-  $ if and only if $s_1\rhd s_2$. 
\end{lem}

\begin{proof}
	If $s_1\rhd s_2$, \cref{l:peakjb} implies $s_1\in \boc_{s_2}^-$.
	
	For the converse, letting $s=s_2=[i_2,j_2]$ in \eqref{boundp'} and taking $p\in\mathbb P_s$ and $s'=s_1=[i_1,j_1]$ in \eqref{boundp-}, since $\supp(s_1)=j_1-i_1$ and $p(\supp(s_1))=\spar(s_1)=i_1+j_1$, it follows that 
	\begin{equation*}
		\max\{2j_2 - j_1+i_1,\ j_1-i_1+2i_2\}< p(j_1-i_1)=i_1+j_1.
	\end{equation*}
	Hence, $i_2<i_1$ and $j_2<j_1$. Working with the second inequality in \eqref{boundp'}, we see
	\begin{equation*}
		j_1+i_1= p(j_1-i_1)\leq \min \{j_1-i_1+2j_2,\ 2h+2i_2 -j_1+i_1\}
	\end{equation*}
	and, therefore, $i_1\leq j_2$ and $j_1-i_2\leq h$, which shows $s_1\rhd s_2$ as desired.
\end{proof}

We remark the following consequences of the above two lemmas:
\begin{equation}\label{e:pathw!max}
	\{p\in\mathbb P_{s}: |\boc^-_p|=1\} = \{p^{s'}_s: s'\rhd s\},
\end{equation}
and, if $s=[i,j]$, then  
\begin{equation}\label{alluplowcor}
	\boc_s^+ =\{[m,n]: i\leq m<j\leq n<h+i\} \ \ {\rm and} \ \ \boc_s^- =\{[m,n]: i< m\leq j< n\leq h+i\}.
\end{equation}
One easily checks the latter set is exactly $\{s'\in \tilde \seg: s'\rhd s\}$ and, hence, equal to $\boc^-_s$ by \cref{ellwttau}. The above characterization of $\boc^+_s$ then easily follows from \eqref{uplowcorner}.

\begin{lem}\label{lowtoup}
		If $s=[i,j]\in \tilde\seg$, $p\in \mathbb P_{s}$, and $s'=[m,n]\in\boc_{p}^-\setminus\{{}^*s\}$, there exists $s''\in \boc^+_p$ such that $\spar(s'')=\spar(s')-|\supp(s'')-\supp(s')|$. More precisely, if $m<j$, $s''=[m,n']$ for some $m<n'<n$, while $s''= [m',n]$ for some $n-h<m'<m$ if $n<h+i$. 
	\end{lem}
	
	\begin{proof}
		The assumption $s\in\tilde\seg$ implies that either $m<j$ or $n<h+i$ and, since $[m,n]\in\boc^-_p$, we have $p(\supp([m,n])) = \spar([m,n])$. Suppose  $m<j$ and note
	\begin{equation*}
		\spar([m,n])-\supp([m,n]) = (m+n) - (n-m)= 2m < 2j = p(0).
	\end{equation*}
	This implies 
	\begin{equation*}
		E_p^+\cap [1,n-m-1]\ne\emptyset.
	\end{equation*}
	Indeed, if this were not true, it would follow that
	\begin{equation}\label{e:lowtoup}
		p(a) =  p(\supp([m,n])) - (\supp([m,n])-a) = 2m+a \quad\text{for all}\quad a\in[0,m-n],
	\end{equation}
	yielding a contradiction when we plug $a=0$. Letting $k$ be the maximal element in the above set, it follows that \eqref{e:lowtoup} holds for all $a\in [k,m-n]$ and, therefore,
	\begin{equation*}
		p(k) = 2m+k = \spar([m,m+k]).	
	\end{equation*} 
	In other words, $s_{p,k} = [m,m+k]\in\boc_p^+$ and $n':=m+k$ satisfies the desired properties.
	
	Similarly, if $n<h+i$ we have
	\begin{equation*}
		\spar([m,n])- (h-\supp([m,n])) = 2n-h<h+2i = p(h)
	\end{equation*}
	and one concludes that there exists $\supp([m,n])<k<h$ such that 
	\begin{equation*}
		k\in E_p^+ \ \ {\rm and} \ \ p(k) = p(\supp([m,n]))- (k-\supp([m,n])) = 2n-k.
	\end{equation*}
	Therefore, $s_{p,k} = [n-k,n]\in\boc_p^+$ and $m':=n-k$ satisfies the desired properties.
\end{proof}

\begin{lem}\label{l:valleyjb}
	If $s_1\in\seg$  and $s_2\in \boc_{s_1}^+$, there exists a unique $p\in\mathbb P_{s_1}$ such that $\boc_p^+=\{s_2\}$. Moreover, in that case,  $\boc_{p}^- = \{s_2\cap {}^*  s_1,s_2\cup {}^* s_1\}\cap \tilde\seg   =\{s_1\lrcorner s_2,{}^*(s_1\llcorner s_2)\}\cap \tilde\seg$.
\end{lem}

\begin{proof}
	Uniqueness is clear since $p$ is determined by $\boc^\pm_p$. 
	Note \eqref{uplowcorner} implies $s_2+1\in \boc_{s_1}^-$ which, together with \cref{ellwttau}, implies $(s_2+1)\rhd s_1$. Writing $s_1=[i_1,j_1]$ and $s_2= [i_2,j_2]$, the latter implies $i_1\leq i_2<j_1\leq j_2$.
	Let $p:[0,h]\to \mathbb Z$ be defined by
	$$p(k)=\begin{cases}
		2j_1 + k,& 0\leq k\leq j_2-j_1,\\ 
		2j_2 -k, & j_2-j_1<k\leq j_2-i_2,\\
		2i_2+k, & j_2-i_2<k\leq h+i_1-i_2,\\
		2h+2i_1 -k, & h+i_1-i_2 < k\leq h.
	\end{cases}$$
	It is straightforward to check that $p\in\mathbb P_{s_1}$. Moreover, $E^+_p=\{j_2-i_2\}$ and $p(j_2-i_2)=j_2+i_2$, which implies $\boc^+_p = \{s_2\}$. In order to check the claim about $\boc^-_p$, note
	\begin{equation*}
		E^-_p = \{j_2-j_1,h+i_1-i_2\}\cap[1,N], \ \ p(j_2-j_1) = j_1+j_2, \ \ \text{and}\ \ p(h+i_1-i_2) = h+i_1+i_2.
	\end{equation*}
	If $j_1<j_2$, then $[j_1,j_2]=s_1\lrcorner s_2\in\boc^-_p$, while $[i_2,h+i_1] = {}^*(s_1\llcorner s_2)\in\boc^-_p$ provided $i_1<i_2$.  Since $(s_2+1)\rhd s_1$, in particular, $j_2+1-i_1\le h$ and, hence, \begin{equation*}
		[j_1,j_2] = s_2\cap{}^*s_1 \ \ \text{and}\ \ [i_2,h+i_1] = s_2\cup{}^*s_1.
	\end{equation*} 
\end{proof}

\begin{lem}\label{l:p*(dom)}
	If $s\rhd s'$, then $p_s^*(\supp({}^*\!s'))\ge \spar({}^*\!s')+2$ and $p^*_{s'}(\supp({}^*\!s))\le \spar({}^*\!s)-2$. 
\end{lem}

\begin{proof}
 Write $s=[i,j]$ and $s'=[i',j']$, so ${}^*s' = [j,i+h], \supp({}^*\!s') = h+i'-j'$, and $\spar({}^*\!s') = i'+j'+h=h+\spar(s')$. Then, since $i'<i\le j'<j$, we get
	\begin{align*}
		p_s^*(\supp({}^*\!s')) & \overset{\eqref{e:pp*mod}}{=} h+i+j - |(h+i-j)-(h+i'-j')|\\
		& \ge h + \min\{2i-i'+j',2j-j'+i'\}\ge h+i'+j'+2,
	\end{align*}
as desired. Similarly,
	\begin{align*}
		p_{s'}^*(\supp({}^*\!s))&\overset{\eqref{e:pp*mod}}{=} h+i'+j' - |(h+i-j)-(h+i'-j')|\\
		&\le h + \max\{2i'-i+j,2j'-j+i\}\le h+i+j -2 = \spar(^*\!s)-2.
	\end{align*}
\end{proof}

\subsection{On Tuples of MY Paths} 
We now collect some combinatorial facts about certain tuples of MY paths. 	
Given $\bs s = (s_1,\dots,s_l)\in\mseg$, let $\mathbb P_{\bs s}$ be the following subset of $\mathbb P_{s_1}\times\cdots\times \mathbb P_{s_l}$. An element $\bs p = (p_1,\dots,p_l)\in \mathbb P_{\bs s}$ if and only if
\begin{equation}\label{pbomega}
	p_m(k)<p_n(k) \quad\text{for all}\quad  k\in[0,h],\   1\le m<n\le l.
\end{equation}  
The following is easily established. 

\begin{lem}\label{weights1} 
	Let $\bs s$ be a ladder of length $l$ and $0\le m< n\le l$.
	\begin{enumerate}[(a)]
		\item If $(p_1,\dots,p_l)\in\mathbb P_{\bs s}$, then $(p_{m+1},\dots,p_n)\in \mathbb P_{\bs s(m,n)}$.
		\item If $(p_{m+1},\cdots, p_n)\in\mathbb  P_{\bs s(m,n)}$, then $(p_{s_1},\dots, p_{s_m}, p_{m+1},\dots, p_n, p_{s_{n+1}}^*,\dots, p_{s_l}^*)\in\mathbb P_{\bs s}$. \qed
	\end{enumerate}
\end{lem}

The next lemma gives a characterization of $\mathbb P_{\bs s}$ in terms of local maxima.

\begin{lem}\label{l:nonol}
	Let $\bs s = (s_1,\dots, s_l)$ be a ladder and let $\bs p = (p_1,\cdots, p_l)\in \mathbb P_{s_1}\times\cdots\times\mathbb P_{s_l}$. Then $\bs p \in \mathbb P_{\bs s}$ if and only if
	\begin{equation}\label{lowcornotov}
		p_{m+1}(\supp(s))>p_m(\supp(s)) \quad\text{for all}\quad s\in \boc_{p_m}^-, \ \ 1\leq m<l.
	\end{equation}
\end{lem}

\begin{proof}
	If $\bs p\in \mathbb P_{\bs s}$ then \eqref{lowcornotov} holds by definition. 
	Consider 
	\begin{equation*}
		S=\{(m,n): 1\le m<n\le l \text{ and }\exists\ k\in[0,h] \text{ such that } p_m(k)\ge p_n(k)\}.
	\end{equation*}
	If $\bs p\notin \mathbb P_{\bs s}$, then $S\ne\emptyset$ and we need to show
	\begin{equation}\label{e:nonol}
		\exists\ 1\le m<l \ \ \text{and}\ \ s\in\boc_{p_m}^- \ \ \text{with}\ \ p_m(\supp(s))\ge p_{m+1}(\supp(s)). 
	\end{equation}
	Choose $(m,n)\in S$ such that
	\begin{equation}\label{e:nonolm}
		n-m =\min\{n'-m': (m',n')\in S\}.
	\end{equation}
	The minimality of $n-m$ implies $(p_{m+1},\dots,p_n)\in\mathbb P_{\bs s(m,n)}$. It then follows from \cref{weights1}(b) that  $p_{s_m}(k)<p_n(k)$ for all $k\in [0,h]$ and, since $(m,n)\in S$, we conclude $p_m\neq  p_{s_m}$. An application of \eqref{uniquecorner} then implies $\boc_{p_m}^-\neq \emptyset$.  We claim 
	\begin{equation}\label{e:nonols}
		 \exists\ s\in \boc_{p_m}^- \ \ \text{such that}\ \ p_{m}(\supp(s))\geq p_{n}(\supp(s)).
	\end{equation} 
	If $n=m+1$, this implies \eqref{e:nonol}, thus completing the proof. Let us check we cannot have $n>m+1$. Indeed, in that case, \eqref{e:nonolm} implies $(m,m+1)\notin S$ and, hence,  $p_m(\supp(s))<p_{m+1}(\supp(s))$. But then, 
	\begin{equation*}
		p_n(\supp(s))\leq p_m(\supp(s))<p_{m+1}(\supp(s)),
	\end{equation*}
	showing $(m+1,n)\in S$ and yielding a contradiction with  \eqref{e:nonolm}.
		
To prove the claim, let $k\in [0,h]$ such that $p_m(k)\geq p_n(k)$. Note that, since $\bs s\in\lad$, 
	\begin{equation}\label{e:nonol0h}
		p_m(0)<p_n(0) \quad\text{and}\quad p_m(h)<p_n(h),
	\end{equation}
	so $k\in[1,N]$. If $k=\supp(s)$ for some $s\in\boc_{p_m}^-$ there is nothing to prove. Otherwise, i.e., if $k$ is not a point of local maximum for $p_m$, there exists $r\in \mathbb Z\setminus\{0\}$ such that $|r|$ is maximal with respect to the following property:
	\begin{equation*}
		k+r\in [0,h] \quad\text{and}\quad p_m(k+r)=p_m(k)+|r|.
	\end{equation*}
	In particular,  $p_m(k+r) \geq p_n(k+r)$ and \eqref{e:nonol0h} implies $k+ r \in [1,N]$. The maximality of $|r|$ then implies
	\begin{equation*}
		p_m(k+r\pm 1) = p_m(k+r) -1,
	\end{equation*}
	thus showing $s= s_{p_m, k+ r}\in \boc_{p_m}^-$ and completing the proof of \eqref{e:nonols}.
\end{proof}

\begin{lem}\label{l:pjbial}
	If $\bs s,\bs s' \in\mathbb L^l$ and $\bs s\rhd \bs s'$, then $(p_{s'_1}^{s_1},\dots,p_{s'_l}^{s_l})\in\mathbb P_{\bs s'}$. 
\end{lem}

\begin{proof}
	Write $s_k = [i_k,j_k], 1\le k\le l$ and, for ease of notation set, $p_k = p_{s'_k}^{s_k}$. By definition of $p_k$ (\cref{l:peakjb}),  $\boc_{p_k}^- = \{s_k\}$ and, hence, the last equality of \eqref{e:defboc} implies
	\begin{equation*}
		p_k(j_k-i_k)=j_k+i_k  \quad\text{for all}\quad 1\le k\le l.
	\end{equation*} 
	By Lemma \ref{l:nonol}, it suffices to check that 
	\begin{equation*}
		p_{k+1}(j_k-i_k)>j_k+i_k \quad\text{for all}\quad 1\le k<l.
	\end{equation*}
	Suppose this were not the case for some $k$. Since $s_{k+1}\in \boc_{p_{k+1}}^-$, it would follow that
	\begin{align*}
		p_{k+1}(j_{k+1}-i_{k+1}) - p_{k+1}(j_k-i_k) & \geq  (j_{k+1}+ i_{k+1}) -(j_k +i_k) \\
			& = (j_{k+1}- j_{k}) +(i_{k+1} -i_k) >0,
	\end{align*}
	where the last inequality follows since $\bs s$ is a ladder. On the other hand, \eqref{boundp} implies
	\begin{equation*}
		|p_{k+1}(j_{k+1}-i_{k+1}) - p_{k+1}(j_k-i_k)|\leq |(j_{k+1}-i_{k+1}) - (j_k -i_k)|.
	\end{equation*}
	One easily checks the combination of these two inequalities contradicts the assumption that $\bs s$ is a ladder. 
\end{proof}

\begin{lem}\label{l:pjbialn}
    Let $\bs s\in \seg^l$ and $\bs p \in \mathbb P_{s_1}\times \cdots \times \mathbb P_{s_l}$. If one of the following conditions are satisfied, then $\bs p\notin\mathbb P_{\bs s}$.
    \begin{enumerate}[(a)]
		\item There exist $1\leq k<l$ and $s\in\seg$ such that $p_k = p_{s_k}^s$, $p_{k+1}=p_{s_{k+1}}$ and $\spar(s)-\spar(s_{k+1})\ge |\supp(s)-\supp(s_{k+1})|$. 
		\item There exist $1\le k<m\le l$ and $(s,s')\in \mathbb L$ such that $p_{k}=p_{s_k}^{s'}$ and $p_m = p_{s_m}^{s}$.
	\end{enumerate}
\end{lem}

\begin{proof}
	For part (a), since $s\in\boc^-_{p_k}$ by \cref{l:peakjb}, we have $\spar(s) \overset{\eqref{e:supspspk}}{=} p_k(\supp(s))$. An application of \eqref{pbomega}  then gives
	\begin{align*}
		\spar(s) < p_{k+1}(\supp(s)) \overset{\eqref{e:pp*mods}}{=} \spar(s_{k+1})+|\supp(s)-\supp(s_{k+1})|,
	\end{align*}
	yielding a contradiction. 
	
	For part (b), \eqref{pbomega}  would imply $p_{k}(\supp(s))<p_{m}(\supp(s))=\spar(s)$ and, hence,
	\begin{equation*}
		p_{k}(\supp(s'))-p_{k}(\supp(s)) = \spar(s') - p_{k}(\supp(s))  > \spar(s')-\spar(s).
	\end{equation*}
	Since $(s,s')\in\lad$, it then follows from \eqref{e:defladsusp} that
	\begin{equation*}
		p_{k}(\supp(s'))-p_{k}(\supp(s))> |\supp(s')-\supp(s)|,
	\end{equation*} 
	contradicting \eqref{boundp}. 
\end{proof}

\subsection{MY Paths and $\ell$-weights}\label{ss:myweight}
Given $s\in \seg$ and $p\in \mathbb P_s$, define
\begin{equation}\label{omegapdef}
	\bs\omega(p) = \left(\prod_{s'\in \bs{\rm c}_p^+}\bs\omega_{s'}\right)\left( \prod_{s'\in \bs{\rm c}_p^-}\bs\omega_{s'}^{-1}\right)\in \cal P.
\end{equation}
In particular, 	
\begin{equation}\label{e:lwmainpaths}
	\bs\omega(p_{s})=\bs\omega_{s}, \quad \bs\omega(p_{s}^*)=({}^*\bs\omega_{s})^{-1},\quad \text{and}\quad \bs\omega(p_s^{s'}) = \bs\omega_{s'}^{-1}\bs\omega_{s\diamond s'}.
\end{equation}
Given $\bs s= (s_1,\dots, s_l)\in\seg^l$ and $\bs p = (p_1,\dots, p_l)\in \mathbb P_{s_1}\times \cdots\times\mathbb P_{s_l}$, define 
\begin{equation}\label{e:omegatuplepaths}
	\bs\omega(\bs p)= \prod_{k=1}^l \bs\omega(p_k).
\end{equation}

\begin{rem}\label{r:reduced}
	If $\bs s$ is a ladder and $\bs p\in \mathbb P_{\bs s}$, the conditions in \eqref{pbomega} guarantee that the expression for $\bs\omega(\bs p)$ given by \eqref{omegapdef} and \eqref{e:omegatuplepaths} is a reduced word on the fundamental $\ell$-weights regarded as free generators of $\mathcal P$.  Moreover, if $\bomega_{i,j}^{a}$ occurs in $\bomega(\bs p)$ for some $a\in \mathbb Z\setminus\{0\}$, then $a\in \{-1,1\}$.\endd
\end{rem}

Let $\bs s\in\lad^l$. If $\bs s$ is not connected, there exists $1\leq k< l$ minimal such $\bs s(k-1,k+1)$ is not connected, i.e., $s_{k+1}\ntriangleright s_k$. Let $\bs s_1 = \bs s(0,k)$, which is a connected ladder, and note $\bs s = \bs s_1\vee \bs s (k,l)$. Iterating this process, we obtain a sequence of connected ladders $\bs s_1, \dots, \bs s_m$,  such that 
\begin{equation*}
	\bs s = \bs s_1\vee\cdots \vee \bs s_m. 
\end{equation*}
The next result was proved in \cite{MY12a,MY12}.

\begin{prop}\label{mysnake}  
	Let $\bs s\in\lad^l$. 
	\begin{enumerate}[(a)]
		\item For all $\bs\pi\in\mathcal P$, we have $\dim V(\bs\omega_{\bs s })_{\bs\pi}\le 1$. Moreover, $\wt_\ell (V(\bs\omega_{\bs s})) =\{ \bs\omega(\bs p): \bs p\in\mathbb P_{\bs s}\}$, $\wt_\ell^+ (V(\bs\omega_{\bs s}))=\{\bs\omega_{\bs s}\}$, and $\wt_\ell^-( V(\bs\omega_{\bs s}))=\{({}^*\bs\omega_{\bs s})^{-1}\}$.
		\item If $\bs p, \bs p'\in \mathbb P_{\bs s}$, then   
		\begin{equation*}
			\bs\omega(\bs p')\le \bs\omega(\bs p)  \quad\Leftrightarrow\quad  p_r(k)\le  p_r'(k) \ \ {\rm{for\ all }} 		\ \ r\in[1,l], \ k\in [0,h].
		\end{equation*}
		\item For $\bs s'\in\lad$  such that $\bs s \vee\bs s'\in\lad$ and we have that $V(\bs\omega_{\bs s})\otimes V(\bs\omega_{\bs s'})$ is reducible if and only if $s_1'\rhd s_l$. \qed
	\end{enumerate}
\end{prop}	

Suppose $s_2\rhd s_1$, let $\bs s =(s_1,s_2)$, and recall the definition of $p_{s_1}^{s_2}$ from Lemma \ref{l:peakjb}. Comparing with \eqref{e:tpfundnhw}, one easily checks we have
a short exact sequence
\begin{equation}\label{e:tpfundpath}
	0\to V(\bs\omega(p_{s_1}^{s_2})\, \bs\omega_{s_2}) \to V(\bs\omega_{s_2})\otimes V(\bs\omega_{s_1})\to V(\bs\omega_{\bs s})\to 0. 
\end{equation}
More generally, we have the following. 

\begin{prop}\label{l:sessnake}
	If $\bs s\in\lad^l$ with $l>1$, then $V(\bs\omega_{s_l})\otimes V(\bs\omega_{\bs s(0,l-1)})$ is reducible if and only if $s_l\rhd s_{l-1}$. 
	Moreover, if $s_l\rhd s_{l-1}$, there exists a short exact sequence
	\begin{equation}\label{e:sessnake}
		0\to V(\bs\omega_{\bs s'}) \to  V(\bs\omega_{s_l})\otimes V(\bs\omega_{\bs s(0,l-1)})\to V(\bs\omega_{\bs s})\to 0,
	\end{equation}
	where $\bs s'=\bs s(0,l-2)\vee (s_l\diamond s_{l-1})$. 
\end{prop}

\begin{proof}
	The first statement is immediate from \cite[Lemma 6.2.1]{Naoi24}. Since $\bs s\in\lad$, \cref{p:weylper} implies the tensor product in \eqref{e:sessnake} is highest-$\ell$-weight. Thus, if $s_l\rhd s_{l-1}$, it is also reducible and \cite[Proposition 4.7]{KKOP} implies it has length two. Therefore, there exists a short exact sequence 
	\begin{equation*}
		0\to V(\bs\omega) \to  V(\bs\omega_{s_l})\otimes V(\bs\omega_{\bs s(0,l-1)})\to V(\bs\omega_{\bs s})\to 0,
	\end{equation*}
	for some $\bs\omega\in \cal P^+$ and we are left to show $\bs\omega =\bs\omega_{\bs s'}$. In particular, 
	\begin{equation*}
		\Hom(W(\bs\omega), V(\bs\omega_{s_l})\otimes V(\bs\omega_{\bs s(0,l-1)}))\ne 0
	\end{equation*}
	and, hence, \Cref{l:existhwv}, together with \cref{mysnake}(a), implies
	\begin{equation*}
		\bs\omega = \bs\omega_{s_l}\,\bs\omega(\bs g) \ \ {\rm for\ some} \ \ \bs g\in \mathbb P_{\bs s(0,l-1)}.
	\end{equation*}
	It follows from Remark \ref{r:reduced} and \eqref{uniquecorner} that there exists a unique $1\leq k\leq l-1$ such that 
	\begin{equation*}
		\boc_{g_k}^-=\{s_l\}
	\end{equation*}
	and, moreover,
	\begin{equation*}
	g_m=p_{s_m} \quad\text{for}\quad m\neq k.
	\end{equation*}
	We need to show $k=l-1$, which is clear if $l=2$. Otherwise, if it were $k<l-1$, then $g_{l-1}=p_{l-1}$ and 
	\begin{equation*}
		\spar(s_l)=p_{s_l}(\supp(s_l))\overset{\eqref{pbomega}}{>}p_{s_{l-1}}(\supp(s_l))=g_{l-1}(\supp(s_l))\overset{\eqref{pbomega}}{>}g_k(\supp(s_l))=\spar(s_l),
	\end{equation*}
	yielding a contradiction. 
\end{proof}

\section{Proof of Theorem \ref{t:soc} and Proposition \ref{p:tpsocmap}}\label{s:soc}

In this section we fix $\bs s,\bs s'$ as in Theorem \ref{t:soc}.  Set also
\begin{equation*}
	l=\ell(\bs s)=\ell(\bs s') \quad\text{and}\quad \bs\pi = \bs\omega_{\bs s\diamond\bs s'}.
\end{equation*}

\subsection{Proof of \Cref{p:tpsocmap}}\label{ss:tpsocmap}
 We begin by showing that the general case of \Cref{p:tpsocmap} follows assuming we have proved it if one of the following holds:
\begin{enumerate}[(i)]
	\item\label{i:tpsocmapi} $\supp(s_k\cap s'_k)=0$, $1\leq k\leq l$, 
	\item\label{i:tpsocmapii} $\supp(s_k\cup s'_k)= h$, $1\leq k\leq l$.
\end{enumerate}
Indeed, by \eqref{e:capcuplad} and \eqref{e:cornlad},  $(\bs s'\lrcorner\bs s, \bs s\cap\bs s')$ is a pair of ladders and one easily checks that $\bs s'\lrcorner\bs s\rhd\bs s\cap\bs s'$ and (i) holds. Similarly, $(\bs s\cup\bs s',(\bs s'\lrcorner\bs s)^*)$ is a pair of ladders such that $\bs s\cup\bs s'\rhd (\bs s'\lrcorner\bs s)^*$ and (ii) holds (recall \eqref{e:dualint}). Thus, applications of \Cref{p:tpsocmap} for each of these pairs shows, respectively, that there exist nonzero maps 
\begin{equation*}
	V(\bs\omega_{\bs s\cap \bs s'})\otimes V(\bs\omega_{\bs s\lrcorner\bs s'})\to V(\bs\omega_{\bs s}) \quad\text{and}\quad  V(\bs\omega_{\bs s\lrcorner\bs s'}^*)\otimes V(\bs\omega_{\bs s\cup \bs s'})\to V(\bs\omega_{\bs s'}), 
\end{equation*}
thus generalizing \eqref{e:tplrsim}. An application of \eqref{e:frobrec} shows there also exist nonzero maps
\begin{equation*}
	V(\bs\omega_{\bs s\lrcorner\bs s'})\to V({}^*\bs\omega_{\bs s\cap \bs s'})\otimes V(\bs\omega_{\bs s})\quad\text{and}\quad V(\bs\omega_{\bs s\cup \bs s'})\otimes V({}^*\bs\omega_{\bs s'})\to V(\bs\omega_{\bs s\lrcorner\bs s'}),
\end{equation*}
while \Cref{l:subsandq} then gives nonzero maps
\begin{equation*}
	V(\bs\omega_{\bs s})\otimes V({}^*\bs\omega_{\bs s\cap \bs s'}) \to V(\bs\omega_{\bs s\lrcorner\bs s'}) \quad\text{and}\quad V(\bs\omega_{\bs s\lrcorner\bs s'})\to V({}^*\bs\omega_{\bs s'})\otimes  V(\bs\omega_{\bs s\cup \bs s'}).
\end{equation*}
By composing them, we get a nonzero map
\begin{equation*}
	V(\bs\omega_{\bs s})\otimes V({}^*\bs\omega_{\bs s\cap \bs s'})\to V({}^*\bs\omega_{\bs s'})\otimes  V(\bs\omega_{\bs s\cup \bs s'}),
\end{equation*}
and another application of \eqref{e:frobrec} completes the proof of \Cref{p:tpsocmap} in general.

\subsubsection{} Let us prove  \Cref{p:tpsocmap} assuming that \eqref{i:tpsocmapi} holds, i.e., $\bs\omega_{\bs s \cap \bs s'}=\bold 1$ and, hence, $\bs\pi =\bs\omega_{\bs s\diamond\bs s'}= \bs\omega_{\bs s\cup \bs s'}$ and we need to show there exists a nonzero map
\begin{equation}\label{p:head}
	V(\bs\omega_{\bs s'})\otimes V(\bs\omega_{\bs s})\to V(\bs \pi).
\end{equation}
Note the assumption in (i) is equivalent to saying that $|s_k\cap s'_k|=1$ which, together with the assumptions that $\bs s\in\lad$ and $\bs s\rhd \bs s'$, implies
\begin{equation}\label{farappart}
	s'_k\cap s_m =\emptyset \quad\text{for all}\quad 1\le k<m\le l.
\end{equation} 

We proceed by induction on $l\geq 1$, noting that \eqref{e:tpfundnhw} establishes the first step of induction. 
For $l>1$, we claim there exists a nonzero map 
\begin{equation}\label{e:midstepmap}
	V(\bs\omega_{s_l\cup s'_l}\,\bs\omega_{\bs s'(0,l-1)})\otimes V(\bs\omega_{\bs s(0,l-1)})\to V(\bs\pi).
\end{equation}
Assuming this, we complete the proof of \eqref{p:head} as follows. Let 
\begin{equation*}
	V =  V(\bs\omega_{\bs s'})\otimes V(\bs\omega_{s_l}).
\end{equation*}
Since $\bs s'\vee s_l$ is a ladder and $s_l\rhd s'_l$, \Cref{l:sessnake}, together with \cref{l:subsandq}, implies there exists a surjective map  
\begin{equation*}
	V \to V(\bs\omega_{s_l\cup s'_l}\,\bs\omega_{\bs s'(0,l-1)})
\end{equation*}
and, therefore, there also exists a surjective map  
\begin{equation*}
	V\otimes V(\bs\omega_{\bs s(0,l-1)})\to V(\bs\omega_{s_l\cup s'_l}\,\bs\omega_{\bs s'(0,l-1)})\otimes V(\bs\omega_{\bs s(0,l-1)}).
\end{equation*}
Together with \eqref{e:midstepmap}, this implies we have a non-zero map 
\begin{equation}\label{e:maxsep1}
	V\otimes V(\bs\omega_{\bs s(0,l-1)})\to V(\bs\pi).
\end{equation}
On the other hand, if $s_l\rhd s_{l-1}$, \Cref{l:sessnake} gives a short exact sequence 
\begin{equation}\label{e:maxsep1o}
	0\to  V(\bs\omega_{\bs s(0,l-2)}\,\bs\omega)\to V(\bs\omega_{s_l})\otimes V(\bs\omega_{\bs s(0,l-1)})\to  V(\bs\omega_{\bs s}) \quad\text{with}\quad \bs\omega = \bs\omega_{s_l\diamond s_{l-1}}.
\end{equation}
If $s_l\ntriangleright s_{l-1}$, we have  $V(\bs\omega_{s_l})\otimes V(\bs\omega_{\bs s(0,l-1)})\cong  V(\bs\omega_{\bs s})$ by Proposition \ref{mysnake}(c) and, hence, \eqref{e:maxsep1o} remains valid by setting $V(\bs\omega_{\bs s(0,l-2)}\,\bs\omega)=0$. 
To shorten notation, set $\bs s^\diamond= \bs s(0,l-2)\vee (s_{l-1}\diamond s_l)$ so that \begin{equation*}
	\bs\omega_{\bs s(0,l-2)}\,\bs\omega = \bs\omega_{\bs s^\diamond}.
\end{equation*}
It follows that there also exists a short exact sequence
\begin{equation}\label{e:maxsep2}
	0\to V(\bs\omega_{\bs s'})\otimes V(\bs\omega_{\bs s^\diamond})\to V\otimes V(\bs\omega_{\bs s(0,l-1)})\to V(\bs\omega_{\bs s'})\otimes V(\bs\omega_{\bs s})\to 0.
\end{equation}

In light of \eqref{e:maxsep1} and \eqref{e:maxsep2}, \eqref{p:head} follows if we show 
\begin{equation*}
	\Hom (V(\bs\omega_{\bs s'})\otimes V(\bs\omega_{\bs s^\diamond}), V(\bs\pi))=0,
\end{equation*}
or, equivalently,
\begin{equation*}
	\Hom (W(\bs\pi), V(\bs\omega_{\bs s^\diamond})\otimes V(\bs\omega_{\bs s'}))=0.
\end{equation*}
Evidently, there is nothing to do if $V(\bs\omega_{\bs s^\diamond}) = 0$, so assume this is not the case. 
To prove the latter, note \Cref{l:existhwv} implies it suffices to show that 
\begin{equation}\label{e:nohwvms}
	\bs\pi\notin \bs\omega_{\bs s^\diamond} \wt_\ell V(\bs\omega_{\bs s'}).
\end{equation}
We prove this by induction on $l\ge 2$. To seek a contradiction, suppose $\bs g\in \mathbb P_{\bs s'}$ is such that 
\begin{equation}\label{e:nonohwvms}
	\bs\pi=\bs\omega_{\bs s^\diamond}\,\bs\omega(\bs g)
\end{equation}
and note that, since $\bs s,\bs s'\in\lad$ and $\bs s\rhd\bs s'$, $\bs\pi$ and $\bs\omega_{\bs s^\diamond}$ are in the subgroup of $\mathcal P$ generated by 
\begin{equation*}
	\bs\omega_r \quad\text{with}\quad \spar(r)>\spar(s'_1). 
\end{equation*}
In particular, \eqref{e:nonohwvms} implies we cannot have $g_1=p_{s'_1}$ and, hence, there exists $r\in \boc_{g_1}^-$. This implies $\bs\omega_{r}^{-1}$ occurs in the reduced expression of $\bs\omega(\bs g)$ and it follows  from \eqref{e:nonohwvms} that 
\begin{equation*}
	r\in \{s_1,\dots,s_{l-2}, s_{l-1}\cap s_{l}, s_{l-1}\cup s_l\}.
\end{equation*}
Since $r\in \boc_{g_1}^-\subseteq\mathbb P_{s_1'}$, \cref{ellwttau} implies $r\rhd s_1'$. Therefore, one of the following must be true:
\begin{equation*}
	(s_{l-1}\cap s_{l})\rhd s_1', \quad (s_{l-1}\cup s_{l})\rhd s_1', \quad  s_k\rhd s_1' \ \ \text{for some}\ \ 1\le k\le l-2.
\end{equation*}
Suppose first that $l>2$ and note \eqref{farappart} then implies 
\begin{equation*}
	(s_{l-1}\cap s_{l})\cap s_1'= (s_{l-1}\cup s_{l})\cap s_1'= s_k\cap s_1'=\emptyset \ \ \text{for all}\ \ k>1.
\end{equation*}
Hence, $r=s_1$ and $\boc^-_{g_1}=\{s_1\}$. An application of \Cref{l:peakjb}, together with \eqref{omegapdef}, then implies $g_1 = p_{s_1'}^{s_1}$ and  $\bs\omega(g_1)=\bs\omega_{s_1\cup s_1'}\,\bs\omega_{s_1}^{-1}$. Plugging this back in \eqref{e:nonohwvms}, we get
\begin{equation*}
	\bs\omega_{(\bs s\cup\bs s')(1,l)} = \bs \pi\bs\omega_{s_1\cup s_1'}^{-1} =  \bs\omega_{\bs s(1,l-2)}\,\bs\omega\, \bs\omega(g_2)\cdots \bs\bomega(g_l).
\end{equation*}
But this, together with inductive assumption on $l$, yields a contradiction with \eqref{e:nohwvms} applied to $\bs s(1,l)$ and $\bs s'(1,l)$ in place of $\bs s$ and $\bs s'$, respectively. 

If $l=2$,  since $s_1'\cap(s_1\cap s_2)=\emptyset$ as before, we must have $r=s_1\cup s_2$. An application of \Cref{l:peakjb} then implies $g_1 = p_{s_1'}^{s_1\cup s_2}$ and  $\bs\omega(g_1)=\bs\omega_{s_1'\cup s_1 \cup s_2 }\,\bs\omega_{s_1\cup s_2}^{-1}$. In particular $s_1'\cup s_1\cup s_2\in \seg$ and, hence, $s_1\cup s_1'$, $s_2\cup s_2'$, $s_1\cup s_2\in \tilde\seg$. Plugging this back in \eqref{e:nonohwvms}, we get
\begin{equation*}
	\bs\omega_{s_1\cup s_1'}\,\bs\omega_{s_2\cup s_2'} =  \bs\omega_{s_1\cap s_2}\,\bs\omega_{s_1 \cup s_2 \cup s_1'}\, \bs\omega(g_2).
\end{equation*}
Since $s_j \cup s_j'\in\tilde\seg$, $j=1,2$, it follows that $\bs\omega_{s_1\cup s_1'}\bs\omega_{s_2\cup s_2'}$ must occur in a reduced expression of $\bs\omega(g_2)$, yielding a contradiction with \cref{l:ladpmax} since $(s_1\cup s_1', s_2\cup s_2')$ is a ladder. This completes the proof of \eqref{p:head}.

\subsubsection{}
It remains to prove \eqref{e:midstepmap}. The induction assumption on $l$ allows us to use \eqref{p:head} with $\bs s(0,l-1),\bs s'(0,l-1)$ in place of $\bs s,\bs s'$, respectively, to obtain a nonzero map 
\begin{equation*}
	V(\bs\omega_{\bs s'(0,l-1)})\otimes V(\bs\omega_{\bs s(0,l-1)})\to V(\bs\omega_{(\bs s\cup\bs s')(0,l-1)}).
\end{equation*} 
Thus, letting 
\begin{equation*}
	W = V(\bs\omega_{s_l\cup s'_l})\otimes V(\bs\omega_{\bs s'(0,l-1)})\otimes V(\bs\omega_{\bs s(0,l-1)}),
\end{equation*}
it follows that there also exists a nonzero map
\begin{equation}\label{e:maxsep3}
	W\to V(\bs\omega_{\bs s\cup \bs s'})=V(\bs\pi).
\end{equation}
In particular, \eqref{e:midstepmap} follows if $V(\bs\omega_{s_l\cup s'_l})\otimes V(\bs\omega_{\bs s'(0,l-1)})$ is irreducible. Otherwise, noting that $\bos'(0,l-1)\vee (s_l\cup s_l')\in \mathbb L$, \Cref{l:sessnake}  implies that $(s_l\cup s'_l)\rhd s'_{l-1}$.  Letting 
\begin{equation*}
	\bs s'' = \bs s'(0,l-2)\vee (s'_{l-1}\diamond  (s_l\cup s'_l)), 
\end{equation*}
it follows from  \Cref{l:sessnake} that there exists a short exact sequence
\begin{equation}\label{e:maxsep4}
	0\to V(\bs\omega_{\bs s''})\otimes V(\bs\omega_{\bs s(0,l-1)}) \to W \to V(\bs\omega_{\bs s'(0,l-1)\vee (s_l\cup s'_l)})\otimes V(\bs\omega_{\bs s(0,l-1)})\to 0.
\end{equation}
As in the proof of \eqref{p:head} (cf. paragraph after \eqref{e:maxsep2}), in light of \eqref{e:maxsep3} and \eqref{e:maxsep4},  \eqref{e:midstepmap} follows if we show 
\begin{equation}\label{e:midstepmaphom}
	\Hom (W(\bs\pi),V(\bs\omega_{\bs s(0,l-1)})\otimes V(\bs\omega_{\bs s''}))=0.
\end{equation}

If this is not the case,  \Cref{l:existhwv}, together with \Cref{mysnake}(a), implies
\begin{equation}\label{e:nonohwvmsb}
	{}^*\bs\omega_{\bs s''}={}^*\bs\pi\,\bs\omega(\bs g) \ \ \text{for some}\ \ \bs g\in \mathbb P_{\bs s(0,l-1)}.
\end{equation}
One easily checks that 
\begin{equation*}
	\max\spar(\bs s'')\leq \max(\spar(s'_{l-1}\diamond  (s_l\cup s'_l)))< \max\spar(s_l\cup s'_l),
\end{equation*}
which implies $({}^*\bomega_{s_l\cup s'_l})^{-1}$ must occur in $\bomega(\bs g)$. In other words, writing $\bs g= (g_1,\cdots, g_{l-1})$, we must have 
\begin{equation*}
	{}^*(s_l\cup s'_l)\in \boc_{g_m}^{-1}\ \ \text{for some}\ \ 1\leq m\leq l-1.
\end{equation*} 
\cref{ellwttau} then implies  ${}^*(s_l\cup s'_l)\rhd s_m$ for some $1\leq m\leq l-1$. Writing $s_k=[i_k,j_k]$ and $s'_k = [i'_k,j'_k]$, we have
\begin{equation*}
	{}^*(s_l\cup s'_l) = [j_l,h+i'_l]
\end{equation*} 
and, therefore, we must have $j_l\le j_m$ for some $m<l$, yielding a contradiction since $\bs s\in\lad$. 
This completes the proof of \Cref{p:tpsocmap} in the case that \eqref{i:tpsocmapi} holds.

\subsubsection{}
Finally, assume \eqref{i:tpsocmapii} holds, i.e., $\bs\omega_{\bs s \cup \bs s'}=\bold 1$ and, hence, $\bs\pi =\bs\omega_{\bs s\diamond\bs s'}= \bs\omega_{\bs s\cap \bs s'}$. 
Therefore, we need to check \eqref{p:head} holds for this $\bs\pi$. 
One easily checks ${}^*\bs s$ and ${}^*\bs s'$ satisfy \eqref{i:tpsocmapi}. Since \Cref{p:tpsocmap} has been proved in this case, we have an epimorphism
\begin{equation}\label{e:tpsocmapii}
	V(\bs\omega_{{}^*\bs s'})\otimes V(\bs\omega_{{}^*\bs s}) \to V(\bs\omega_{{}^*\bs s\cup {}^*\bs s'}).
\end{equation}
Noting that  $({}^* \bs s\cup {}^*\bs s')^* = \bs s\cap \bs s'$, it follows from \eqref{e:dualtp} and \eqref{e:tpsocmapii} that we have a monomorphism
\begin{equation*}
	V(\bs\omega_{\bs s\cap \bs s'})\to V(\bs\omega_{\bs s})\otimes V(\bs\omega_{\bs s'}).
\end{equation*}
An application of \Cref{l:subsandq} completes the prof of \eqref{p:head} in this case and, therefore, also that of \Cref{p:tpsocmap}.

\subsection{Proof of \Cref{t:soc}}\label{ss:pfsoc} 
Comparing with \eqref{e:tpfundpath} and recalling \Cref{l:pjbial}, note
\begin{equation}\label{e:mypsoc}
	\bs\pi = \bs\omega_{\bs s}\, \bs\omega(\bs p) \quad\text{with}\quad \bs p=(p^{s_1}_{s'_1},\dots, p^{s_l}_{s'_l})\in\mathbb P_{\bs s'}. 
\end{equation}
To shorten notation, set $p_r = p^{s_r}_{s'_r}$.  In particular, recalling \cref{l:peakjb}, we have
\begin{equation}\label{lowcorsoc}
	\boc_{p_r}^- =\{s_r\}, \ \ 1\leq r\leq l.
\end{equation}
We claim
\begin{equation}\label{e:soc'}
	\Hom(W(\bs\varpi),V(\bomega_{\bs s})\otimes V(\bomega_{\bs s'}))\neq 0 \ \text{ with }\ \bs\varpi \leq \bs\pi \quad\Rightarrow\quad \bs\varpi = \bs\pi.
\end{equation}
Assuming this claim, Proposition \ref{p:tpsocmap}, together with \Cref{l:subsandq}, gives a non-zero map 
\begin{equation*}
	V(\bs\pi)\to V(\bs\omega_{\bs s})\otimes V(\bs\omega_{\bs s'}).
\end{equation*}
Since $V(\bs\omega_{\bs s})$ is real by \cite{KKKO15},  it follows from \Cref{simsoc} that 
$V(\bs\omega_{\bs s})\otimes V(\bs\omega_{\bs s'})$ has simple socle and 
\begin{equation*}
	\soc (V(\bs\omega_{\bs s})\otimes V(\bs\omega_{\bs s'}))\cong  {\rm hd}(V(\bs\omega_{\bs s'})\otimes V(\bs\omega_{\bs s})),
\end{equation*}
which completes the proof.

To prove \eqref{e:soc'}, note its first assumption, together with \Cref{l:existhwv}, implies 
\begin{equation}\label{e:soc''}
	\bs\varpi =\bs\omega_{\bs s}\, \bs\omega(\bs g) \quad\text{for some}\quad  \bs g\in \mathbb P_{\bs s'},
\end{equation}
while, since $\bs p \in \mathbb P_{\bs s'}$, we have $\bs\omega(\bs g)\leq \bs\omega(\bs p)$ by the second assumption. 
Together with \Cref{mysnake}(b), this implies 
\begin{equation}\label{e:varp<p}
	p_r(k)\leq  g_r(k), \ \ 1\leq r\leq l, \ k\in [0,h].
\end{equation}
We shall proceed by induction on $l$ to prove that 
\begin{equation}\label{e:socfc}
	\bs\omega_{\bs s}\, \bs\omega(\bs g) \in\mathcal P^+, \quad \bs\omega(\bs g)\le\bs\omega(\bs p) \qquad\Rightarrow\qquad \bs g=\bs p,
\end{equation}
which completes the proof. 

Since $s_r\in\bs{\rm c}^-_{p_r}$, \eqref{e:varp<p} implies $g_r\ne p_{s'_r}$ and it follows from \eqref{uniquecorner} that
\begin{equation}\label{e:musthavemax}
	\bs{\rm c}^-_{g_r}\ne\emptyset \quad\text{for all}\quad 1\le r\le l.
\end{equation}
By the first assumption in \eqref{e:socfc} and \Cref{r:reduced}, it follows that 
\begin{equation*}
	\bigcup_{r=1}^l \bs{\rm c}_{g_r}^-\subseteq \{s_1,\dots, s_l\}.
\end{equation*}
Moreover, if $s_m\in  \boc_{g_r}^-$ for some $r$, the last equality in \eqref{e:defboc} and \eqref{lowcorsoc} imply
\begin{equation*}
	g_r(\supp(s_m))= \spar(s_m) = p_m(\sup(s_m)).
\end{equation*}
In particular, for $r=l$, we must have $m=l$ since, otherwise, \eqref{e:varp<p} would imply
\begin{equation*}
	g_l(\supp(s_m))=p_m(\supp(s_m))\leq g_m(\supp(s_m)),
\end{equation*}
contradicting $\bs g\in \mathbb P_{\bs s'}$ (recall \eqref{pbomega}). But then, \eqref{e:musthavemax} implies $\bs{\rm c}^-_{g_l}=\{s_l\}$ and, hence, $g_l=p_l$ by \cref{l:peakjb}. In particular, induction starts when $l=1$. If $l>1$, setting $\bs p'= (p_1,\cdots, p_{l-1})$ and $\bs g'=(g_1,\cdots, g_{l-1})$, one easily checks 
\begin{equation*}
	\bs\omega_{\bs s(0,l-1)}\,\bs\omega(\bs g') \in\mathcal P^+ \quad\text{and}\quad \bs\omega(\bs g')\le\bs\omega(\bs p').
\end{equation*}
Since $\bs p',\bs g'\in \mathbb P_{\bs s'(0,l-1)}$ by \cref{weights1}, the induction hypothesis  implies $\bs g'=\bs p'$, thus completing the proof.

\section{Proof of \Cref{p:keystepgen}}\label{s:JHs}

\subsection{Proof of \eqref{e:hlwsubs}}\label{ss:keystepgen}
Let us start with the ``if'' part, i.e, 
\begin{equation}\label{e:colapontas}
	{\rm Hom} (W(\bs\pi_k),V) \ne 0  \quad\text{for all}\quad 0\le k\leq p. 
\end{equation}
For $p=0$ this follows from \Cref{t:soc}. Thus, assume $p>0$. Let us check that \eqref{e:colapontas} follows from an application of \Cref{p:pontas}. It follows from Proposition \ref{p:weylper} that we have epimorphisms
\begin{equation*}
	V(\bs\omega_{\bs s(l-k,l)})\otimes V(\bs\omega_{\bs s(0,l-k)}) \to V(\bs\omega_{\bs s}) \quad\text{and}\quad 	V(\bs\omega_{\bs s'(k,l)})\otimes V(\bs\omega_{\bs s'(0,k)}) \to V(\bs\omega_{\bs s'}),
\end{equation*}
thus showing the first two assumptions in \Cref{p:pontas} hold.  The third assumption follows by taking $\bs\omega= \bs\omega_{\bs s(0,l-k)\diamond \bs s'(k,l)}$, since \Cref{p:tpsocmap} implies
\begin{equation*}
	V(\bs\omega) = {\rm soc}\left( V(\bs\omega_{(\bs s(0,l-k)})\otimes V(\bs\omega_{(\bs s'(k,l)})\right). 
\end{equation*}
Noting that 
\begin{align*}
	\max \spar (\bs s'(0,l-k)) & \leq \min \spar(\bs s(0,l-k)\diamond \bs s'(k,l))\\
	& \leq \max \spar(\bs s(0,l-k)\diamond \bs s'(k,l))\leq \min \spar(\bs s(l-k,l))
\end{align*}
and using that $V(\bs\omega_{\bs s''})$ is real if $\bs s''\in \mathbb L$ (c.f. \cite[Theorem 3.4]{DLL19}), 
another application of \cref{p:weylper} implies the last three assumptions in \Cref{p:pontas} are satisfied. The conclusion of \Cref{p:pontas} then gives \eqref{e:colapontas}. 

Let us turn to the ``only if'' part. Thus, assume $\Hom(W(\bs\pi),V)\neq 0$ for some  $\bs\pi\in \wt_\ell^+W(\bs\pi_p)$. For $0\leq k\leq p$ let $\bs g^k \in \mathbb P_{\bs s'}$ be given by 
\begin{equation}\label{e:nottopfromm}
	g^k_m = p_{s'_m}, \ \ 1\leq m\leq k, \ \ \ g^k_m = p_{s'_m}^{s_{m-k}}, \ \ k<m\leq l.
\end{equation}
It follows from \cref{weights1} that $(g^k_{k+1},\cdots ,g^k_l)\in \mathbb P_{\bs s'(k,l)}$. Also (cf. \eqref{e:mypsoc}),  
\begin{equation*}
	\bs\pi_k = \bs\omega_{\bs s}\,\bs\omega(\bs g^k).
\end{equation*}
The assumption  $\bs\pi\in \wt_\ell^+W(\bs\pi_p)$ implies $\bs\pi\le\bs\pi_p$. This, together with the assumption $\Hom(W(\bs\pi),V)\neq 0$ alongside \Cref{l:existhwv}, implies
\begin{equation}\label{e:onlyif1}
	\bs\pi = \bs\omega_{\bs s}\, \bs\omega(\bs f) \quad\text{for some}\quad \bs f\in\mathbb P_{\bs s'} \quad\text{with}\quad \bs\omega(\bs f)\le \bs\omega(\bs g^p).
\end{equation}	
In light of \Cref{mysnake}(b), the latter condition is equivalent to
\begin{equation}
	g_m^p(j)\le f_m(j) \quad\text{for all}\quad j \in [0,h], m\in[1,l]. 
\end{equation}
Thus, showing that $\bs\pi=\bs\pi_k$ for some $k$ is equivalent to showing that  
\begin{equation}\label{e:keystepgen=>}
	\exists\ 0\le k\le p \text{ such that } f_m = p_{s_m'} \text{ if } 1\leq m\le k \text{ and } f_m = p_{s_m'}^{s_{m-k}} \text{ for } k<m\leq l. 
\end{equation}

Begin by noting that, if $m>p$,  \eqref{e:nottopfromm} alongside \eqref{e:pathlub} implies $f_m\ne p_{s'}$ and \eqref{uniquecorner} then implies ${\rm \bs c}^-_{f_m}\ne\emptyset$. More generally, if $\boc_{f_m}^-\neq \emptyset$, then \eqref{e:onlyif1} implies  
${\rm \bs c}^-_{f_m} \subseteq \{s_i : 1\leq i\leq l\}$ and it then follows from \cref{l:ladpmax} that
\begin{equation*}
	|{\rm \bs c}^-_{f_m}|\le 1 \quad\text{for all}\quad 1\le m\le l.
\end{equation*}
Thus, by \eqref{uniquecorner} and \eqref{e:pathw!max}, either $f_m = p_{s_m'}$ or $f_m = p_{s_m'}^{s_r}$ for some $1\le r\le l$. 
Moreover, an application of \Cref{l:pjbialn}(b) gives
\begin{equation}\label{e:maxareord}
	f_m = p_{s_m'}^{s_r} \ \ \text{and}\ \ f_n=p_{s_n'}^{s_t} \ \ \text{with}\ \  n>m \quad\Rightarrow\quad t>r.
\end{equation}
Let us check 
\begin{equation}\label{e:nogaplocmax}
	{\rm \bs c}^-_{f_m}\ne \emptyset \quad\Rightarrow\quad {\rm \bs c}^-_{f_{m+1}}\ne \emptyset, 
\end{equation}
which implies 
\begin{equation}\label{e:maxsbefores}
	s_r \in {\rm \bs c}^-_{f_m} \quad\Rightarrow\quad r\in\{1,\cdots, m\}.
\end{equation}
Note \eqref{e:nogaplocmax} is vacuous for $m\geq p$ since ${\rm \bs c}^-_{f_m}\ne \emptyset$ for all $m>p$. This latter fact, together with \eqref{e:maxareord}, also implies
\begin{equation*}
	s_n\in\boc^-_{f_m} \ \ \text{and}\ \ m<p \quad\Rightarrow\quad n\le p. 
\end{equation*}
Thus, fix $1\le m<p$ with $\boc_{f_m}^-=\{s_n\}$ for some $1\le n\le p$. 
The assumption $\bs s\rhd_p \bs s'$ implies $s_1\rhd s'_{p+1}$. In particular,
\begin{equation*}
	(s'_m,s'_{m+1},s'_{p+1},s_1,s_n)\in\lad.
\end{equation*}
An application of \eqref{e:defladsusp} to this ladder then says
\begin{equation*}
	\spar(s_n)-\spar(s'_{m+1})>|\supp(s_n)-\supp(s'_{m+1})|
\end{equation*}
and \Cref{l:pjbialn}(a) would lead to a contradiction if \eqref{e:nogaplocmax} failed.

With the above paragraph in mind, let us check \eqref{e:keystepgen=>} holds with $k$ defined by 
\begin{equation*}
	k = \min\{m: \bs{\rm c}^-_{f_m}\ne \emptyset\}-1.
\end{equation*}
If $k=0$, \eqref{e:maxsbefores} implies $\bs f = \bs g^0$, so \eqref{e:keystepgen=>} holds. Thus, assume $k>0$ and let $\phi:[k+1, l]\to [1,l]$ be such that 
\begin{equation*}
	\boc_{f_m}^-=\{s_{\phi(m)}\}\quad\text{for all}\quad k<m\le l.
\end{equation*}
Note that, for this choice of $k$, \eqref{e:keystepgen=>} is equivalent to
\begin{equation}\label{e:keystepgen=>phi'}
	\phi(k+1+m)=m+1  \quad\text{for all}\quad 0\le m< l-k.
\end{equation}
We will use a recursive procedure on $m$ for proving \eqref{e:keystepgen=>phi'}.
Set $\bs s_\phi = (s_{\phi(k+1)}, \cdots, s_{\phi(l)})$ and note $\bs s_\phi\in\lad$ since $\phi$ is strictly increasing by \eqref{e:maxareord}. As part of the inductive procedure, we shall prove
\begin{equation}\label{e:weighcomb}
	\bs\omega_{\bs s'(k+m,l)}= \bs\omega_{\bs s'(k+m,l)\diamond \bs s_\phi(m,l-k)}\,\bs\omega(\bs f'(m,l))^*\prod_{r\notin \Im\phi}\bs\omega_{s_r}, \ \ \  0\leq m<l-k,
\end{equation}
for some $\bs f'\in\mathbb P_{\bs s}$, where $\bs f'(m,l) = (f'_{m+1},\dots,f_l')$, as well as
\begin{equation}\label{e:claimn=1gen}
	\boc_{f_{m+1}'}^+ = \{{}^*s_{k+m+1}'\}.
\end{equation}
The flow of the recursive argument is:
\begin{equation}\label{e:theflow}
	\eqref{e:weighcomb} \text{ for } m \ \ \Rightarrow\ \ \eqref{e:claimn=1gen} \ \ \Rightarrow\ \ \eqref{e:keystepgen=>phi'} \ \ \Rightarrow\ \ \eqref{e:weighcomb} \text{ for } m +1,
\end{equation}
with the last implication valid provided $m<l-k-1$.

Let us start by checking that \eqref{e:weighcomb} holds for $m=0$. It follows from the description of $\bs f$ given in the previous paragraphs, together with \eqref{e:lwmainpaths}, that
\begin{equation*}
	\bs\omega(\bs f) = \bs\omega_{\bs s'(0,k)}\, \bs\omega_{\bs s_\phi}^{-1}\, \bs\omega_{\bs s'(k,l)\diamond \bs s_\phi}. 
\end{equation*}
Plugging this back in \eqref{e:onlyif1}, it follows that 
\begin{equation*}
	\bs\pi = \bs\omega_{\bs s'(0,k)}\,\bs\omega_{\bs s'(k,l)\diamond \bs s_\phi}\prod_{r\notin \Im\phi}\bs\omega_{s_r}.
\end{equation*}
Similarly to \eqref{e:onlyif1}, the second part of  \Cref{l:existhwv} implies 
\begin{equation*}
	\bs\omega_{\bs s'} = \bs\pi\, \bs\omega(\bs f')^* \quad\text{for some}\quad \bs f'\in\mathbb P_{\bs s}.
\end{equation*}	
Combining the above, we get
\begin{equation}\label{weighcomb}
	\bs\omega_{\bs s'(k,l)}= \bs\omega_{\bs s'(k,l)\diamond \bs s_\phi}\,\bs\omega(\bs f')^*\prod_{r\notin \Im\phi}\bs\omega_{s_r},
\end{equation}
which is \eqref{e:weighcomb} with $m=0$. Thus, assume henceforth that \eqref{e:weighcomb} is valid for all $0\le m\le n$ for some $0\le n< l-k$ and that \eqref{e:claimn=1gen} and  \eqref{e:keystepgen=>phi'} are valid for $0\le m<n$. In particular,
\begin{equation}\label{e:Imphiind}
	[1,n]\subseteq \Im\phi
\end{equation}

Before checking that \eqref{e:claimn=1gen} holds for $m=n$, let us assume this is true as well and check the last two implications in \eqref{e:theflow}. Note first  that \eqref{e:claimn=1gen}, together with Lemmas \ref{l:valleyjb} and \ref{l:propdiamond} (since $s_{n+1}\rhd s'_j, j=n+1,k+n+1$), implies  
\begin{equation}\label{e:c+fn+1}
	\boc_{f_{n+1}'}^- = \{{}^*s_{n+1}\cap {}^*(s_{k+n+1}'),\ {}^*s_{n+1}\cup {}^*(s_{k+n+1}')\}\subseteq \tilde\seg.
\end{equation}
In particular $\bs\omega_{s_{n+1}\cap s'_{k+n+1}}^{-1}$ occurs in $\bs\omega(\bs f'(n,l))^*$. Since $\spar(s_r)>\spar(s_{n+1}\cap s'_{k+n+1})$ for all $r>n$, \eqref{e:Imphiind} and \eqref{e:weighcomb} with $m=n$ imply this term must occur in $\bs\omega_{\bs s'(k+n,l)\diamond \bs s_\phi(n,l-k)}$. But
\begin{equation}\label{e:s'diamsphi}
	\begin{aligned}
		s'(k+n,l)\diamond \bs s_\phi(n,l-k)  = \, & (s'_{k+n+1}\cap s_{\phi(k+n+1)},\dots,s'_l\cap s_{\phi(l)})\\
		&\vee (s'_{k+n+1}\cup s_{\phi(k+n+1)},\dots,s'_l\cup s_{\phi(l)})
	\end{aligned}
\end{equation}
and, since $\bs s,\bs s'\in\lad$, we must have $s_{n+1}\cap s'_{k+n+1} =  s'_{k+n+1}\cap s_{\phi(k+n+1)}$ and $\phi(k+n+1)=n+1$. This proves \eqref{e:keystepgen=>phi'} holds for $m=n$. Plugging \eqref{e:keystepgen=>phi'} with $m=n$ back in \eqref{e:weighcomb} with $m=n$ and keeping \eqref{e:c+fn+1} as well as \eqref{e:s'diamsphi} in mind, we obtain \eqref{e:weighcomb} with $m=n+1$, thus proving the last implication.

It remains to check \eqref{e:claimn=1gen} holds for $m=n$. 
Since $k\leq p, n<l-k$, and $\bs s\rhd_p \bs s'$, we have $s_{n+1}\rhd s_{k+n+1}'$. In particular,  
\begin{equation*}
	\spar(s_r)>\spar(s_{k+n+1}') \ \ \text{for all}\ \ n< r\leq l,  
\end{equation*}
and 
\begin{equation*}
	\min\spar(\bs s'(k+n,l)\diamond \bs s_\phi(n,l-k))>\spar(s_{k+n+1}'), 
\end{equation*}
where we also used \eqref{e:nabig} for the latter. Therefore, since $\bs\omega_{s_{k+n+1}'}$ occurs on the left hand side of \eqref{e:weighcomb} with $m=n$, it must also occur in $\bs\omega(\bs f'(n,l))^*$ or, equivalently, $^*\!s_{k+n+1}'\in \boc_{f_t'}^+$, for some $n< t\leq l$. Note this implies  $\boc_{f_{t}'}^+=\{{}^*\!s_{k+n+1}'\}$ since, otherwise, \eqref{e:weighcomb}  would imply that ${}^*\!s_r'\in \boc_{f_{r+1}'}^+$ for some $k+n+1<r\leq l$, contradicting \cref{l:ladpmax}. Thus, checking \eqref{e:claimn=1gen} is equivalent to showing $t=n+1$. 

Suppose for a contradiction that $t>n+1$ and note that, since ${}^*\!s_{k+n+1}'\in \boc_{f_t'}^+$,  \eqref{uplowcorner} implies
\begin{equation*}
	1+{}^*\!s'_{k+n+1}\in\boc^-_{\tau_{\supp({}^*\!s'_{k+n+1})}f'_t}\subseteq \boc^-_{s_t}
\end{equation*}
and, hence, $1+{}^*\!s'_{k+n+1}\rhd s_t$ by \cref{l:peakjb}. On the other hand, since $s_{n+1}\rhd s_{k+n+1}'$, it follows from \eqref{e:dualcover} that  $1+{}^*\!s_{k+n+1}'\rhd s_{n+1}$ and an application of \cref{l:complete} to the ladder $(s_{n+1},s_t,1+{}^*\!s_{k+n+1}')$ then implies $s_t\rhd s_{n+1}$.  Let us check $f'_{n+1}\ne p^*_{s_{n+1}}$. 
Indeed, \cref{l:p*(dom)} implies 
\begin{equation*}
	p^*_{s_{n+1}}(\supp({}^*s_{k+n+1}'))> \spar({}^*s_{k+n+1}') = f_t(\supp({}^*s_{k+n+1}')),
\end{equation*}
where the last equality follows from \eqref{e:supspspk} using that ${}^*s_{k+n+1}'\in \boc_{f_t'}^+$ once more. Since we are assuming $t>n+1$, we reach a contradiction with \eqref{pbomega}. 

Since $f_{n+1}'\neq p_{s_{n+1}}^*$,  \eqref{uniquecorner} implies there exists $s\in \boc_{f_{n+1}'}^+$. It then follows that $\bs\omega_{s^*}$ occurs on the left-hand-side of \eqref{e:weighcomb} , i.e., $s = {}^*\!s_r'$ for some $k+n+1<r\leq l$. Therefore, using that ${}^*\!s_r'\in \boc_{f_{n+1}'}^+$, $(s_{n+k+1}',s_r')\in \mathbb L, {}^*\!s_{k+n+1}'\in \boc_{f_t'}^+$, and $f'_t$ is an MY path, respectively, we see
\begin{eqnarray*}
	f_{n+1}'(\supp({}^*\!s_r')) \overset{\eqref{e:supspspk}}{=} \spar({}^*\!s_r')&\overset{\eqref{e:defladsusp}}{>}& \spar({}^*\!s_{k+n+1}') + |\supp({}^*\!s_r') - \supp({}^*\!s_{k+n+1}')|\\ 
	&\overset{\eqref{e:supspspk}}{=}&f_t'({}^*\!s_{k+n+1}')+ |\supp({}^*\!s_r') - \supp({}^*\!s_{k+n+1}')|\\
	&\overset{\eqref{boundp}}{\geq} & f_t'(\supp({}^*\!s_{r}')),
\end{eqnarray*}
which yields a contradiction with \eqref{pbomega} once more if $t>n+1$. Therefore, $t=n+1$, thus completing the proof of  \eqref{e:claimn=1gen}.

\subsection{Proof of \Cref{p:keystepgen} - The inclusions}\label{ss:keystepgeninc} 

Setting $M_0={\rm soc}(V)$, the comment following \eqref{pikdef} and \Cref{t:soc} imply $M_0 \cong V(\bs\pi_0)$. Moreover, by the first part of the proposition, $V$ contains highest-$\ell$-weight submodules with highest-$\ell$ weight $\bs\pi_k$ for each $0\le k\le p$. In particular, since $M_0$ is contained in any submodule of $V$, it follows that we have an inclusion $M_0\subseteq M_1$ for any choice highest-$\ell$-weight submodule $M_1$ with highest-$\ell$ weight $\bs\pi_1$. If $p=1$, this completes the proof of the proposition.  Thus, assume $p>1$. 

Since $\bs s\rhd_k \bs s'$ for $0\leq k\leq p$, we have $\bs s(0,l-1)\rhd_{k}\bs s'(1,l)$ for all $0\leq k< p$. Set 
\begin{equation*}
	V' = V(\bs\omega_{\bs s(0,l-1)})\otimes V(\bs\omega_{\bs s'(1,l)}) 
\end{equation*}
and define $\bs\pi'_k$ by making $\bs s(0,l-1)$ and $\bs s'(1,l)$ play the roles of $\bs s$ and  $\bs s'$ in definition \eqref{pikdef}, respectively. Note
\begin{equation}\label{e:keystepgeninc}
	\bs\pi_{k+1} = \bs\pi'_k\, \bs\omega_{s_1'}\, \bs\omega_{s_l} \quad\text{for all}\quad 0\le k<p.
\end{equation}
An inductive argument on $p$ implies there exist submodules $M'_k$ of $V'$, $0\le k<p$, which are highest-$\ell$-weight with highest $\ell$-weight $\bs\pi_k'$ and such that $M'_k\subseteq M'_{k+1}$. On the other hand, there exists an epimorphism
\begin{equation*}
	f:V(\bs\omega_{s_l})\otimes V' \otimes V(\bs\omega_{s_1'}) \to V.
\end{equation*} 
Let $N_k = V(\bs\omega_{s_l})\otimes M_k' \otimes V(\bs\omega_{s_1'})$ and $M_{k+1} = f(N_k)$. Since $N_k$ is highest-$\ell$-weight by \Cref{p:weylper}, it follows that $M_{k+1}$ is also highest-$\ell$-weight and the highest $\ell$-weight is $\bs\pi_{k+1}$ by \eqref{e:keystepgeninc}. This completes the proof.

\section{Proof of \Cref{t:imagfromsnakegen}}\label{s:imagproof}

\subsection{The Diamond Weight Space is Thin}\label{ss:thinwsp} We shall need the following generalization of \cite[Lemma 4.3.2]{BC23} for proving \Cref{t:imagfromsnakegen}. Henceforth, we set $V= V(\bomega_{\bs s})\otimes V(\bomega_{\bs s'})$.

\begin{lem}\label{l:head}
	If $\bs s,\bs s'\in\mathbb L^l$ are such that $\bs s \rhd \bs s'$, then
	\begin{equation}\label{e:socwsdim}
		\dim(V)_{\bs\omega_{\bs s\diamond\bs s'}} = 1 = \dim(V\otimes V)_{(\bs\omega_{\bs s\diamond\bs s'})^2}. 
	\end{equation}
\end{lem}

\begin{proof} To shorten notation, let $\bs \pi = \bs\omega_{\bs s \diamond \bs s'}$. Let also $\bs p\in \mathbb P_{\bs s'}$ as in \eqref{e:mypsoc}, i.e., $\bs p= (p_{s_1'}^{s_1},\cdots, p_{s_l'}^{s_l})$.  By \Cref{mysnake}(a), $\bs\pi\in \wt_\ell V$ and, therefore,  the equalities in \eqref{e:socwsdim} follow if we prove $\dim V_{\bs\pi}\leq 1$ and $\dim (V\otimes V)_{\bs\pi^2}\leq 1$. 

Let $\bs g\in \mathbb P_{\bs s}$ and $\bs g'\in \mathbb P_{\bs s'}$ be such that $\bs\omega(\bs g)\bs\omega(\bs g')=\bs\pi$. Since $V(\bs\omega_{\bs s'})$ is thin (its $\ell$-weight spaces are one-dimensional) by \Cref{mysnake}(a), it suffices to show 
	\begin{equation}\label{e:head1stclaim}
		\bs\omega(\bs g)\bs\omega(\bs g')=\bs\pi\quad\Rightarrow\quad \bs\omega(\bs g) = \bs\omega_{\bs s}.
	\end{equation}
Noting that $\bs\pi$ is in the subgroup of $\mathcal P$ generated by $\bs\omega_{s}$ with $s\in \seg$ such that 
	\begin{equation*}
		\spar(s_1')< \spar(s) <  \spar(s_l),
	\end{equation*}
it follows that $\bs\omega_{s_1'}$ does not divide the polynomial $\bs\pi$. Then  $g_1'\neq p_{s_1'}$ and, hence, $\boc_{g_1'}^-\neq\emptyset$, by \eqref{uniquecorner}. We claim 
	\begin{equation}\label{s1lowcor}
		\boc_{g_1'}^- = \{s_1\},
	\end{equation} 
	which is equivalent to saying that $g_1' = p^{s_1}_{s_1'}$ by \cref{l:peakjb}. This claim  implies $g_1=p_{s_1}$ since, otherwise, by \eqref{e:subgpweight} $\bs\omega(\bs g)$ would be in the subgroup of $\cal P$ generated by $\bs\omega_{s}$ with $\spar(s)>\spar(s_1)$ and this would imply that $\bs\omega_{s_1}^{-1}$ occurs in $\bs\pi$, yielding a contradiction.  In particular, \eqref{e:head1stclaim} follows if $l=1$.  If $l>1$, letting 
	\begin{equation*}
		\bs{\tilde\pi} = \bs\omega_{\bs s(1,l)\diamond\bs s'(1,l)},
	\end{equation*}
	$\bs{\tilde g} = (g_2,\dots,g_l)$, and similarly for $\bs{\tilde g'}$, it follows that
	\begin{equation*}
		\bs{\tilde \pi} = \bs\omega(\bs{\tilde g}) \bs\omega(\bs{\tilde g'}),
	\end{equation*}
	and an obvious inductive argument on $l$ completes the proof.
	
	Before proceeding with the proof of \eqref{s1lowcor}, let us check
	\begin{equation}\label{e:cauppropinc+}
		s\in\{s_1\cap s_1',s_1\cup s_1'\}\cap \tilde\seg \quad\Rightarrow\quad s\in \boc^+_{g_1'}. 
	\end{equation}
	We write down the argument for $s_1\cap s_1'$. The other case can be proved similarly or by using this case and duality (cf. \eqref{e:concacudual}). In this case, $\bs 1\neq \bs\omega_{s_1\cap s_1'}$ divides $\bs\pi$ and, since  $\spar(s_1\cap s_1')<\spar(s_1)$, it follows that $\bs\omega_{s_1\cap s_1'}$ occurs in a reduced expression of $\bs\omega(\bs g')$. More precisely,  $s_1\cap s_1'\in \boc^+_{g_m'}$ for some $1\leq m\leq l$. Since $j_1'<j_m'$ for $m>1$ and $s_1\cap s_1'= [i_1,j_1']$, it easily follows from \eqref{alluplowcor} it must be $m=1$. Let us also note that
	\begin{gather}\notag 
			s_1\cap s_1'\in\tilde\seg \ \ \Rightarrow\ \ E_{g_1'}\subset [\supp(s_1\cap s_1'),N]\\ \label{Eg1'} \quad\text{while}\quad\\ \notag
			s_1\cup s_1'\in\tilde\seg \ \ \Rightarrow\ \ E_{g_1'}\subset [1,\supp(s_1\cup s_1')].
		\end{gather}	
		Indeed, the former follows since
		\begin{equation*}
			g_1'(0)-g_1'(\supp(s_1\cap s_1'))= j_1'- i_1 = \supp (s_1\cap s_1'), 
		\end{equation*}
		while the latter follows because
		\begin{equation*}
			g_1'(h)-g_1'(\supp(s_1\cup s_1'))= h+i_1'- j_1 = h-\supp (s_1\cup s_1'). 
		\end{equation*}
	
    We are ready to proceed with the proof of \eqref{s1lowcor}. 
    Assume first $s_1\cap s_1', s_1\cup s_1'\notin  \tilde\seg$, or, equivalently, $s_1 = {}^*\!s_1'$. Since $g_1'\neq p_{s_1'}$, there exists $s\in \boc_{g_1'}^-$. If $s\neq s_1$, then $g_1'\neq p_{s_1'}^*=p^{s_1}_{s_1'}$ and, hence, $\spar(s)<\spar(s_1)$ by \eqref{e:pathbound}. Therefore $\bs\omega_s^{-1}$ occurs in $\bs\omega(\bs g)\bs\omega(\bs g')$, yielding a contradiction. Hence, we must have $s=s_1$, thus completing the proof  \eqref{s1lowcor}.
	
	We write down the argument for the case $s_1\cap s_1'\in \tilde\seg$. 	The case $s_1\cup s_1'\in \tilde \seg$ analogous (and follows from this case by duality), so we omit details. It follows from  \eqref{e:cauppropinc+} and \eqref{Eg1'}, together with $\boc_{g_1'}^-\neq \emptyset$, that there exists $r\ge 1$ maximal such that 
	\begin{equation*}
		g_1'(\supp(s_1\cap s_1') + r) = \spar(s_1\cap s_1') + r.
	\end{equation*}
	The maximality of $r$ implies $\supp(s_1\cap s_1')+r\in E_{g_1'}^-$ and it follows from \eqref{spkint} that
	\begin{equation*}
		s:= [i_1,j_1'+r]\in \boc_{g_1'}^-.
	\end{equation*}
	Moreover, \eqref{alluplowcor} implies $j_1'+r\leq h+i_1'$. In particular, $\bs\omega_s^{-1}$ occurs in $\bs\omega(\bs g')$ and, hence, we must have $s\in \boc_{g_n}^+$ for some $1\leq n\leq l$. Since $i_n>i_1$  if $n>1$, a further application of \eqref{alluplowcor} implies $n=1$ as well as 
	\begin{equation}\label{boudnsg1'}
		j_1\leq j_1'+r\leq h+i_1'.
	\end{equation}
	
	If $s_1\cup s_1'\notin \tilde \seg$, then $j_1 = h+i_1'$ and \eqref{boudnsg1'} gives $s_1=s\in \boc_{g_1'}^-$. Since 
	\begin{equation*}
		g_1'(\supp(s_1))-g_1'(h) = i_1 -i_1' = h- \supp(s_1),
	\end{equation*}
	it follows that $E_{g_1'}^-\subseteq [1,\supp(s)]$ which,	along with \eqref{Eg1'}, implies $E_{g_1'}^- = \{\supp(s_1)\}$, thus proving \eqref{s1lowcor} in this case. Otherwise, if $s_1\cup s_1'\in \tilde \seg$, then $s_1\cup s_1'\in \boc_{g_1'}^+$ by \eqref{e:cauppropinc+}. Moreover, if $j_1<j_1'+r$ 
	we would have $s\rhd s_1\cup s_1'$, which contradicts \eqref{e:nocpinpath}.  Hence, $j_1=j_1'+r$ and we again conclude $s_1=s\in \boc_{g_1'}^-$. 	
	The discussion so far gives 
	\begin{equation*}
		s_1\cap s_1', s_1\cup s_1'\in \boc_{g_1'}^+ \quad {\rm and}\quad  s_1\in \boc_{g_1'}^-.
	\end{equation*}
	In light of \eqref{Eg1'}, \eqref{s1lowcor} follows after noting that 
	\begin{equation*}
		g_1'(\supp(s_1))-g_1'(\supp(s_1\cup s_1')) = i_1-i_1' =\supp(s_1\cup s_1') - \supp(s_1),
	\end{equation*}
	and
	\begin{equation*}
		g_1'(\supp(s_1))-g_1'(\supp(s_1\cap s_1')) = j_1-j_1'= \supp(s_1)-\supp(s_1\cap s_1').
	\end{equation*}

	To prove the second equality in \eqref{e:socwsdim}, we proceed in a similar fashion. Let  $\bs g, \tilde {\bs g}\in \mathbb P_{\bs s}$ and $\bs g', \tilde {\bs g}'\in \mathbb P_{\bs s'}$ such that 
\begin{equation}\label{e:pi2gg'tilgg'}
	\bs\omega(\bs g)\bs\omega(\tilde {\bs g})\bs\omega(\bs g')\bs\omega(\tilde {\bs g}') = \bs\pi^2.
\end{equation}
As before, we will prove
\begin{equation}\label{e:s1lowcor2}
	g_1'=\tilde g_1'= p^{s_1}_{s_1'},
\end{equation}

Assume first $s_1\cap s_1',s_1\cup s_1'\notin \tilde \seg$, i.e., $s_1= {}^*\!s_1'$, and note \eqref{e:s1lowcor2} is equivalent to $g_1'=\tilde g_1'= p^*_{s_1'}$.
If either $g_1'\neq p_{s_1'}^*$ or $\tilde {g_1}'\neq p_{s_1'}^*$, by \eqref{uniquecorner}, we can choose $s\in \boc_{g_1'}^+\cup\boc_{\tilde g_1'}^+$ such that 
\begin{equation*}
	\spar(s) = \min\{\spar(r): r\in \boc^+_{g_1'}\cup\boc^+_{\tilde g_1'}\}.
\end{equation*}
Without loss of generality, suppose $s\in \boc_{g_1'}^+$ and recall from \eqref{boundp+} that this implies
\begin{equation}\label{bounupcor} 
	\spar(s)-\supp(s)< g_1'(0) \quad\text{and} \quad \spar(s)-\supp(s)^*< g_1'(h).
\end{equation}
Note also that
\begin{equation*}
	s\notin \boc^-_{\tilde g_1'}
\end{equation*}
since, otherwise, \cref{lowtoup} would imply there exist $s'\in \boc_{\tilde g_1'}^+$ with $\spar(s')<\spar(s)$, contradicting our choice of $s$. Let us check that
\begin{equation}\label{g1'scond}
	\tilde g_1'(\supp(s))\ge \spar(s) \quad\text{and, moreover,}\quad  
	\tilde g_1'(\supp(s))= \spar(s) \ \ \Leftrightarrow\ \ s\in \boc_{\tilde g_1'}^+.
\end{equation}
To do this, it suffices to prove that we reach a contradiction by assuming 
\begin{equation*}
	\tilde g_1'(\supp(s))\leq \spar(s) \quad\text{and}\quad s\notin \boc^+_{\tilde g_1'}.
\end{equation*} Indeed,  since $\tilde g_1'(0) = g_1'(0)$ and $\tilde g_1'(h)=g_1'(h)$, the first of these assumptions and \eqref{bounupcor} imply
\begin{equation*}
	\tilde g_1'(\supp (s)) - \supp(s)^*<\tilde g_1'(h) \quad\text{and}\quad \tilde g_1'(\supp (s)) - \supp(s)<\tilde g_1'(0). 
\end{equation*}
The assumption $s\notin \boc^+_{\tilde g_1'}$ then implies we can use \cref{l:mysmallerspar} to conclude there exists $s'\in\boc^+_{\tilde g_1'}$ such that $\spar(s')<\spar(s)$, contradicting our choice of $s$.

Next, we show
\begin{equation}\label{notcorbel} 
	s \notin \{s_k\cup s_k', \ s_k\cap s_k': 2\leq k\leq l\}.
\end{equation}
Indeed, \eqref{alluplowcor} implies $s=[m,n]$ for some $i_1'\leq m<j_1'\leq n <h+i_1'$. On the other hand, since $s_1 = {}^*\!s_1'$, it follows that $j_1'<i_2$ and $j_2-i_1'>h$.   A straightforward inspection then proves \eqref{notcorbel}. 

It follows from \eqref{notcorbel} that $\bs\omega_s$ does not divide $\bs\pi$. 
Since $\spar(s)<\spar(s_1)$, to avoid contradiction we must have $s\in \boc_{\tilde g_k'}^-$ for some $1\leq k\leq l$. \cref{lowtoup} then implies there exists $s'\in \boc_{\tilde g_k'}^+$, such that 
\begin{equation}\label{ss'dist}\spar(s)-\spar(s') = |\supp(s)-\supp(s')|.
\end{equation}
Since $\tilde{\bs g}'\in \mathbb P_{\bs s'}$, it follows from \eqref{pbomega} that
\begin{equation*}
	\tilde g_1'(\supp (s'))\leq \tilde g_k'(\supp (s'))= \spar(s') = \spar(s) - |\supp(s)-\supp(s')|,
\end{equation*}
with equality holding only if $k=1$. Combined with \eqref{g1'scond}, unless $k=1$ and $\tilde g_1'(\supp (s)) = \spar(s)$, we get
\begin{equation*}
	\tilde g_1'(\supp(s))- \tilde g_1'(\supp(s'))> |\supp(s)-\supp(s')|,
\end{equation*}
which contradicts \eqref{boundp}. Hence, $k=1$ and $\tilde g_1'(\supp (s)) = \spar(s)$. Together with the last part of \eqref{g1'scond}, we conclude $s, s'\in \boc_{\tilde g_1'}^+$, which contradicts \eqref{ss'dist}, thus completing the proof in this case. 

	As in \eqref{e:cauppropinc+}, we conclude 
\begin{equation}\label{e:cauppropinc+2}
	s\in\{s_1\cap s_1',s_1\cup s_1'\}\cap\tilde\seg \ \ \Rightarrow\ \ s\in \boc_{g_1'}^+\cap \boc_{\tilde g_1'}^+.
\end{equation} 
As usual, we write down the argument for the case $s_1\cap s_1'\in \tilde\seg$. Let $r,\tilde r\geq 1$ be maximal such that 
$$s=[i_1,j_1'+r]\in \boc_{g_1'}^-, \ \ \ \tilde s=[i_1,j_1'+\tilde r]\in \boc_{\tilde g_1'}^-,$$
and assume without loss of generality that $r\leq \tilde r$. By \eqref{alluplowcor} we have 
$$j_1'+r \leq j_1'+\tilde r\leq h+i_1'.$$
An argument similar to the one leading to \eqref{boudnsg1'} shows that $s,\tilde s\in \boc_{g_1}^+\cup \boc_{\tilde g_1}^+$ and
\begin{equation}\label{boudnsg1'2}
	j_1\leq j_1'+r \leq j_1'+\tilde r\leq h+i_1'.
\end{equation}

If $s_1\cup s_1'\notin\tilde\seg$, then  $j_1'+r= j_1$ and we conclude $s=\tilde s=s_1$. The exact same argument used for the proof of \eqref{s1lowcor} in this case also completes the proof of \eqref{e:s1lowcor2}. Otherwise, if $s_1\cup s_1'\in\tilde\seg$, then 
\begin{equation*}
	s_1\cup s_1'\in \boc_{g_1'}^+\cap  \boc_{\tilde g_1'}^+,
\end{equation*}
by \eqref{e:cauppropinc+2} and the same argument  used for the proof of \eqref{s1lowcor} in this case also completes the proof of \eqref{e:s1lowcor2}.
\end{proof}

\subsection{Proof of \Cref{t:imagfromsnakegen}}\label{ss:imagfromsnakegen}

In the notation  of \Cref{p:keystepgen}, 
we have inclusions $V(\bs\pi_0)\cong M_0\hookrightarrow M_1\hookrightarrow V$ with $M_1$ being a quotient of $W(\bs\pi_1)$. 
By abuse of notation, we may write $V(\bs\pi_0)$ in place of $M_0$. 
Proceeding as in  \cite[Lemma 4.3.3]{BC23}, we conclude
\begin{equation}\label{e:m=1ses}
	M_1/V(\bs\pi_0) \cong V(\bs\pi_1).
\end{equation}
Indeed, if this were not the case, there should exist $\bs\omega\in\mathcal P^+$ with $\bs\omega<\bs\pi_1$ such that $V(\bs\omega)$ is a simple submodule of $M_1/V(\bs\pi_0)$. The first part of \eqref{e:socwsdim} implies $\bs\omega\ne\bs\pi_0$, which implies $0\ne \Hom(W(\bs\omega),M_1)\subseteq \Hom(W(\bs\omega),V)$, contradicting the ``only if'' part of \Cref{p:keystepgen}. 

It follows from \eqref{e:m=1ses} that 
\begin{equation}\label{e:m=12ftp}
	V(\bs\pi_1)^{\otimes 2}\cong M_1^{\otimes 2}/ \left(M_1\otimes V(\bs\pi_0) + V(\bs\pi_0)\otimes M_1\right). 
\end{equation}
Let 
\begin{equation*}
	V^\# = V(\bs\omega_{\bs s})\otimes V(\bs\pi_0)\otimes V(\bs\omega_{\bs s'}), \quad V' = V(\bs\omega_{\bs s'}) \otimes V(\bs\omega_{\bs s}),
\end{equation*}
and consider the map $\Phi:V\otimes V\to V^\#$ given by 
\begin{equation*}
	V\otimes V =  V(\bs\omega_{\bs s})\otimes V'\otimes V(\bs\omega_{\bs s'}) \stackrel{1\otimes \phi\otimes 1}{\longrightarrow} V^\#,
\end{equation*}
where $\phi$ is the projection $V'\to {\rm hd}(V')\cong V(\bs\pi_0)$. In order to prove \Cref{t:imagfromsnakegen}, we shall prove in \Cref{ss:extrasq} the following fact about $N:=\Phi(M_1\otimes M_1)$ (cf. \cite[Proposition 4.3.4]{BC23}).
\begin{equation}\label{e:extrasq}
	\Hom(N,V(\bs\omega))\ne 0 \quad\text{for some}\quad \bs\omega\ne \bs\pi_1^2. 
\end{equation}
For any $\bs\omega$ satisfying the first part of \eqref{e:extrasq}, there exists a nonzero map $f:N \to V(\bs\omega)$ and, hence, a nonzero map
\begin{equation*}
	M_1^{\otimes 2} \stackrel{\Phi}{\to} N \stackrel{f}{\to} V(\bs\omega).
\end{equation*}
If, in addition, 
\begin{equation}\label{e:fPhi=0}
	f\circ\Phi \left(L\right) = 0 \quad\text{with}\quad L:=M_1\otimes V(\bs\pi_0) + V(\bs\pi_0)\otimes M_1,
\end{equation}
it follows from \eqref{e:m=12ftp} that we have an induced nonzero map
\begin{equation*}
	V(\bs\pi_1)^{\otimes 2} \to V(\bs\omega).
\end{equation*}
Thus, if $\bs\omega\ne\bs\pi_1^2$, which is true if we choose $\bs\omega$ satisfying the second part of \eqref{e:extrasq} as well, it follows that  $V(\bs\pi_1)$ is imaginary, as desired. 

Before proceeding with the argument, let us remark the following case.
If $\bs s'=\bs s^*$ so that $\bs\pi_0=\bs 1$, analogous arguments to those used in \cite[Proposition 4.3.4]{BC23} imply that there exists $\bs\omega$ as in  \eqref{e:extrasq} such that  
\begin{equation}\label{e:extrasq'}
	\bs\omega\notin \left\{\bs 1,\, \bs\pi_1\right\}.
\end{equation} 
Thus, in this case, \eqref{e:fPhi=0} is clear from \eqref{e:m=1ses}, which finishes the proof of \Cref{t:imagfromsnakegen}. 

In general, i.e., if $\bs\omega$ satisfies the first part of \eqref{e:extrasq}, it follows from \Cref{p:bc1.3.3} that $({}^*\bs\pi_1)^{-1}\bs\pi_1\in \wt_\ell(V(\bs\omega))$. If \eqref{e:fPhi=0} were false, since $V(\bs\omega)$ is irreducible, we would have 
\begin{equation*}
	f\circ\Phi(L)=V(\bs\omega) \quad\text{and, hence}, \quad L_{({}^*\bs\pi_1)^{-1}\bs\pi_1}\ne 0.
\end{equation*}
Using that $V(\bs\pi_1)$ and $V(\bs\pi_0)$ are the simple factor of $M_1$ by \eqref{e:m=1ses}, we get a contradiction provided
\begin{equation}\label{e:novarpidvarpi}
	(V(\bs\pi_0)\otimes V(\bs\pi_0))_{({}^*\bs\pi_1)^{-1}\bs\pi_1} = 0 = (V(\bs\pi_0)\otimes V(\bs\pi_1))_{({}^*\bs\pi_1)^{-1}\bs\pi_1}.
\end{equation}
The first equality holds with no further assumption since  $\bs\omega_{s'_1}$ occurs in the reduced expression of $({}^*\bs\pi_1)^{-1}\bs\pi_1$ and any $\ell$-weight of $V(\bs\pi_0)$ is in the subgroup of $\mathcal P$ generated by $\bs\omega_{s}, s\in \tilde\seg$, $\spar(s)>\spar(s'_1)$, by \eqref{e:nabig} and \eqref{e:subgpweight}.

In the case that $\bs s' = \bs s^*$, proceeding as in \cite{BC23}, it can be shown that
\begin{equation}\label{e:novarpivarpiW}
	W(\bs\pi_0\bs\pi_1)_{({}^*\bs\pi_1)^{-1}\bs\pi_1}=0,
\end{equation}
which, in particular, implies the second part of \eqref{e:novarpidvarpi}. However, as seen in \cref{exweylfails} below, \eqref{e:novarpivarpiW} is not true in general. Still, we believe \Cref{t:imagfromsnakegen} is true without the assumption that \eqref{almostdual'} holds.  In particular, the theorem is true for \cref{exweylfails} as a consequence of the discussion in \cref{ss:l=2}. In the next sections, we prove  \eqref{e:novarpivarpiW} holds under the extra assumption \eqref{almostdual'} of \Cref{t:imagfromsnakegen}. 

\begin{ex}\label{exweylfails}
	Let $N=6$, $\bs s= ([1,4], [2,5])$ and $\bs s' = ([-1,2],[0,3])$, so that 
	\begin{equation*}
		\bs\pi_1 = \bs\omega_{-1,2}\, \bs\omega_{1,3}\,\bs\omega_{0,4}\, \bs\omega_{2,5}, \ \ 
		\bs\pi_0 = \bs\omega_{-1,4}\, \bs\omega_{1,2}\,\bs\omega_{0,5}\, \bs\omega_{2,3},
	\end{equation*}
	and
	\begin{equation*}
		({}^*\bs\pi_1)^{-1}\bs\pi_1 = \bs\omega_{-1,2}\, \bs\omega_{1,3}\,\bs\omega_{0,4}\, \bs\omega_{2,5}
		(\bs\omega_{2,6}\, \bs\omega_{3,8}\,\bs\omega_{4,7}\, \bs\omega_{5,9})^{-1}.
	\end{equation*}
	Then, we see $({}^*\bs\pi_1)^{-1}\bs\pi_1\in \wt_\ell W(\bs\pi_0\bs\pi_1)$ 	by noting that
	\begin{gather*}
		\bs\omega_{-1,2}\in \wt_{\ell} W(\bs\omega_{-1,2}),\ \, \bs\omega_{2,3}\in \wt_{\ell} W(\bs\omega_{2,3}), \\
		\bs\omega_{3,8}^{-1}\in \wt_\ell W(\bs\omega_{1,3}), \ \ \bs\omega_{4,7}^{-1}\in \wt_{\ell} W(\bs\omega_{0,4}), \ \ \bs\omega_{5,9}^{-1}\in \wt_{\ell} W(\bs\omega_{2,5}),\\
		\bs\omega_{0,4}\,\bs\omega_{0,6}^{-1}\in \wt_{\ell} W(\bs\omega_{-1,4}),\ \ \bs\omega_{2,3}^{-1}\,\bs\omega_{1,3}\in \wt_\ell W(\bs\omega_{1,2}), \ \ \bs\omega_{2,6}^{-1}\,\bs\omega_{2,5}\,\bs\omega_{0,6}\in \wt_{\ell} W(\bs\omega_{0,5}).
	\end{gather*}
	Indeed, the claims on the first line are obvious, the ones on the second line are of the form 
	\begin{equation*}
		(^*\!\bs\omega_s)^{-1} = \bs\omega(p_s^*) \in\wt_\ell(W(\bs\omega_s)),
	\end{equation*}
	 while the ones on the last line are of the form
	 \begin{equation*}
	 	\bs\omega_s^{-1}\, \bs\omega_{s\diamond s'}=\bs\omega(p^s_{s'})\in\wt_\ell(W(\bs\omega_{s'})).
	 \end{equation*}  \endd
\end{ex}

\subsection{Proof of \eqref{e:extrasq}}\label{ss:extrasq}

It follows from \eqref{e:socwsdim} and \eqref{e:m=1ses} that
\begin{equation}\label{e:socwsdim'}
	\dim(V)_{\bs\pi} = 1 = \dim(V^{\otimes 2})_{\bs\pi^2}  = \dim(M^{\otimes 2})_{\bs\pi^2}. 
\end{equation}
In particular, since $\Phi$ is surjective by definition, we have $\bs\pi^2\in\wt_\ell(N)$, so $N$ is not zero.
Thus, \eqref{e:extrasq} follows if we show 
\begin{equation}%\label{e:extrasq'}
	\Hom(N,V(\bs\pi_1^2))= 0.
\end{equation}
For this it suffices to check that 
\begin{equation}\label{e:novarpisq}
	\bs\pi_1^2\notin\wt_\ell(V^\#). 
\end{equation}
Let $\cal P_{s_1'}^+$ be the submonoid of $\mathcal P^+$ generated by $\bs\omega_{s}$ with $s\in \tilde\seg$, $\spar(s)>\spar(s_1')$, and let $\cal P_{s_1'}$ the subgroup of $\cal P$ generated by $\cal P_{s_1'}^+$. Recall from \eqref{pikdef} that
\begin{equation*}
	\bs\pi_1= \bs\omega_{s_1'}\bs\pi'
\end{equation*}
with $\bs\pi'\in \cal P_{s_1'}^+$, and that
\begin{equation*}
	V^\# = L\otimes V(\bs\omega_{\bs s'}) \quad\text{with}\quad L = V(\bs\omega_{\bs s})\otimes V(\bs\pi).
\end{equation*}
In particular, $\wt_\ell(L) \subset \mathcal P_{s_1'}$. Therefore, \eqref{e:novarpisq} follows if we show the reduced expression of any $\ell$-weight of $V(\bs\omega_{\bs s'})$ does not contain $(\bs\omega_{s_1'})^2$. But this is clear from \Cref{mysnake} together with \Cref{r:reduced}.

\subsection{A Crucial Factorization}\label{ss:crfact}

We continue assuming $\bs s,\bs s'$ are as in \cref{p:keystepgen}. Let $\cal P_{\le}$ be the subgroup of $\cal P$ generated by $\bs\omega_{s}$ for all $s\in\seg$ such that 
\begin{equation}\label{e:novarpivarpiW<}
	  \spar(s) \le \spar({}^*\!s_1')- |\supp(s) - \supp({}^*\!s_1')|
\end{equation}
and let $\cal P_{>}$ be the one generated by $\bs\omega_{s}$ for all $s\in\seg$ such that \eqref{e:novarpivarpiW<} fails. 

Clearly 
\begin{equation}\label{supPleq}
	\bs\omega_{s_1'}, \bs\omega_{^*\!s_1'}\in \cal P_{\leq} \ \ \ {\rm and} \ \ \  \bs\omega_{s}\in \cal P_{\leq }\ \Rightarrow\ \spar(s)\leq \spar(^*\!s_1').
\end{equation}

Write $\bs\pi_1\bs\pi_0 = \bs\omega_1\bs\omega_2$ with $\bs\omega_1\in\cal P_\leq$ and $\bs\omega_2\in \cal P_>$. The following fact about this factorization of $\bs\pi_1\bs\pi_0$ will be crucial for the proof of \eqref{e:novarpivarpiW}. Namely, if \eqref{almostdual} holds, then
\begin{equation}\label{e:smallpart}
	\bs\omega_1 = \bs\omega_{s_1'}\,\bs\omega_{s_1\diamond s_1'}\,\bs\omega_{s_1\diamond s_2'}.
\end{equation}
Since
\begin{equation}\label{e:pikdef01}
	\bs\pi_0 = \prod_{r=1}^l \bs\omega_{s_r\diamond s'_r} \quad\text{and}\quad \bs\pi_1 = \bs\omega_{s_1'}\bs\omega_{s_l}\prod_{r=1}^{l-1} \bs\omega_{s_r\diamond s'_{r+1}},
\end{equation}
by \eqref{pikdef}, the following lemma clearly implies \eqref{e:smallpart}.

\begin{lem}\label{Pleqfund}
	We have $\bs\omega_{s_1\diamond s'_1}, \bs\omega_{s_1\diamond s'_2} \in \cal P_{\leq}$. Moreover, if \eqref{almostdual} holds, then 
	\begin{equation*}
		\bs\omega_{s_l}, \bs\omega_{s_r\cap s'_r}, \bs\omega_{s_r\cup s'_r}, \bs\omega_{s_r\cap s'_{r+1}}, \bs\omega_{s_r\cup s'_{r+1}}\in \cal P_{>} \quad\text{for all}\quad 2\le r\le l.
	\end{equation*}
\end{lem}

\begin{proof}
	We begin by noting the second part of \eqref{e:pp*mods} can be rewritten as
	\begin{equation*}
		p^*_s(k) = \spar(^*\!s) - |k-\supp(^*\!s)| \quad\text{for all}\quad s\in\seg, k\in[0,h].
	\end{equation*}	
	In particular, 
	\begin{equation*}
		\spar(^*\!s_1') = p^*_{s_1'}(\supp(s)) + |\supp(s)-\supp(^*\!s_1')| \quad\text{for all}\quad s\in\seg.
	\end{equation*}
	On the other hand, if $s\in\boc^\pm_{s_1'}$, say  $s\in\boc^\pm_p$ with $p\in\mathbb P_{s_1'}$, \eqref{e:supspspk} implies $\spar(s) = p(\supp(s))$. 
	Hence, given $s\in\boc^\pm_{s_1'}$ and letting $p\in\mathbb P_{s_1'}$ be such that $s\in\boc^\pm_p$, we have 
	\begin{align*}
		\spar(s) - \spar(^*\!s'_1) & = p(\supp(s)) - p^*_{s_1'}(\supp(s))- |\supp(s)-\supp(^*\!s_1')|\\
		& \overset{\eqref{e:pathlub}}{\le} -|\supp(s)-\supp(^*\!s_1')|,
	\end{align*}
	thus showing that $\bs\omega_s\in\mathcal P_{\le}$. In particular,  by \cref{l:peakjb}, $\bs\omega_{s_1\cap s_1'}$ and  $\bs\omega_{s_1\cup s_1'}$ are in $\cal P_\le$. Turning to $\bs\omega_{s_1\diamond s'_2}$,  we need to show \eqref{e:novarpivarpiW<} is satisfied for $s\in\{s_1\cap s_2', s_1\cup s_2'\}$. For $s=s_1\cap s_2'=[i_1,j_2']$, \eqref{e:novarpivarpiW<} is the same as
	\begin{equation*}
		i_1+j_2' \le i_1'+j_1' + h - |j_2'-i_1+j_1'-i_1'-h|.
	\end{equation*}
	If the term within $|\dots|$ is nonnegative, one easily checks the above is equivalent to $j_2'-i_1'\le h$, which is true since $s_2'\rhd s_1'$. Otherwise, the above is equivalent to $i_1\le j_1'$, which holds since $s_1\rhd s_1'$. For $s=s_1\cup s_2' = [i_2',j_1]$, \eqref{e:novarpivarpiW<} is the same as
	\begin{equation*}
		i_2'+j_1' \le i_1'+j_1' + h - |j_1-i_2'+j_1'-i_1'-h|.
	\end{equation*}
	If the term within $|\dots|$ is nonnegative, the above is equivalent to $j_1-i_1'\le h$, which is true since $s_1\rhd s_1'$. Otherwise, the above is equivalent to $i_2'\le j_1'$, which holds since $s_2'\rhd s_1'$.
	
	Since $\bs s\in\lad$, \eqref{almostdual} implies $i_l>j_1'$ and $j_l >i_1'+h$, which gives $\spar(s_l)-\spar({}^*\!s_1')>0$ and, hence, $\bs\omega_{s_l}\in \cal P_>$.  
	
	For the segments $s_r\cap s_{t}' = [i_r,j'_t], t\in\{r,r+1\}$, suppose first that $\supp(s_r\cap s_t')-\supp({}^*s_1')< 0$. In that case,
	\begin{equation*}
		\spar(s_r\cap s_{t}')- \spar (^*\!s_1') - \supp(s_r\cap s_{t}')+\supp(^*\!s_1')= 2i_r - 2j_1'>0,
	\end{equation*}
	with the last inequality following from the first part of \eqref{almostdual} together with $\bs s\in\lad$. Otherwise, if $\supp(s_r\cap s_{t}')-\supp(^*\!s_1')\geq 0$, we have
	\begin{equation*}
		\spar(s_r\cap s_{t}')- \spar (^*\!s_1') + \supp(s_r\cap s_{t}')-\supp(^*\!s_1') = 2j_{t}' - 2h -2i_1'.
	\end{equation*}
	Using again that $i_r\geq i_2>j_1'$,  
	\begin{equation*}
		j_{t}'-h -i_1' >  j_{t}'-i_r -(h+i_1'-j_1') = \supp(s_r\cap s_{t}') - \supp({}^*s_1'),
	\end{equation*}
	which is nonnegative by assumption. Thus, $\bs\omega_{s_r\cap s_{t}'}\in \cal P_>$ as desired. 
	
	For the segments $s_r\cup s_{t}' = [i'_t,j_r], t\in\{r,r+1\}$ suppose first that $\supp(s_r\cup s_t')-\supp({}^*s_1')<0$ or, equivalently, 
	\begin{equation*}
		i_t'-j_1'>j_r-h-i_1'>0,
	\end{equation*}
	where we have used the second part of \eqref{almostdual}, together with $\bs s\in\lad$, for the second inequality. Then
	\begin{equation*}
		\spar(s_r\cup s_{t}')- \spar (^*\!s_1') + \supp(^*\!s_1') - \supp(s_r\cup s_{t}')= 2i_{t}'-2j_1'> 0.
	\end{equation*}
	On the other hand, if $\supp(s_r\cup s_t')-\supp({}^*s_1')\geq 0$, then 
	\begin{equation*}
		\spar(s_r\cup s_{t}')- \spar (^*\!s_1') + \supp(s_r\cup s_{t}')-\supp(^*\!s_1') = 2j_r - 2h -2i_1'> 0,
	\end{equation*}
	with the last inequality following from \eqref{almostdual} as before. 
\end{proof}

\subsection{Proof of \eqref{e:novarpivarpiW}}\label{s:novarpivarpiW}
Throughout the remainder of this section we assume that \eqref{almostdual} holds.

	Write $\bs\pi_1\bs\pi_0=\bs\omega_1\bs\omega_2$ with $\bs\omega_j$ as in \cref{ss:crfact}. Suppose \eqref{e:novarpivarpiW} does not hold, i.e.,
	$$\bs\pi_1({}^*\bs \pi_1)^{-1}\in \wt_\ell W(\bs\pi_1\bs\pi_0) = \wt_\ell W(\bs\omega_1)\wt_\ell W(\bs\omega_2),$$
	and let $\bs\varpi_j \in \wt_\ell W(\bs\omega_j)$, $j=1,2$, $\bs\mu_1\in \mathcal P_\le$, and $\bs\mu_2\in\mathcal P_>$ be such that 
	\begin{equation*}
		\bs\pi_1({}^*\bs \pi_1)^{-1}= \bs\varpi_1\bs\varpi_2 = \bs\mu_1\bs\mu_2.
	\end{equation*}
	
	Begin by noting that 
	\begin{equation}\label{e:pi1*>}
		\bs\omega_{^*\!s_1'}({}^*\bs \pi_1)^{-1}\in\mathcal P_>.
	\end{equation}
	Indeed, let $s$ be a segment among the ones in $s_l, s_r\diamond s'_{r+1}, r\ge 1$. In particular, since $\bs s,\bs s'\in\lad$ and $\bs s\rhd\bs s'$, it follows from \eqref{e:nabig} that
	\begin{equation*}
		\spar(s_1')< \spar(s). 
	\end{equation*}
	If it were $\bs\omega_{^*\!s}^{-1}\in\mathcal P_\le$, \eqref{supPleq} would imply $\spar(^*\!s)\le \spar(^*\!s'_1)$, yielding a contradiction. Therefore, \Cref{Pleqfund} and \eqref{e:pi1*>} implies that 
	\begin{equation}\label{e:p1c3mu1}
		\bs\mu_1 = \bs\omega_{s_1'}\bs\omega_{s_1\diamond s_2'}\bs\omega_{{}^*\!s_1'}^{-1}.
	\end{equation}
	Since $\bs\omega_2\in\mathcal P_>$, it follows from \eqref{e:subgpweight} that $\wt_\ell W(\bs\omega_2)\subseteq\mathcal P_>$. Thus, 
$\bs\mu_1$ is part of a reduced expression for $\bs\varpi_1$. In other words, $\bs\varpi_1 = \bs\mu_1\bs\mu$ for some $\bs\mu\in\mathcal P_>$. On the other hand, since $\spar(s_1')<\min\{\spar(s_1\diamond s_1'), \spar(s_1\diamond s_2')\}$,  Proposition \ref{p:weylper} implies that
\begin{equation*}
	W(\bs\omega_1)\cong W(\bs\omega_{s_1\diamond s_1'}\bs\omega_{s_1\diamond s_2'})\otimes W(\bs\omega_{s_1'})
\end{equation*}
and then $\bs\omega_{s_1\diamond s_2'}\bs\omega_{{}^*\!s_1'}^{-1}$ occurs in a reduced expression of $\bs\varpi_1\bs\omega_{s_1'}^{-1}\in\wt_\ell W(\bs\omega_{s_1\diamond s_1'}\bs\omega_{s_1\diamond s_2'})$. Therefore,
\begin{equation*}
	\bs\omega_{s_1\diamond s_2'}\bs\omega_{^*\!s_1'}^{-1}\bs\mu \in \wt_\ell W(\bs\omega_{s_1\diamond s_1'}\bs\omega_{s_1\diamond s_2'})
\end{equation*}
and, hence, there exists $(g_1,g_2,g_3,g_4)\in \mathbb P_{s_1\cap s_1'}\times \mathbb P_{s_1\cup s_1'}\times \mathbb P_{s_1\cap s_2'}\times \mathbb P_{s_1\cup s_2'}$ such that 
\begin{equation}\label{e:thegs}
	\bs\omega(g_1)\bs\omega(g_2)\bs\omega(g_3)\bs\omega(g_4)=\bs\omega_{s_1\cap s_2'}\bs\omega_{s_1\cup s_2'}\bs\omega_{{}^*\!s_1'}^{-1}\bs\mu.
\end{equation}
In particular, $s_1\cap s_2'\in\boc_{g_k}^+$ for some $1\le k\le 4$ and similarly for $s_1\cup s_2'$. In fact, 
\begin{gather*}
	s_1\cap s_2'\in \boc_{g_k}^+ \Rightarrow k\in \{1,3\} \quad \text{while}\quad  
	s_1\cup s_2'\in \boc_{g_k}^+ \Rightarrow k\in \{2,4\}.
\end{gather*}
Indeed, if $s_1\cap s_2'\in \boc_{g_k}^+$, \eqref{uplowcorner} implies $1+(s_1\cap s_2')\in \boc_{g_k}^-$. But then, if $k\in\{2,4\}$, it would follow from \cref{ellwttau} that $1+(s_1\cap s_2')\rhd (s_1\cup s_{k/2}')$, which implies $j_1<j_2'+1$, yielding a contradiction with $s_1\rhd s_2'$. Similarly, if $s_1\cup s_2'\in \boc_{g_k}^+$, \eqref{uplowcorner} implies $1+(s_1\cup s_2')\in \boc_{g_k}^-$. But then, if $k\in\{1,3\}$, it would follow from \cref{ellwttau} that $1+(s_1\cup s_2')\rhd (s_1\cap s_{(k+1)/2}')$, which implies $i_1<i_2'+1$, yielding a contradiction with $s_1\rhd s_2'$ again.

We now perform a case-by-case analysis of the possible pairs $(k,k')$ such that $s_1\cap s_2'\in \boc^+_{g_k}$ and $s_1\cup s_2'\in\boc^+_{g_{k'}}$.  The proof is completed by showing that \eqref{e:thegs} yields a contradiction in all cases. The fact that $\bs\mu\in\mathcal P_>$ is a key point for reaching the contradiction.

\subsubsection{\bf{Case $(3,4)$}} Suppose $s_1\cap s_2'\in \boc_{g_3}^+$ and $s_1\cup s_2'\in \boc_{g_4}^+$, which implies $\bs\omega(g_3)= \bs\omega_{s_1\cap s_2'}$, $\bs\omega(g_4)=\bs\omega_{s_1\cup s_2'}$, and, therefore, \eqref{e:thegs} becomes
\begin{equation}\label{case34}
	\bs\omega(g_1)\bs\omega(g_2)= \bs\omega_{{}^*\!s_1'}^{-1}\bs\mu.
\end{equation}
In particular, ${}^*\!s_1'\in \boc_{g_1}^-\cup\boc_{g_2}^-$ and, hence,
\begin{equation*}
	\{s\in \boc_{g_1}^-\cup\boc_{g_2}^-: \bs\omega_s\in\mathcal P_\le\}\ne\emptyset. 
\end{equation*}
Choose $s$ in this set such that $\spar(s)$ is minimal and let $k\in\{1,2\}$ be such that $s\in \boc_{g_k}^-$. Note
\begin{equation*}
	\spar(s)\le \spar({}^*\!s_1') < \min \spar(^*\!(s_1\diamond s_1')),
\end{equation*}
where the first inequality follows from \eqref{e:novarpivarpiW<} since $\bs\omega_s\in\mathcal P_\le$. In particular, $g_k$ is not the lowest path and, hence, by Lemma \ref{lowtoup} there exists $s'\in \boc_{g_k}^+$ such that 
\begin{equation*}
	\spar(s')= \spar(s)-|\supp(s)-\supp(s')|<\spar(s).
\end{equation*} 
Since $\bs\omega_s\in\mathcal P_\le$, this implies $\bs\omega_{s'}\in \cal P_{\leq}$. The minimality of $\spar(s)$ implies  $s'\notin \boc_{g_{k'}}^-$, where $\{k,k'\}=\{1,2\}$. This shows $\bs\omega_{s'}\in \cal P_{\le}$ and occurs in a reduced expression of $\bs\omega(g_1)\bs\omega(g_2)$, yielding a contradiction with \eqref{case34} since $\bs\mu\in\mathcal P_>$. 

\subsubsection{\bf{Case $(3,2)$}} Suppose $s_1\cap s_2'\in \boc_{g_3}^+$ and $s_1\cup s_2'\in \boc_{g_2}^+\setminus \boc_{g_4}^+$, which implies $\bs\omega(g_3)= \bs\omega_{s_1\cap s_2'}$ and $\bs\omega_{s_1\cup s_2'}$ occurs in a reduced expression for $\bs\omega(g_2)$. Letting $\bs\omega'(g_2) = \bs\omega(g_2)\bs\omega_{s_1\cup s_2'}^{-1}$, \eqref{e:thegs} becomes
\begin{equation}\label{case32p}
	\bs\omega(g_1)\bs\omega'(g_2)\bs\omega(g_4)= \bs\omega_{{}^*\!s_1'}^{-1}\bs\mu.
\end{equation}

Let us check
\begin{equation}\label{e:*s1'nig2}
	^*\!s_1'\in\boc^-_{g_1}\cup \boc^-_{g_4}.
\end{equation}
Indeed, since we are assuming that $s_1\cup s_2'\in \boc_{g_2}^+$, we have $s_1\cup s_1'\in \tilde\seg$ and hence 
\begin{equation}\label{e:case32j1<h+i1'}
	j_1-i_1'<h.
\end{equation} 
In particular we have $i_2'\le j_1'< j_1<i_1'+h$, i.e., ${}^*\!s_1'\rhd (s_1\cup s_2')$. It then follows from \eqref{e:nocpinpath}  that $^*\!s_1'\notin\boc^-_{g_2}$, thus proving \eqref{e:*s1'nig2}. We now analyze each possibility in \eqref{e:*s1'nig2} separately. 

Suppose first $^*\!s_1'\in\boc_{g_1}^-$ which implies $\bs\omega_{^*\!s_1'}^{-1}$ occurs in a reduced expression for $\bs\omega(g_1)$. Letting $\bs\omega'(g_1) = \bs\omega(g_1)\bs\omega_{^*\!s_1'}$, \eqref{case32p} becomes
\begin{equation}\label{case321}
	\bs\omega'(g_1)\bs\omega'(g_2)\bs\omega(g_4)= \bs\mu.
\end{equation}
Since $^*\!s_1'\neq {}^*(s_1\cap s_1')$, \eqref{uniquecorner} implies that $\boc_{g_1}^+\neq \emptyset$. An application of  \cref{lowtoup} with $s_1\cap s_1', g_1$, and $^*\!s_1'=[j_1',i_1'+h]$ in place of $s,p$, and $[m,n]$, respectively, implies 
\begin{equation}\label{e:case32s}
	\exists\ i_1'<m<j_1' \ \ \text{such that}\ \ s:=[m,h+i_1']\in \boc_{g_1}^+.
\end{equation}
One easily checks equality holds in \eqref{e:novarpivarpiW<} for this $s$ and, hence,  $\bomega_s\in\cal P_{\leq}$. In light of  \eqref{case321} and the fact that $\bs\mu\in\mathcal P_>$, we must have $s\in \boc_{g_2}^-\cup \boc_{g_4}^-$. In fact, let us show  $s\in \boc_{g_4}^-$. Suppose for a contradiction that $s\in \boc_{g_2}^-$. Since we also have $s_1\cup s_2'\in \boc_{g_2}^+$, it follows that $s_1\cup s_2'$ and $s$ cannot be connected. This, alongside \eqref{e:case32j1<h+i1'} and the fact that $m<j_1$, implies $i_1'<m\leq i_2'$ and, hence,
\begin{equation*}
	\spar(s) + \supp(s)^* = h+2m\leq h +2i_2'< h +2i_1 = g_{1}(h),
\end{equation*}
where the last equality follows by definition of $g_1\in\mathbb P_{s_1\cap s_1'}$. This contradicts \eqref{boundp''}, thus showing $s\in \boc_{g_4}^-$ as claimed. 

It then follows from \cref{ellwttau} that $s\rhd (s_1\cup s_2')$, or equivalently, 
\begin{equation}\label{e:case32s4-}
	i_2'<m<j_1<h+i_1'.
\end{equation}
An application of the first part of  \cref{lowtoup} with $s_1\cup s_2', g_4$, and the above $s$ in place of $s,p$, and $[m,n]$, respectively, implies 
\begin{equation}\label{e:case32s'}
	\exists\ m<n<h+i_1' \ \ \text{such that}\ \ s':=[m,n]\in \boc_{g_4}^+.
\end{equation}
Let us check that $\bs\omega_{s'}\in \cal P_{\leq}$. On one hand, the second inequalities in \eqref{e:case32s} and \eqref{e:case32s'} imply
\begin{equation*}
	\spar(^*\!s_1')-\spar(s') = h+i_1'+j_1' -m-n   > \max\{ j_1'-m, h+i_1'-n\}.  
\end{equation*}
Thus, in order to check $s'$ satisfies \eqref{e:novarpivarpiW<}, it suffices to show
\begin{equation*}
	|\supp(s')-\supp({}^*\!s_1')| \le\max\{ j_1'-m, h+i_1'-n\}.
\end{equation*}
But $\supp(s')-\supp({}^*\!s_1') = (n-h-i_1') + (j_1'-m)$, and the above easily follows by using the second inequality in \eqref{e:case32s'} if this is nonnegative or the second inequality in \eqref{e:case32s} otherwise.
Writing $\bs\omega(g_4) = \bs\omega'(g_4)\bs\omega_{s'}$,   \eqref{case321} becomes
\begin{equation}\label{case32s'}
	\bs\omega'(g_1)\bs\omega'(g_2)\bs\omega'(g_4)\bs\omega_{s'}= \bs\mu.
\end{equation}
In order to reach a contradiction, it now suffices to show 
\begin{equation}\label{case32s'fail}
	s'\notin \boc_{g_1}^-\cup\boc_{g_2}^-.
\end{equation}

If $s'\in \boc_{g_1}^-$, \cref{ellwttau} and \eqref{e:case32s} imply $i_1<m< j_1'<n$. But then, \eqref{e:case32s'} implies $m<j_1'<n <i_1'+h$. In particular, ${}^*\!s_1\rhd s'$, yielding a contradiction with  \eqref{e:nocpinpath} since $s', {}^*\! s_1'\in \boc_{g_1}^-$. If $s'\in \boc_{g_2}^-$, \cref{ellwttau}  implies $i_1'<m\leq j_1<n$. On the other hand, since $i_2'<m$ by \eqref{e:case32s4-}, it follows that  $s'\rhd (s_1\cup s_2')$ and, since $s_1\cup s_2'\in \boc_{g_2}^+$ by the initial assumption of this (3,2) case, we reach a contradiction with \eqref{e:nocpinpath} again. This completes the proof in the case ${}^*\!s_1'\in \boc_{g_1}^-$.

We are left to deal with the case ${}^*\!s_1'\in \boc_{g_4}^-$. Letting $\bs\omega'(g_4) = \bs\omega(g_4)\bs\omega_{^*\!s_1'}$, this time \eqref{case32p} becomes
\begin{equation}\label{case324}
	\bs\omega(g_1)\bs\omega'(g_2)\bs\omega'(g_4)= \bs\mu.
\end{equation}
Since $^*\!s_1'\neq {}^*(s_1\cup s_2')$, \eqref{uniquecorner} implies that $\boc_{g_4}^+\neq \emptyset$. An application  of  \cref{lowtoup} with $s_1\cup s_2', g_4$, and $^*\!s_1'$ in place of $s,p$, and $[m,n]$, respectively, implies 
\begin{equation}\label{e:case32s}
	\exists\ j_1'<n<h+i_1' \ \ \text{such that}\ \ s:=[j_1',n]\in \boc_{g_4}^+.
\end{equation}
As in the previous case,  $\bomega_s\in\cal P_{\leq}$ which, this time, implies $s\in \boc_{g_2}^-\cup \boc_{g_1}^-$, and we show $s\in \boc_{g_1}^-$. If this were not the case, exactly as before, we conclude $s_1\cup s_2'$ and $s$ cannot be connected and, hence, $i_2'<j_1'<n\leq j_1$. On the other hand, since $s\in \boc_{g_2}^-$, \cref{ellwttau} implies $s\rhd (s_1\cup s_1')$, i.e., $i_1'<j_1'\leq j_1 <n$, yielding a contradiction. 

Then $s\in \boc^-_{g_1}$ and, since $n<h+i_1'<h+i_1$ use can \cref{lowtoup} with $s_1\cap s_1'$, $g_1$ and $s$ in the place of $s, p$ and $[m,n]$, respectively, to conclude 
\begin{equation*}
	\exists\ n-h<m<j_1' \ \ \text{such that}\ \ s':=[m,n]\in \boc_{g_1}^+.
\end{equation*}
One easily checks that $\bs\omega_{s'}\in\cal P_{\leq}$ and, hence, writing $\bs\omega(g_1)=\bs\omega_{s'}\bs\omega'(g_1)$, \eqref{case324} becomes
\begin{equation*}
	\bs\omega'(g_1)\bs\omega'(g_2)\bs\omega'(g_4)\bs\omega_{s'}= \bs\mu.
\end{equation*}
As in the previous case, in order to reach a contradiction, it suffices to show 
$$s'\notin \boc_{g_4}^-\cup\boc_{g_2}^-.$$
Recall that, by definition, $m<j_1'<n<h+i_1'$ and, therefore, ${}^*\!s_1'\rhd s'$. Since ${}^*\!s_1'\in\boc_{g_4}^-$ by assumption,  \eqref{e:nocpinpath} implies $s'\notin \boc_{g_4}^-$.  
Similarly, if it were $s'\in\boc_{g_2}^-$ we would have $i_1'<m\leq j_1<n$. Moreover, since $s'\in \boc_{g_1}^+$, \eqref{uplowcorner} and  \cref{ellwttau} further imply that $i_1\leq m$. In particular,  $i_2'<i_1\leq m\leq j_1<n$, thus showing $s'\rhd (s_1\cup s_2')$, which contradicts \eqref{e:nocpinpath} since $s',(s_2'\cup s_1)\in \boc_{g_2}^+\cup\boc_{g_2}^-$. This completes the proof of case (3,2).

\subsubsection{{\bf Case $(1,4)$}}
One easily checks that the pair $({}^*\!s,{}^*\!s')$ satisfies the same assumptions satisfied by the pair $(\bs s,\bs s')$, including \eqref{almostdual}. In light of the two last parts of \eqref{e:concacudual}, it is clear that case (1,4) for  $(\bs s,\bs s')$ is equivalent to case (3,2) for $({}^*\!s,{}^*\!s')$. Thus, there is nothing to do.

\subsubsection{{\bf Case $(1,2)$}} Suppose $s_1\cap s_2'\in \boc_{g_1}^+$ and $s_1\cup s_2'\in \boc_{g_2}^+$. Letting $\bomega'(g_1)= \bomega(g_1)\bomega_{s_1\cap s_2'}^{-1}$ and  $\bomega'(g_2)= \bomega(g_2)\bomega_{s_1\cup s_2'}^{-1}$, \eqref{e:thegs} becomes
\begin{equation}\label{case12p}
	\bs\omega'(g_1)\bs\omega'(g_2)\bs\omega(g_3)\bs\omega(g_4)= \bs\omega_{{}^*\!s_1'}^{-1}\bs\mu.
\end{equation}
We shall need the following observation arising from the assumptions of this case:
\begin{equation}\label{gcorn12}
	E_{g_1} \subset [1,\supp(s_1\cap s_2')] \ \ {\rm and} \ \ E_{g_2} \subset [\supp(s_1\cup s_2'),N].
\end{equation}
Indeed, 
$$g_1(\supp(s_1\cap s_2'))+(\supp(s_1\cap s_2'))^*=\spar(s_1\cap s_2')+ h-\supp(s_1\cap s_2')=  h+2i_1 =g_1(h),$$
and $g_2(\supp(s_1\cup s_2')) + \supp(s_1\cup s_2') =\spar(s_1\cup s_2') + \supp(s_1\cup s_2')=  2j_1 = g_2(0)$
and we are done by \eqref{e:bottonline}. 

Note also that ${}^*\!s_1'\rhd (s_1\cap s_2')$ and ${}^*\!s_1'\rhd (s_1\cup s_2')$ (cf. \eqref{e:case32j1<h+i1'}). Then, an application of \eqref{e:nocpinpath} gives  
\begin{equation*}
	{}^*\!s_1'\in \boc_{g_3}^-\cup \boc_{g_4}^-.
\end{equation*}
 The case ${}^*\!s_1'\in \boc_{g_4}^-$ follows from the case ${}^*\!s_1'\in \boc_{g_4}^-$ by duality in the same spirit of the argument used for Case (1,4).
Thus, henceforth, we assume ${}^*\!s_1'\in \boc_{g_3}^-$ and write $\bs\omega(g_3) = \bs\omega'(g_3)\bs\omega_{^*\!s_1'}^{-1}$ so that \eqref{case12p} becomes
\begin{equation}\label{case123}
	\bs\omega'(g_1)\bs\omega'(g_2)\bs\omega'(g_3)\bs\omega(g_4)= \bs\mu.
\end{equation}

 An application of the second part of Lemma \ref{lowtoup} with $s_1\cap s_2', g_3$ and ${}^*\!s_1'$ in place of $s,p$, and $[m,n]$, respectively, implies 
$$\exists\ i_1'< m<j_1' \ \  {\rm such\ that } \ s:= [m,h+i_1']\in \boc_{g_3}^+$$
and one easily checks equality holds in \eqref{e:novarpivarpiW<}. Thus, $\bs\omega_s\in \cal P_{\leq}$ and, hence, \eqref{case123} implies 
$s\in \boc_{g_1}^-\cup \boc_{g_2}^-\cup\boc_{g_4}^-.$
In fact, we show that $$s\in \boc_{g_4}^-.$$ If it were $s\in \boc_{g_1}^-$, we claim $m\leq i_1$. Indeed, if it were $i_1<m$, since  $m<j_1'<j_2'<h+i_1'$, it would follow that $s\rhd(s_1\cap s_2')$. As $s_1\cap s_2'\in \boc_{g_1}^+$, this contradicts \eqref{e:nocpinpath}. The inequality $m\le i_1$, together with the fact that $j'_2\le h+i'_1$ coming from $s_2'\rhd s_1'$, implies 
\begin{equation}\label{12g1case}
	\supp(s)= h+i_1'- m >j_2'-i_1 = \supp(s_1\cap s_2'),
\end{equation}
which contradicts \eqref{gcorn12}.  
Similarly, if $s\in \boc_{g_2}^-$, \eqref{e:nocpinpath} alongside with the fact that $(s_1\cup s_2')\in \boc_{g_2}^+$ and $m<j_1'<j_1<h+i_1'$ forces $m\leq i_2'$. In particular, 
\begin{equation}\label{12g2case}
	g_3(\supp(s))+ \supp(s)^* = h+2m < h+2i_1 = g_3(h),
\end{equation}
which contradicts \eqref{boundp''} and shows that $s\in \boc_{g_4}^-$. 

Another application of the first part of Lemma \ref{lowtoup} with $s_1\cup s_2'$, $g_4$, and $s$ in place of $s$, $p$, and $[m,n]$, respectively, implies 
$$\exists \ m<n<h+i_1' \ \ {\rm such \ that} \ \ s':= [m,n]\in \boc_{g_4}^+.$$
Again, it is straightforward to check that $\bomega_{s'}\in \cal P_{\leq}$ and \eqref{case123} implies $s'\in \boc_{g_1}^-\cup\boc_{g_2}^- \cup \boc_{g_3}^-$. 
Let us check  
\begin{equation}\label{12s'notin}
	s'\in \boc_{g_2}^-.
\end{equation}
Recall from our assumptions that $i_1'<m<n<h +i_1'$ and $m<j_1'$. In particular, ${}^*\!s_1'\rhd s'$ and, since we assuming ${}^*\!s_1' \in\boc_{g_3}^-$, \eqref{e:nocpinpath} implies $s'\notin \boc_{g_3}^-$. If it were $s'\in \boc_{g_1}^-$, then $s'\rhd (s_1\cap s_1')$, i.e., $i_1<m\le j_1'<n$. Then, if we had $j'_2<n$, it would follow that $s'\rhd (s_1\cap s'_2)$, yielding a contradiction with \eqref{e:nocpinpath} since $s_1\cap s'_2\in g_1^+$ by assumption. Thus, $i_1<m\leq j_1'<n\leq j_2'$ and, hence, 
\begin{equation*}
	g_4(\supp(s'))+ \supp(s') = 2n \leq 2j_2'<2j_1 = g_4(0),
\end{equation*}
contradicting \eqref{boundp'''} and completing the proof of \eqref{12s'notin}. 

It follows from \eqref{12s'notin} that $s'\rhd(s_1\cup s_1')$, i.e., $i'_1< m\le j_1<n$. This in turn implies $m\le i'_2$ since, otherwise, we would have $s'\rhd (s_1\cup s'_2)$, yielding a contradiction with \eqref{e:nocpinpath} since $s'\in\boc_{g_2}^-$ and $s_1\cup s'_2\in g_2^+$ by assumption. On the other hand, since $s'\in \boc_{g_4}^+$, \eqref{uplowcorner}, together with \cref{ellwttau}, implies $1+s'\rhd (s_1\cup s'_2)$, i.e., $i_2'\leq m<j_1\leq n$. Thus, 
\begin{equation*}
	m=i_2'<j_1<n<h+i_1'
\end{equation*}
and an application of the second part of \cref{lowtoup} with $s_1\cup s_1'$ and $g_2$ in the place of $s$ and $p$, respectively, implies 
\begin{equation*}
	\exists \ n-h<m'<i_2' \ \ {\rm such\ that} \ \ s'':=[m', n]\in \boc_{g_2}^+.
\end{equation*}
Again, one easily checks $\bs\omega_{s''}\in \cal P_{\leq}$ and concludes using \eqref{case123} that we must have $s''\in \boc_{g_1}^-\cup \boc_{g_3}^-\cup \boc_{g_4}^-$. \cref{ellwttau} would then imply $s''\rhd s'''$ with $s'''\in \{s_1\cap s_1',s_1\cap s_2',s_1\cup s_2'\}$. All options imply $i_2'<m'$, yielding a contradiction as desired. 

\subsection{Length-2 Snakes}\label{ss:l=2}
In this section, we prove \cref{t:imagfromsnakegen} dropping the assumption \eqref{almostdual'}, but assuming $\bs s,\bs s'\in\lad^2$. The proof will use \cref{t:imagfromsnakegen} with the assumption \eqref{almostdual'} together with usual arguments arising from diagram subalgebras. Thus, let us start by recalling the latter. 

For a subdiagram $J\subseteq I$,  let $\lie g_J$ be the subalgebra of $\lie g$ generated by the corresponding simple root vectors. Let also $U_q(\lie a_J)$, with $\lie a=\lie g, \tlie g,\tlie h$, etc., be the respective quantum groups associated to $\lie a_J$. Let also $U_q(\lie a)_J$ be the the subalgebra of $U_q(\tlie g)$ generated by the generators corresponding to $J$. It is well known that there is an algebra isomorphism
\begin{equation*}
	U_q(\lie a)_J\cong U_q(\lie a_J).
\end{equation*}
This is a Hopf algebra isomorphism only if $\lie a\subseteq\lie g$. We shall always implicitly identify $U_q(\lie a)_J$ with $U_q(\lie a_J)$ without further notice.

For $\bs\varpi\in\mathcal P$ and $J\subseteq I$, let $\bs\varpi_J$ be the associated
$J$-tuple of rational functions and let $\cal
P_J=\{\bs\varpi_J:\bs\varpi\in \mathcal P\}$. Similarly define $\cal
P_J^+$. Notice that $\bs\varpi_J$ can be regarded as an element of
the $\ell$-weight lattice of $U_q(\tlie{sl}_{h_J})\cong U_q(\tlie g)_J$, where $h_J=|J|+1$. 
If $V$ is a highest-$\ell$-weight module with highest-$\ell$-weight vector $v$ and $J\subset I$, we let $V_J$ denote the $U_q(\tlie g)_J$-submodule of $V$ generated by $v$. Evidently, if $\bs\pi$ is the highest-$\ell$-weight of $V$, then $V_J$ is highest-$\ell$-weight with highest $\ell$-weight $\bs\pi_J$.

Suppose $J=[a,b]$ with $1\le a\le b\le N$ and let $N_J = b-a+1=|J|$. Note that if $c\in J$, when regarded as a node for the Dynkin diagram of 
\begin{equation*}
	U_q(\tlie{sl}_{h_J})\cong U_q(\tlie{g}_J)\cong U_q(\tlie g)_J,
\end{equation*}
$c$ is the $(c-a+1)$-th node. Suppose also that $s=[i,j]\in\seg$ is such that 
\begin{equation*}
	a\le j-i=\supp(s)\le b.
\end{equation*} 
This means $(\bs\omega_s)_J\ne\bs 1$ and, hence, when regarded as an $\ell$-weight for $U_q(\tlie{sl}_{h_J})$, $(\bs\omega_s)_J$ corresponds to
\begin{equation*}
	 \bs\varpi_{\supp(s)-a+1,\spar(s)},
\end{equation*} 
which, using \eqref{e:fundtoint}, can be represented by $\bs\omega_{\tilde s}$ where $\tilde s$ is the $N_J$-segment 
\begin{equation}\label{e:NJseg}
	\tilde s = [i+\frac{a-1}{2},j-\frac{a-1}{2}].
\end{equation}
Moreover, if $s'\in\seg$ also satisfies $a\le \supp(s')\le b$ and $s\rhd s'$, one easily checks that
\begin{equation}
	(\bs\omega_{s\diamond s'})_J = \bs\omega_{\tilde s\diamond\tilde s'}.
\end{equation}

Since $U_q(\tlie g)_J$ is not a sub-coalgebra of $U_q(\tlie g)$, if $M$ and $N$ are $U_q(\tlie g)_J$-submodules of $U_q(\tlie g)$-modules $V$ and $W$, respectively, it is in general not true that $M\otimes N$ is a $U_q(\tlie g)_J$-submodule of $V\otimes W$.  Denote by $M\otimes_J N$ the $U_q(\tlie g)_J$-module obtained by using the coalgebra structure from $ U_q(\tlie g_J)$. By \cite[Proposition 2.2]{CP96}, if $V$ and $W$ are highest-$\ell$-weight modules, then $V_J\otimes W_J$ is a $U_q(\tlie g)_J$-submodule of $V\otimes W$ isomorphic to $V_J\otimes_J W_J$ via the identity map. We will need the following consequence of this proposition (see \cite[Corollary 3.2.4]{MS24}).

\begin{lem}\label{subdsimple}
	Let $J\subset I$ be a connected subdiagram. If $V\otimes W$ is a simplie $U_q(\tilde{\lie g})$-module, then $V_J\otimes W_J$ is a simple $U_q(\tilde{\lie g})_J$-module. 
\end{lem}

We are ready to prove \cref{t:imagfromsnakegen} without the assumption \eqref{almostdual'} in the case $l=2$.
Consider the subdiagram 
\begin{equation}\label{e:l=2subd}
	J = [a,b] \ \ \text{with}\ \ a=\min\supp (\bs s\cap \bs s')+1 \ \ \text{and}\ \ b=\max\supp(\bs s\cup \bs s')-1.
\end{equation} 
It easily follows that
\begin{equation*}
	a\le \supp(s_k)\le b \ \ \text{for}\ \ k\in\{1,2\},
\end{equation*}
and similarly for $\bs s', \bs s\cap_1\bs s'$, and $\bs s\cup_1\bs s'$. Thus, the fundamental $\ell$-weights appearing in $\bs\pi_1$ arising from the corresponding $N$-segments remain non-trivial when regarded as $\ell$-weights for $U_q(\tlie{sl}_{h_J})$ and, hence, the associated fundamental $\ell$-weight for $U_q(\tlie{sl}_{h_J})$ can be represented by an $N_J$-segment according to \eqref{e:NJseg}. We denote by $\tilde s_k=[\tilde i_k,\tilde j_k]$ the $N_J$-segment corresponding to $s_k$ and similarly define $\tilde s_k'$.  

We will check that $\tilde{\bs s}$ and $\tilde{\bs s}'$ satisfy \eqref{almostdual} (with $h$ replaced by $h_J$). It then follows from \cref{t:imagfromsnakegen} applied to the algebra $U_q(\tlie g_J)$ that $V(\bs\pi_1)_J $ is imaginary,  i.e.,
\begin{equation*}
	V(\bs\pi_1)_J  \otimes_J V(\bs\pi_1)_J \ \ \text{is reducible}.   
\end{equation*} 
An application of \cref{subdsimple} then implies the same is true for $V(\bs\pi_1)  \otimes V(\bs\pi_1)$, thus completing the proof.

We need to check
\begin{equation}
	\tilde j_1'<\tilde i_2 \ \ \text{and}\ \ \tilde j_2-\tilde i_1'>h_J.
\end{equation}
Using \eqref{e:NJseg} and that $h_J=b-a+2$, this is equivalent to
\begin{equation}
	j_1'-i_2 < \min\supp (\bs s\cap \bs s') \ \ \text{and}\ \ j_2-i_1' > \max\supp(\bs s\cup \bs s'). 
\end{equation}
Since $\bs s\rhd_1\bs s'$, we have
\begin{equation*}
	\supp (\bs s\cap \bs s') = \{j_1'-i_1,j_2'-i_2\} \ \ \text{and}\ \ \supp (\bs s\cup \bs s') = \{j_1-i_1', j_2-i_2'\}.
\end{equation*}
On the other hand, since $\bs s,\bs s'\in \mathbb L^2$, we also have 
\begin{gather*}
	j_1'-i_2 < j_1'-i_1, \ \ j_1'-i_2< j_2'-i_2, \ \ j_2-i_1'> j_2-i_2', \ \  j_2-i_1'>j_1 -i_1',
\end{gather*}
which clearly completes the proof.

\subsection{The Case of KR Modules}\label{ss:KRcase}
We now use a similar argument to the one used in \cref{ss:l=2} to  prove that the final conclusion of \cref{t:imagfromsnakegen} holds if the assumption \eqref{almostdual'} is replaced by
\begin{equation*}
	V(\bs\omega_{\bs s}) \ \ \text{and}\ \ V(\bs\omega_{\bs s'}) \ \ \text{are Kirillov-Reshetikhin modules.}
\end{equation*}
Recall that this means
\begin{equation}\label{e:KRcase}
	i_{k+1}=i_k+1 \ \ \text{and}\ \ j_{k+1}=j_k+1 \ \ \text{for all} \ \ 1\leq k<l
\end{equation} 
and similarly for $\bs s'$. In general, \eqref{e:KRcase} does not imply \eqref{almostdual}, so \cref{t:imagfromsnakegen} cannot be used directly to conclude $V(\bs\pi_1)$ is imaginary.

In fact, we will prove something more general. Instead of \eqref{e:KRcase}, we assume
\begin{equation}\label{e:minmaxat1}
	\min\supp (\bs s\cap \bs s')=\supp(s_1\cap s_1') \quad\text{and}\quad \max\supp(\bs s\cup \bs s') = \supp(s_1\cup s_1').
\end{equation}
Since \eqref{e:KRcase} implies $|\supp (\bs s\cap \bs s')| = |\supp (\bs s\cup \bs s')|=1$, it follows that the case of KR modules is a particular case of multisegments satisfying \eqref{e:minmaxat1}. Thus, assume \eqref{e:minmaxat1} holds and let $J$ be as in \eqref{e:l=2subd}. One easily checks \eqref{e:minmaxat1} implies 
\begin{equation*}
	a\le \supp(s_k)\le b \ \ \text{for}\ \ 1\le k\le l
\end{equation*}
and similarly for  $\bs s', \bs s\cap_1\bs s'$, and $\bs s\cup_1\bs s'$. From this point on, a copy of the argument of \cref{ss:l=2} leads to the conclusion that $V(\bs\pi_1)$ is imaginary.

\section{Characterizing and Counting Examples}\label{s:count}

We now characterize and count how many $1$-covering pairs of connected ladders exist (up to uniform shift).

\subsection{General Set Up}For the sake of record keeping, let us start by counting the segments which are covered by a given proper segment. Thus, given $s\in\tilde\seg$, set 
\begin{equation*}
	C(s) = \{s'\in\seg: s\rhd s'\}.
\end{equation*}
Writing $s=[i,j]$ and $s'=[i',j']$, then $s\rhd s'$ if and only if $i'<i\le j'<j\le h+i'$, or, equivalently,
\begin{equation*}
	i\le j'< j \ \ \text{and}\ \ j-h\le i'<i.
\end{equation*}
Thus, there are $j-i$ choices for $j'$ and $h+i-j$ choices for $i'$. In other words,
\begin{equation}\label{e:covercount}
	|C(s)| = \supp(s)\supp(s^*). 
\end{equation}
In the next section, the following subset of $C(s)$ will be relevant:
\begin{equation*}
	C_1(s) = \{s'\in C(s): i\le j'< j-1  \text{ and } j-h\le i'<i-1\},
\end{equation*}
whose cardinality is 
\begin{equation}\label{e:covercount1}
	|C_1(s)| = (\supp(s)-1)(\supp(s^*)-1). 
\end{equation}

\begin{lem}\label{l:coverred}
	Let $\bs s,\bs s'\in\lad^l$ and $0\le p<l$. Then, 
	\begin{equation*}
		\bs s\rhd_p\bs s' \quad\Leftrightarrow\quad \bs s(k,p+1+k)\rhd_p \bs s'(k,p+1+k) \ \ \text{for all}\ \ 0\le k<l-p.
	\end{equation*} 
\end{lem}

\begin{proof}
	The $\Rightarrow$ part is immediate from the definition of $\rhd_p$. Thus, assume the right hand side of $\Leftrightarrow$ holds. Note that, for $p=0$, this is exactly the definition of $\bs s\rhd\bs s'$, so there is nothing to do. Also, if $p=l-1$, then $k=0$ and the assumption reads $\bs s(0,l)\rhd \bs s'(0,l)$, which coincides with the left hand side. Hence, we can assume $p\le l-2$. Let us proceed by induction on $l$, which starts when $l=2$ by the case $p=0$. Thus, assume $l>2$ and that the claim holds for ladders of length $l-1$. 
	
	Let $\tilde{\bs s}=\bs s(0,l-1)\in\lad^{l-1}$ and similarly define $\tilde{\bs s}'$. By assumption,
	\begin{equation*}
		\tilde{\bs s}(k,p+1+k)\rhd_p \tilde{\bs s}'(k,p+1+k) \ \ \text{for all}\ \ 0\le k<l-1-p.
	\end{equation*}
	The inductive assumption then implies $\tilde{\bs s}\rhd_p \tilde{\bs s}'$ or, equivalently
	\begin{equation}
		\bs s(0,l-1)\rhd_p\bs s'(0,l-1).
	\end{equation}
	Similarly, it follows that 
	\begin{equation}
		\bs s(1,l)\rhd_p\bs s'(1,l).
	\end{equation}
	These two facts together clearly imply $\bs s\rhd_p\bs s'$, as desired. 
\end{proof}
This lemma is a step towards reducing the task of characterizing the pairs $(\bs s,\bs s')$ such that $\bs s\rhd_p\bs s'$ to the case $p=l-1$.

\subsection{$1$-Covers of Length $2$}\label{ss:l21c}
Henceforth, let us focus on the case $p=1$. 
Given $s_1,s_1'\in\tilde\seg$ such that $s_1\rhd s_1'$, let
\begin{equation}\label{e:C1l=2}
	C_1(s_1,s_1') = \{(\bs s,\bs s')\in\lad^2\times\lad^2: \bs s(0,1)=s_1, \bs s'(0,1)=s_1', \text{ and } \bs s\rhd_1\bs s' \}.
\end{equation}
We will show
\begin{equation}\label{e:l=21covers}
	\begin{aligned}
		|C_1(s_1,s_1')| = &\ \frac{1}{4}(\supp(s_1'\lrcorner s_1)-1)(\supp(s_1'\llcorner s_1)-1)\times \\
		& \times  (\supp(^*\!s_1)+\supp(^*\!s_1\cap {}^*\!s_1'))(\supp(s_1)+\supp(s_1\cap s_1')).
	\end{aligned}
\end{equation}
Combining this with \eqref{e:covercount1}, it follows that, for each $s_1\in\tilde\seg$, there are
\begin{equation}
	|C_1(s_1)||C_1(s_1,s_1')|
\end{equation}
examples of imaginary modules covered by the conjecture for $k=1$ and $l=2$. On the other hand, if we let $D_2(s_1)$ be the subset of $C_1(s_1,s_1')$ of elements arising from the context of \cite{BC23}, we will see
\begin{equation}\label{e:BCl=2}
	D_2(s_1)=\supp(s_1^*)-1.
\end{equation}
We shall see below that, if $N=3$ and $C_1(s_1,s_1')\ne\emptyset$, then $C_1(s_1,s_1')=D_2(s_1)$ has a single element.

We want to characterize and count the elements in $C_1(s_1,s_1')$. As usual, write $\bs s=(s_1,s_2)$, $\bs s'=(s_1',s_2')$, $s_k = [i_k,j_k]$, and $s'_k=[i'_k,j'_k]$. Thus, we want to characterize the possible choices of $i_2,j_2,i'_2,j'_2$. The requests that $\bs s\in\lad$  and $\bs s \rhd_1\bs s'$ imply $s_2\rhd s_1$, i.e.,
\begin{equation*}
	i_1< i_2\le j_1<j_2\le h-i_1,
\end{equation*}
and similarly for $\bs s'$. On the other hand, the request $\bs s\rhd_1\bs s'$ is equivalent to
\begin{equation*}
	s_2\rhd s_2' \ \ \text{and}\ \ s_1\rhd s_2'.
\end{equation*}
Since the latter implies $i_2'<i_1$ and $j_2'<j_1$, it follows that we must have 
\begin{equation}\label{e:s'sis ladder}
	i_1'<i_2'<i_1<i_2 \ \ \text{and}\ \ j_1'<j_2'<j_1<j_2.
\end{equation}
In particular, $C_1(s_1,s_1')=\emptyset$ unless
\begin{equation}\label{e:sup>1}
	i_1-i_1'\ge 2 \ \ \text{and}\ \ j_1-j_1'\ge 2.
\end{equation}
Note this implies $N\ge 3$. Indeed, since the assumption $s_1\rhd s_1'$ says
\begin{equation*}
	i_1'<i_1\le j_1'<j_1\le h+i_1',
\end{equation*}
it follows that $j_1-i_1'\le h$. On the other hand,
\begin{equation*}
	j_1-i_1'\ge j_1- (i_1-2) \ge j_1-(j_1'-2)\ge 4,
\end{equation*}
so $h\ge 4$, as claimed. Moreover, in the case $h=4$, the above inequalities also imply 
\begin{equation*}
	\supp(s_1)=2, \quad s_1' = s_1^*, \quad s_2 = 1+s_1, \quad\text{and}\quad s_2' = 1+s_1'=s_2^*.  
\end{equation*}
In particular, $|C_1(s_1,s_1')|=1$, $V(\bs\omega_{\bs s})$ is a KR module and $V(\bs\omega_{\bs s'})\cong V(\bs\omega_{\bs s})^*$. 

Henceforth, let us assume $s_1$ and $s_1'$ also satisfy \eqref{e:sup>1}. Choose any $s_2'$ such that
\begin{equation}\label{e:chooses2'}
	i_1'<i_2'<i_1 \ \ \text{and}\ \ j_1'<j_2'<j_1.
\end{equation} 
There are $(i_1-i_1'-1)(j_1-j_1'-1)$ such choices. Recalling that $s_1'\lrcorner s_1 = [i_1',i_1]$ and $s_1\llcorner s_1 = [j_1',j_1]$, this number coincides with
\begin{equation*}
	(\supp(s_1'\lrcorner s_1)-1)(\supp(s_1'\llcorner s_1)-1).
\end{equation*}
Finally, choose any $s_2$ such that
\begin{equation}\label{e:chooses2}
	i_1<i_2\le j_2' \ \ \text{and}\ \ j_1<j_2\le h+i_2'.
\end{equation}
Note the inequalities $i_2\le j_2'$ and $j_2\le h+i_2'$ are imposed by the request $s_2\rhd s_2'$. This completes the characterization of $C_1(s_1,s_1')$. 

It remains to check \eqref{e:l=21covers}. Indeed, there are $j_2'-i_1$ choices for $i_2$ and $h+i_2'-j_1$ choices for $j_2$. Note
\begin{equation*}
	j_2'-i_1 \ge j_2'-j_1'> 0 \ \ \text{and}\ \ h+i_2'-j_1>h+i_1'-j_1\ge 0,
\end{equation*}
and, hence, for each choice of $s_2'$, there is at least one choice of $s_2$. In order to compute $|C_1(s_1,s_1')|$, we need to add over all possible choices of $s_2'$ the number of choices of $s_2$. In other words
\begin{align*}
	|C_1(s_1,s_1')|  = &\sum_{i_2' = i_1'+1}^{i_1-1} \sum_{j_2' = j_1'+1}^{j_1-1}  (j_2'-i_1)(h+i_2'-j_1)\\
		= & \left(\sum_{i_2' = i_1'+1}^{i_1-1} h+i_2'-j_1\right)\left(\sum_{j_2' = j_1'+1}^{j_1-1}  j_2'-i_1\right)\\
		= & \left(\frac{1}{2}(2h+i_1+i_1'-2j_1)(i_1-i_1'-1)\right)
		\left(\frac{1}{2}(j_1+j_1'-2i_1)(j_1-j_1'-1)\right)\\
		= &\ \frac{1}{4}(\supp(s_1'\lrcorner s_1)-1)(\supp(s_1'\llcorner s_1)-1)\times
		\\ & \times (\supp({}^*\!s_1)+\supp({}^*\!s_1\cap {}^*\!s_1'))(\supp(s_1)+\supp(s_1\cap s_1')).
\end{align*}
Let us check \eqref{e:BCl=2}. Since $s_1' = s_1^*$, 
\begin{equation*}
	i_1' = j_1-h \ \ \text{and}\ \ j_1' = i_1. 
\end{equation*}
Plugging this back in \eqref{e:chooses2'}, it means we must choose $s_2'$ such that $j_2' = i_2'+\supp(s_1) = i_2'+j_1-i_1$,
\begin{equation*}
	j_1-h<i_2'<i_1,  \ \ \text{and}\ \ i_1<i_2'+j_1-i_1<j_1.
\end{equation*} 
The assumption $\supp(s_1)\ge h/2$ made in \cite{BC23} implies that any arbitrary choice of $i_2'$ satisfying the first set inequalities above also satisfies the second set.  Since $\supp(s_1^*)> h/2$ if $\supp(s_1)< h/2$, we assume, without loss of generality, that  $\supp(s_1)\ge h/2$. In that case, we have $\supp(s_1^*)-1$ choices for $s_2'$. Since $s_2'$ must coincide with $s_2^*$, there are no further choices to make, completing the proof of \eqref{e:BCl=2}.

\begin{ex}
    Let us list all length-$2$ ladders in the case $N=4$ which give rise to imaginary modules.  We must have $j_1-i_1\in\{2,3\}$, we restrain ourselves to case $j_1-i_1=3\ge h/2$, as the other cases can be obtained by passing to dual modules. There are 5 examples if we fix $i_1$: 
    \begin{enumerate}
		\item $s_1'=[i_1-2,i_1]$,  $s_2' = [i_1-1,i_1+1]$, $s_2=[i_1+1,i_1+4]$;
		\item $s_1'=[i_1-2,i_1]$, $s_2' = [i_1-1,i_1+2]$, $s_2=[i_1+1,i_1+4]$;
		\item $s_1'=[i_1-2,i_1]$, $s_2' = [i_1-1,i_1+2]$, $s_2=[i_1+2,i_1+4]$;
		\item $s_1'=[i_1-2,i_1+1]$, $s_2' =[i_1-1,i_1+2]$, $s_2=[i_1+1,i_1+4]$;
			\item $s_1'=[i_1-2,i_1+1]$, $s_2' =[i_1-1,i_1+2]$, $s_2=[i_1+2,i_1+4]$.
	\end{enumerate}
    Only example (1) arises from \cite{BC23}. \endd
\end{ex}

\subsection{Higher Length $1$-Covers}
Let $l\ge 2$ and suppose $(\bs s,\bs s')\in\lad^{l-1}\times\lad^{l-1}$ satisfy $\bs s\rhd_1\bs s'$ and set
\begin{equation}\label{e:C1genl}
	C_1(\bs s,\bs s') = \{(\tilde{\bs s},\tilde{\bs s}')\in\lad^l\times\lad^l: \tilde{\bs s}(0,l-1)=\bs s, \tilde{\bs s}(0,l-1)=\bs s', \tilde{\bs s}\rhd_1\tilde{\bs s}' \}.
\end{equation}
In the case $l=2$, \eqref{e:C1genl}  coincides with \eqref{e:C1l=2}. \cref{l:coverred} implies we need to characterize the pairs $(s_l,s'_l)\in\lad$ such that 
\begin{equation*}
	(s_{l-1},s_l)\rhd_1 (s_{l-1}',s_l'),
\end{equation*} 
which is exactly $C_1(s_{l-1},s'_{l-1})$. This is what we have done in \cref{ss:l21c} and, hence, 
\begin{equation}
	|C_1(\bs s,\bs s')| = |C_1(s_{l-1},s'_{l-1})|. 
\end{equation}

%\bibliographystyle{aomplain}%{abbrv}
%\bibliography{bibfile}

       \providecommand{\bysame}{\leavevmode\hbox to3em{\hrulefill}\thinspace}
\providecommand{\noopsort}[1]{}
\providecommand{\mr}[1]{\href{http://www.ams.org/mathscinet-getitem?mr=#1}{MR~#1}}
\providecommand{\zbl}[1]{\href{http://www.zentralblatt-math.org/zmath/en/search/?q=an:#1}{Zbl~#1}}
\providecommand{\jfm}[1]{\href{http://www.emis.de/cgi-bin/JFM-item?#1}{JFM~#1}}
\providecommand{\arxiv}[1]{\href{http://www.arxiv.org/abs/#1}{arXiv~#1}}
\providecommand{\doi}[1]{\url{https://doi.org/#1}}
\providecommand{\MR}{\relax\ifhmode\unskip\space\fi MR }
\providecommand{\MRhref}[2]{%
  \href{http://www.ams.org/mathscinet-getitem?mr=#1}{#2}
}
\providecommand{\href}[2]{#2}

\end{document}